\documentclass[11pt, a4paper]{amsart}

\usepackage{amssymb, amsmath, comment}
\usepackage[inline]{enumitem}
\usepackage{hyphenat}
\usepackage[colorlinks,linkcolor=blue,citecolor=blue,urlcolor=black]{hyperref}
\usepackage{xcolor}
\usepackage{mathtools}
\usepackage{tikz-cd}
\usepackage{dsfont}
\usepackage[labelformat=empty]{caption}

\title{An integral Hyodo--Kato isomorphism}
\author{Federico Binda}
\address{Dipartimento di Matematica ``Federigo Enriques'',  Universit\`a degli Studi di Milano, Via Cesare Saldini 50, 20133 Milano, Italy}
\email[F. Binda]{federico.binda@unimi.it}

\author{Shuji Saito}
\address{Graduate School of Mathematical Sciences, the University of Tokyo, 3-8-1 Komaba Meguro-ku
Tokyo 153-8914, Japan}
\email[S. Saito]{sshuji.goo@gmail.com}
\usepackage[left=3cm, right=3cm, top = 2.8cm, bottom = 2.8cm, marginparwidth=2.4cm ]{geometry}

\makeatletter
\def\@tocline#1#2#3#4#5#6#7{\relax
	\ifnum #1>\c@tocdepth 
	\else
	\par \addpenalty\@secpenalty\addvspace{#2}%
	\begingroup \hyphenpenalty\@M
	\@ifempty{#4}{%
		\@tempdima\csname r@tocindent\number#1\endcsname\relax
	}{%
		\@tempdima#4\relax
	}%
	\parindent\z@ \leftskip#3\relax \advance\leftskip\@tempdima\relax
	\rightskip\@pnumwidth plus4em \parfillskip-\@pnumwidth
	#5\leavevmode\hskip-\@tempdima
	\ifcase #1
	\or\or \hskip 2em \or \hskip 2em \else \hskip 3em \fi%
	#6\nobreak\relax
	\dotfill\hbox to\@pnumwidth{\@tocpagenum{#7}}\par
	\nobreak
	\endgroup
	\fi}
\makeatother

\usepackage[all]{xy}
\usepackage{color}

\usepackage[bbgreekl]{mathbbol} 
\usepackage{amsfonts} 
\usepackage{amssymb, amsmath}
\DeclareSymbolFontAlphabet{\mathbb}{AMSb} 
\DeclareSymbolFontAlphabet{\mathbbl}{bbold} 
\newcommand{\Prism}{{\mathbbl{\Delta}}}

\numberwithin{equation}{section}
\newtheorem{thm}[subsection]{Theorem}

\newtheorem{lem}[subsection]{Lemma}
\newtheorem{cor}[subsection]{Corollary}
\newtheorem{prop}[subsection]{Proposition}
\theoremstyle{definition}
\newtheorem{defn}[subsection]{Definition}
\newtheorem{ex}[subsection]{Example}

\theoremstyle{definition}

\newtheorem{question}[subsection]{Question}

\newtheorem{const}[subsection]{Construction}

\newtheorem{rmk}[subsection]{Remark}

\DeclareMathOperator*{\colim}{colim}

\newcommand{\eq}[2]{\begin{equation}\label{#1}#2 \end{equation}}

\newcommand{\red}{{\rm red}}

\newcommand{\Hom}{{\rm Hom}}
\newcommand{\im}{{\rm im}}
\newcommand{\Spec}{{\rm Spec \,}}

\newcommand{\sR}{{\mathcal R}}
\newcommand{\sB}{{\mathcal B}}
\newcommand{\sC}{{\mathcal C}}
\newcommand{\sD}{{\mathcal D}}

\newcommand{\sF}{{\mathcal F}}

\newcommand{\sN}{{\mathcal N}}
\newcommand{\sO}{{\mathcal O}}
\newcommand{\sP}{{\mathcal P}}
\newcommand{\sQ}{{\mathcal Q}}
\newcommand{\sA}{{\mathcal A}}
\newcommand{\sS}{{\mathcal S}}
\newcommand{\sT}{{\mathcal T}}

\newcommand{\sX}{{\mathcal X}}
\newcommand{\sY}{{\mathcal Y}}

\newcommand{\F}{{\mathbb F}}
\newcommand{\G}{{\mathbb G}}

\newcommand{\N}{{\mathbb N}}
\renewcommand{\P}{{\mathbb P}}
\newcommand{\Q}{{\mathbb Q}}

\newcommand{\Z}{{\mathbb Z}}

\def\Sp{\Spec}

\newcommand{\Spf}{\operatorname{Spf}}

\newcommand{\PSh}{\operatorname{PSh}}

\newcommand{\CAlg}{\operatorname{CAlg}}

\newcommand{\Fun}{\operatorname{Fun}}

 \def\str{\mathrm{str}}
\usepackage{tikz-cd}

\newcommand{\PreLog}{\mathrm{PreLog}}

\newcommand{\Ani}{\mathrm{Ani}}
\newcommand{\WSat}{\operatorname{WSat}}
\newcommand{\Sat}{\operatorname{Sat}}

\def\Ker{\mathrm{Ker}}

\def\ch{\mathrm{ch}}
\def\rmapo#1{\overset{#1}{\longrightarrow}}

\def\qaq{\text{ and }}

\def\DA{{\sf{DA}}}
\def\DAlog{{\sf{DA}}^{\log}}
\def\DAlogS{{\sf{DA}}^{\log}_{\sR}}
\def\DAlogSs{{\sf{DA}}^{\log,s}_{\sR}}

\def\fX{\mathfrak{X}}

\def\uomega{\Omega}

\def\uomegaSs#1{\uomega^*_{#1/\sR}}
\def\uomegaSshat#1{\widehat{\uomega}^*_{#1/\sR}}
\def\omegaS#1{\omega_{#1/\sR}}
\def\omegaSs#1{\omega^*_{#1/\sR}}
\def\omegaSshat#1{\widehat{\omega}^*_{#1/\sR}}

\def\fb{\overline{f}}

\def\alphab{\overline{\alpha}}

\def\id{\mathrm{id}}

\def\qfor{\;\text{ for }}

\def\tri{\mathrm{tri}}
\def\hal{\widetilde{\alpha}}

\def\uhS{\widetilde{\sR}}
\def\utS{\sR^{\triv}}

\def\klog{k^0}

\def\Wp{W^\times}

\def\triv{\mathrm{triv}}
\def\uhA{\widetilde{\sA}}
\def\utA{\sA^{\triv}}
\def\uhB{\widetilde{\sB}}

\def\utBs{\sB^{\triv,*}}

\def\hf{\widetilde{f}}
\def\tf{f^{\tri}}
\def\uf{\underline{f}}
\def\uhf{\widetilde{\uf}}
\def\utf{\uf^{\tri}}

\def\redd#1{(#1)_{\red}}

\def\AFt#1{A^{(#1)}}
\def\sAFt#1{\sA^{(#1)}}

\def\dR{\mathrm{dR}}
\def\lSmCarpr{\mathrm{lSm}^{\mathrm{Car,pr}}}

\def\Rperf{R_{\mathrm{perf}}}

\def\lSmCar{\mathrm{lSm}^{\mathrm{Car}}}

\def\sRp{\sR^p}

\def\Ainf{A_{\mathrm{inf}}}

\def\Ftt#1#2{{#1}^{(#2)}}

\def\FttOK#1#2{{#1}_{\OK^\pi}^{(#2)}}

\def\OC{\sO_C}
\def\OK{\sO_K}

\def\OKpi{\sO^\pi_K}

\def\uW{W}

\def\uk{k}

\def\fY{\mathfrak{Y}}
\def\tsX{\fX}
\def\tsXp{\fX_p}
\def\tsY{\fY}

\def\XRm{\sX^{(m)}_{\sR}}

\def\fS{\mathfrak{S}}

\def\bL{\mathbb{L}}
\def\fSFt{\fS^{(1)}}
\def\sDFt{\sD^{(1)}}
\def\DFt{D^{(1)}}
\def\Wu{W[[u]]}
\def\xiFt{\xi^{(1)}}

\def\CAlg{\mathrm{CAlg}}

\begin{document}

\begin{abstract}
   Let $\mathcal{O}_K$ be a mixed characteristic complete DVR with perfect residue field $k$, and let $\mathfrak{X}$ be a proper formal scheme over $\mathcal{O}_K$ with semistable reduction. Answering  a question of Fontaine and Jannsen, a classical theorem of Hyodo and Kato gives a rational identification between the log crystalline cohomology of the special fiber $\mathfrak{X}_0$ over $W(k)^0$ with the de Rham cohomology of the generic fiber $\mathfrak{X}_K$. In this note, we use a log variant of the saturated de Rham--Witt complex of Bhatt--Lurie--Mathew to prove that such a comparison holds integrally.
\end{abstract}
\maketitle
\setcounter{tocdepth}{1}
\tableofcontents 
\section{Introduction}
\subsection*{Hyodo--Kato cohomology and a conjecture of Fontaine and Jannsen}Let $K$ be a complete discrete valuation field of mixed characteristic $(0,p)$, valuation ring $\sO_K$ and residue field $k$. Let $F$ be the fraction field of the ring of Witt vectors $W(k)$. A conjecture of   Fontaine and Jannsen, realized by Hyodo and Kato in \cite{HK}, predicted the existence of a weight and a monodromy filtration on the de Rham cohomology $H^*_{dR}(Y/K)$ of any smooth and proper $K$-variety $Y$. More precisely, they conjectured the existence of a graded $F$-vector space $D^*$, equipped with the structure of $(\varphi, N)$-module over $F$ and satisying $D^*\otimes_F K \simeq H^*_{dR}(Y/K)$, thus transporting to $H^*_{dR}(Y/K)$ the required structures, after a scalar extension.

One model of $D^*$ is given by the so-called Hyodo--Kato  cohomology of the special fiber of any semistable formal model $\mathfrak{X}/\mathrm{Spf}(\sO_K)$ of $Y$ (which does always exist, up to alterations,
at least after replacing $K$ with a finite extension). Let $X$ be the special fiber of $\mathfrak{X}$.  We can  endow it with the structure of a log smooth log scheme $\mathcal{X} =(X, \mathcal{M}_{\mathcal{X}})$ over the log point $k^0 =(k,\N\ni1 \to 0)$. As such, \cite{HK} provides the construction of  complexes  $R\Gamma_{ crys}(\mathcal{X}/W(k)^0)$ of $W(k)$-modules, defined in terms of the log crystalline site relative to  $W(k)^0 =(W(k), \N\ni1\mapsto 0)$. 

A Theorem of Hyodo and Kato
\cite[Theorem 5.1]{HK} gives then a quasi-isomorphism, depending on the choice of a uniformizer $\pi$ of $\sO_K$,
\begin{equation}\label{eq:HK-iso} \rho_\pi\colon R\Gamma_{ crys}(\mathcal{X}/W(k)^0) \otimes_{{W(k)}} K \simeq R\Gamma_{dR}(Y/K), \end{equation}
which is now referred to as the Hyodo--Kato isomorphism. The cohomology groups of $R\Gamma_{ crys}(\mathcal{X}/W(k)^0)$ come equipped with a monodromy operator $N$ (see \cite[\S 3.4--3.6]{HK}) and a Frobenius operator $\varphi$ (see \cite[\S 3.2]{HK}), which becomes invertible after rationalization,  satisfying $p\varphi N = N\varphi$. 
Note that in fact \cite{HK}  constructed the isomorphism only at the level of cohomology groups. 
We refer to \cite[Section 4]{ertl-yamada} for a refined statement using rigid cohomology and a complete proof (see \cite[Section 6]{ertl-yamada} for the crystalline version).  

The isomorphism \eqref{eq:HK-iso} is quite indirect. It arises from a zig-zag
\[\begin{tikzcd}
R\Gamma_{ crys}(\mathcal X/W(k)^0) & R\Gamma_{ crys}(\mathcal X/\mathcal{S})\arrow[l, "j_0^*"']\arrow[r, "j_\pi^*"] & R\Gamma_{ crys}(\mathcal X/ {\sO_K^\pi})
\end{tikzcd}
\]
where  $\OK$ is equipped with the PD ideal $(p)$, and $\OK^\pi$ is the log ring $(\OK,\N\to \OK\;;\; 1\to \pi)$ and $\sS$ is the $p$-adic completion of the PD envelope over $\Z_p$ of $W(k)[T]$ with respect to $(p,\Ker(\rho))$, where $\rho:W[T] \to \OK$ is the $W$-algebra homomorphism mapping $T$ to $\pi$. We equip $\sS$   with the log structure $\N \to \sS;\;1\mapsto T$.
The maps $j_0^*$ and $j_\pi^*$ are defined by the exact closed immersions $\Sp(W(k)^0)\to \mathrm{Spf}(\sS)$ and $\Sp(\sO_K^\pi)\to \mathrm{Spf}(\sS)$
given by setting $T\mapsto0$ and $T\mapsto \pi$ respectively. 
 A difficult step in the construction of \eqref{eq:HK-iso} is \cite[(4.13)]{HK}, which shows that 
up to inverting $p$, the map $j_0^*$ admits a section $s_\pi$, compatible with the Frobenius. 
 The composition of this section with $j^*_\pi$ turns out to be a quasi-isomorphism. Composing it further with a natural comparison map between $R\Gamma_{ crys}(\mathcal X/ \sO_K^\pi)[1/p]$  and $R\Gamma_{dR}(Y/K)$ (this is Kato's analogue of the Berthelot--Ogus isomorphism) gives the desired map \eqref{eq:HK-iso}. See \cite[Proposition-Definition 6.8]{ertl-yamada} for details.\medskip 

As remarked above, one of the main reasons to be interested in the cohomology groups of $R\Gamma_{ crys}(\mathcal X/W(k)^0)$ is that they come equipped with a natural monodromy operator $N$ (this is not the case for $R\Gamma_{ crys}(\mathcal X/ \sO_K^\pi)$), and this is exactly the ``hidden structure'' predicted by Fontaine and Jannsen. Note that if one works with rational coefficients, it is possible to completely bypass any reference to log-crystalline cohomology in the definition of the $(\varphi,N)$-structure on $R\Gamma_{dR}(Y/K)$. In fact, only classical, non logarithmic, rigid cohomology is required as input, since the monodromy operator can be completely reconstructed using the properties of rigid analytic motives. This is obtained in \cite[Theorem 1.6]{BGV}. The idea of working directly on the rigid-analytic generic fiber was considered by Beilinson \cite{BeilinsonPeriod}, and Colmez--Nizio\l~   (see for example \cite{CN}).  

\subsection*{The main result} Motivated by the interest in \emph{integral} $p$-adic Hodge theory, after \cite{BMS1} and \cite{BMS2}, our goal in this work was to understand to which extent it was possible to produce an integral refinement of \eqref{eq:HK-iso}. This required a different approach, since the classical argument of Hyodo and Kato is rational in nature, as it relies on inverting $p$ in order to apply an appropriate version of Dwork's trick. Similarly, the motivic version of \cite{BGV} (after the work of Ayoub--Gallauer--Vezzani \cite{AGV}) crucially uses rational coefficients in the identification of the relevant motivic categories, and $\mathbb{A}^1$-invariance on the special fiber together with étale descent immediately forces $p$ to be invertible in the coefficients. 

It turns out that, up to replacing log crystalline cohomology with its ``saturated derived version'', it is possible to prove an \emph{integral} comparison between log
crystalline cohomology taken relative to two different lifts of  the same
log structure modulo $p$. More precisely, write
\begin{equation}\label{W0k0OKpi}
\uW^0=(W={W(k)},\N\ni1\mapsto0),\qquad
\uk^0=(k,\N\ni1\mapsto0),\qquad
\OKpi=(\sO_K,\N\ni1\mapsto\pi),
\end{equation}
so that $\uk^0=\uW^0/p=\OKpi/(\pi)$, and let $e=v_K(p)$ be the absolute
ramification index. For $m\geq0$ with $e\leq p^m$, the element $\pi^{p^m}$ lies in
$p\sO_K$, so the $m$-fold logarithmic Frobenius of $\OKpi/p$ factors through
$\uk^0$; we write
\[
F_m\colon\uk^0\longrightarrow\OKpi/p,
\qquad
\FttOK\sX m:=\sX\times_{\uk^0,F_m}\OKpi/p
\]
for the resulting $m$-fold Frobenius twist of a log scheme $\sX$ over $\uk^0$.
Our main result is the following.
 
\begin{thm}[Theorem \ref{thm.integralHK} and Corollary \ref{cor.integralHK}]\label{thm:intro-integralHK}
Assume $e\leq p^m$. For every fine, quasi-compact and quasi-separated log-smooth
morphism $\sX\to\Spec(\uk^0)$ of Cartier type, there is a natural equivalence in $\CAlg(\mathcal{D}(\mathcal{O}_K)^\wedge_p)$
\[
\rho_{\pi,m}^{\rm int}\colon R\Gamma_{\rm crys}(\sX/\uW^0)
\ \widehat\otimes^{\mathbb L}_{W,\phi^m}\ \sO_K
\;\simeq\;
R\Gamma_{\rm crys}\bigl(\FttOK\sX m/\OKpi\bigr),
\]
depending on $\pi$, functorial for morphisms of such log schemes over $\uk^0$. Here,
$\widehat\otimes^{\mathbb L}_{W,\phi^m}\sO_K$ denotes derived $p$-completed base
change along $W\xrightarrow{\phi^m}W\to\sO_K$.

Assume that ${\mathfrak{X}}$ is a fine, quasi-compact and quasi-separated log smooth scheme of Cartier type over $\OK^\pi$ with log special fiber $\mathcal{X} = {\mathfrak{X}}\times_{\OK^\pi} k^0\to k^0$. Let $L\eta_p$ be the décalage functor, and write $L\eta_p^m$ for its $m$-th fold power.  Setting  
\begin{equation}\label{twistedDR}
R\Gamma_{\dR}^{\eta^m}({\fX}/\OKpi)
	:=
	R\Gamma\!\left(
	\sX_{\acute et},
L\eta_p^m\widehat\omega^*_{{\fX}/\OKpi}
	\right),\end{equation}
	we get, composing with Kato's crystalline--de Rham comparison  \cite[(6.4)]{katolog},  a
		natural equivalence
		\[
		R\Gamma_{\rm crys}(\sX/W^0)
		\widehat\otimes^{\mathbb{L}}_{W,\phi^m}\OK
		\simeq
		R\Gamma_{\dR}^{\eta^m}(\fX/\OKpi),
		\]
		functorial for morphisms over $\OKpi$.
\end{thm}
We remark that the twist by $L\eta_p^m$ is necessary in the ramified case, see Example \ref{ex:Tate_ramified}.
In the special case where $\OK=W$ and $\sX$ is proper over $\OK$,
Theorem \ref{thm:intro-integralHK} implies an equality (where $W^\times=(W,\N\ni 1 \to p)$)
\[
{\rm length}_W (H^i_{\rm crys}(\sX/W^0)_{\rm tors}/p^n) = {\rm length}_W (H^i_{\rm dR}({\fX}/W^\times)_{\rm tors}/p^n)\;\text{ for all } \; i,n,
\]
which is a simultaneously a special case ($e=1$) and a generalization ($\widetilde{\sX}$ does not have good reduction) of Conjecture $\alpha$ in \cite{AbhinandanYoucis}. Note that the example \ref{ex:enriques} provides a useful sanity check, exhibiting matching $2$-torsion classes in log crystalline cohomology of $\sX$ over $W^0$ and in log de Rham cohomology of ${\fX}$ over $W^\times$.  
\medbreak

Our method of the proof of Theorem \ref{thm:intro-integralHK}
makes use of the \emph{saturated de Rham-Witt complex} of Bhatt-Lurie-Mathew \cite{BLM} in its log variant established by Yao \cite{Ya}.  We need in fact a modification of Yao's theory, carried
out in \S\ref{Special pre-log algebras}, since \cite{Ya} works over a base whose log
structure is the Teichm\"uller lift of its reduction modulo $p$, and this
excludes the structure $W^\times$ ($1\mapsto p$) which is one of the two we must
compare. 
The key point is then the following: let $\mathcal{R}$ be a mixed characteristic log base, and consider  a $p$-torsion-free special
$F$-pre-log algebra $\sA$ over $\sR$ (for example, a polynomial algebra with pre-log structure given by a free monoid on a subset of the variables, and with its standard Frobenius lift, see Definition \ref{def;Flogalgebra}). Then,  maps out of the completed log de Rham
complex $\widehat\omega^*_{\sA/\sR}$ into strict log Dieudonn\'e algebras are
detected entirely on the reduction of $\sA$ modulo $p$
(Proposition~\ref{prop;omegaUniversal}). As a consequence, the functor
\[
\sP\longmapsto\WSat\bigl(\widehat\omega^*_{\sP/\sR}\bigr)
\]
factors through the reduction, giving a functor 
$\mathcal W\uomega^*_{-/\sR}\colon\mathrm{lPoly}_{\sR/p}\to\DA_R^{\str}$
(Theorem~\ref{thm;mathcaWomega}), and given two base log structures agreeing modulo
$p$, we can construct an invertible natural transformation relating them
(Corollary~\ref{main_cor_WOmega_agrees}).  As finitely generated log polynomial algebras are the compact projective objects in the category of pre-log rings, we can then animate the resulting functor to get derived log crystalline cohomology, and the required integral comparison of log crystalline cohomology over a base with two different log structures (we can refer to this as an \emph{integral transfer structure}).

To build such transformation, we use the fact that $\DA_R^{\str}$ is in fact (equivalent to) an ordinary $1$-category (as observed by \cite{BLM}), so that higher coherences are in fact automatic.

We can see a sample application. Consider a proper strictly semistable family $\mathfrak{X}$  over $W(k)$: as before, we can see it as  a log smooth vertical log scheme over $W^\times$ (the adjective vertical here refers to the fact that the log structure becomes trivial on the generic fiber). Let  $\mathcal X$ denote its special fiber over $k^0$.
 By Kato's comparison, 
 we can {identify} the relative log de Rham cohomology $R\Gamma_{dR}({\mathfrak{X}}/W^\times)$ with the log crystalline cohomology $R\Gamma_{ crys}(\mathcal X/W^\times)$. Using Theorem \ref{thm:intro-integralHK}, it is then possible to use the weight spectral sequence of Mokrane-Nakkajima \cite{nakkajima} for $R\Gamma_{crys}(\mathcal X/W^0)$ to compute it integrally.
 As a consequence, it is possible to build non-trivial torsion classes in $R\Gamma_{dR}({\mathcal{X}}/W^\times)$ by identifying classes in log crystalline cohomology relative to $W^0$, and the spectral sequence allows us to do so starting from torsion classes in the crystalline cohomology of the smooth components of $X$. We explain this procedure by means of two explicit examples, see Example \ref{ex:enriques} and Example \ref{ex:Tate_ramified}. 

\subsection*{Questions and related works}
 Consider for simplicity a proper semistable formal scheme $\mathfrak{X}$ over 
 $W=W(k)$ with $k$ algebraically closed. 
 Let $\OC$ be the ring of integers in the completion of an algebraic closure of $W[1/p]$ and $A_{\inf}=A_{\inf}(\OC)=W(\OC^\flat)$.
The  $A_{\inf}$-cohomology of \cite{BMS1}, studied in the semistable case by \cite{Cesnavicius-Koshikawa}, 
gives a complex of perfect $A_{\inf}$-modules \[
C_{A_{\inf}}(\mathfrak X)
=
R\Gamma_{A_{\inf}}(\mathfrak X_{\mathcal O_C}).
\]
There are two specialization maps called the crystalline and de Rham specialization:
\begin{equation}\label{specialization}
\vartheta:  (A_{\rm inf},\alpha_{\inf})  \to (W,\alpha_W^0)\:\text{ and }\;
 \theta\colon (A_{\rm inf},\alpha_{\inf})  \to (\OC,\alpha_C^\times),
 \end{equation}
where the relevant log structures are given by
\[\alpha_{\inf}: \sO_C^\flat-\{0\} \to A_{\inf}\;;\; x \to [x], \;
\alpha_W^0:\Q_{\geq0} \to W\;; \; a\to 0 \;(a>0), \;
\alpha_C^\times:\OC-\{0\}\to \OC,\]
and $\vartheta$ is the lift of the quotient map $\OC^\flat \to k$
and $\theta$ is Fontaine's map.
These give rise to complexes
\begin{equation}\label{Csp0x}
\operatorname{sp}_0(C_{A_{\inf}}(\mathfrak X))=C_{A_{\inf}}(\mathfrak X)\otimes_{A_{\inf},\vartheta} W
\quad\text{and}\quad
\operatorname{sp}_\times(C_{A_{\inf}}(\mathfrak X))=C_{A_{\inf}}(\mathfrak X)\otimes_{A_{\inf},\theta} \OC.
\end{equation}
Both specialization maps factor through the $p$-completion $A_{\rm crys}$ of the PD envelope of $A_{\rm inf}$ with respect to  $\Ker(\theta)$ so that the complexes \eqref{Csp0x} are viewed as specializations of the crystalline complexes (see \cite[Th.5.4]{Cesnavicius-Koshikawa}) 
\begin{equation}\label{Acrys-cohomology}
 R\Gamma_{{\rm crys}}({\mathfrak{X}}_{\OC/p} / A_{\rm crys})\simeq R\Gamma_{A_{\inf}}(\mathfrak X) \widehat{\otimes}^{\mathbb{L}}_{A_{\inf}} A_{\rm crys},\end{equation}
 where $A_{\rm crys}$ is equipped with the induced log structure from $A_{\inf}$.
 We may call the complexes \eqref{Csp0x} the  hollow-PD and compactifying log-PD specializations of $R\Gamma_{{\rm crys}}({\mathfrak{X}}_{\OC/p} / A_{\rm crys})$ respectively  (in some sense, the role of $A_{\rm crys}$ here is similar to that $\mathcal{S}$, the $p$-completed PD envelope over $\Z_p$ of $W[T]$ introduced in the beginning of the introduction for  the classical Hyodo--Kato theory).

The work of \cite{Cesnavicius-Koshikawa} does not provide directly any comparison between  the complexes \eqref{Csp0x} while our main result allows us to compare them:
Write 
\[C_{\rm HK} =R\Gamma_{\rm crys}(X/W^0)\;\text{ and }\; C_{\rm dR}=R\Gamma_{{   \rm dR}}(\mathfrak{X}/W^\times),\]
where $X$ is the special fiber of $\mathfrak{X}$ equipped with the standard log structure and $\uW^0$ is from \eqref{W0k0OKpi} and $W^\times=(W,\N\ni1\mapsto p)$.
 Write $\Phi^{\rm int}_{\pi,\mathfrak X}$ for the equivalence between $C_{\rm HK}$ and $C_{\rm dR}$ provided by Theorem \ref{thm:intro-integralHK}. 
Using $\Phi^{\rm int}_{\pi,\mathfrak X}$, we can define a ``transport map''
\begin{equation} \label{eq:intro_compatible_CK}
\tau^{A_{\inf}}_{0,\times}\colon
\operatorname{sp}_0(C_{A_{\inf}}(\mathfrak X))\otimes^{\mathbb L}_{W} \OC
\xrightarrow{\ \sim\ }
\operatorname{sp}_\times(C_{A_{\inf}}(\mathfrak X))
\end{equation}
  such that the following diagram commutes:
\[\begin{tikzcd}[column sep=large,row sep=large] \operatorname{sp}_0(C_{A_{\inf}}(\mathfrak X))\otimes^{\mathbb L}_{W} \OC \arrow[r,"\tau_{0,\times}^{A_{\inf}}"] \arrow[d,"\beta_0"'," \sim"] & \operatorname{sp}_\times(C_{A_{\inf}}(\mathfrak X)) \arrow[d,"\beta_\times","\sim"'] \\ C_{\rm HK}\widehat\otimes_W^{\mathbb L}\mathcal O_C \arrow[r,"\Phi_{\pi,\mathfrak X}^{\rm int}\widehat\otimes_W\mathcal O_C"'] & C_{\rm dR}\widehat\otimes_W^{\mathbb L}\mathcal O_C. \end{tikzcd}\]
where $\beta_0$ and $\beta_\times$ are the comparison isomorphisms provided by \cite[(1.2.1)]{Cesnavicius-Koshikawa}.

Note that the two log structures 
 in the targets of the maps $\vartheta$ and $\theta$ in \eqref{specialization} are quite incompatible. Indeed, for
 a $p$-power compatible sequence $x=(x^{(0)},x^{(1)},x^{(2)}.\dots)$ with $x^{(i)}$ in  the maximal ideal of $\sO_C$, 
$\vartheta$ maps the Teichm\"uller lift $[x]\in A_{\inf}$ to $0\in W$  
(this is exactly why the induced log structure on $W$ is hollow), whereas $\theta$ sends $x$ to $x^{(0)}\in {\sO_C}$, which is  possibly a nonzero element of $\OC$ of arbitrarily small
positive valuation, and in particular does not lie in $p\OC$. 
Thus, the two log rings $(W,\alpha_W^0)\otimes_{W}\OC$ and $(\sO_C,\alpha_C^\times)$ are not congruent modulo $p\OC$
and no transport in the sense of   Theorem \ref{thm:intro-integralHK} is directly available for this pair.

 After inverting $p$, we can formulate a precise question. As observed by Ertl and Yamada, the Hyodo--Kato isomorphism (in rigid cohomology) actually depends on the choice of a branch $\log_q$ of the $p$-adic logarithm. Note that such  a choice determines an embedding $\iota_q\colon B_{\rm st} \to B_{\rm dR}$. 

Following again \v{C}esnavi\v{c}ius--Koshikawa, from $A_{\rm inf}$ cohomology,
 we deduce
 a semistable comparison isomorphism (the semistable Fontaine--Jannsen conjecture, see \cite[Theorem 9.5]{Cesnavicius-Koshikawa})
\[ \alpha_{\rm st}^{\rm CK}\colon R\Gamma_{\acute et}(\mathfrak X_{\overline K},\mathbb Z_p) \otimes_{\mathbb Z_p}^{\mathbb L}B_{\rm st}  \longrightarrow C_{\rm HK}\otimes_W^{\mathbb L}B_{\rm st},\]
compatible with $\varphi, N$ and Galois action.
 On the other hand, we have de Rham comparison isomorphism (given e.g., by \cite[13.1]{BMS1}, see \cite[(6.7.2)]{Cesnavicius-Koshikawa})
\[\alpha_{\rm dR} : R\Gamma_{\acute et}(\mathfrak X_{\overline K},\mathbb Z_p) \otimes_{\mathbb Z_p}^{\mathbb L}B_{\rm dR} \longrightarrow C_{\rm dR}\otimes_W^{\mathbb L}B_{\rm dR}.\]  
 It is tempting to ask the following Question.

\begin{question} 
 Does the identification provided by \eqref{eq:intro_compatible_CK}  
recover, for the same choice of
a branch of the $p$-adic logarithm,  the compatibility between the
$B_{\rm st}$- and {$B_{\rm dR}$-} 
comparison isomorphisms? More precisely, does the diagram
\[\begin{tikzcd}[column sep=large,row sep=large] R\Gamma_{\acute et}({\fX}_{\overline K},\mathbb Z_p) \otimes_{\mathbb Z_p}^{\mathbb L}B_{\rm st} \arrow[r,"\alpha_{\rm st}^{\rm CK}","\sim"'] \arrow[d,"1\otimes\iota_q"'] & C_{\rm HK}\otimes_W^{\mathbb L}B_{\rm st} \arrow[d,"\Phi_{\pi,\mathfrak X}^{\rm int}\otimes\iota_q"] \\ 
R\Gamma_{\acute et}({\fX}_{\overline K},\mathbb Z_p) \otimes_{\mathbb Z_p}^{\mathbb L}B_{\rm dR} \arrow[r,"\alpha_{\rm dR}","\sim"'] & C_{\rm dR}\otimes_W^{\mathbb L}B_{\rm dR}. \end{tikzcd}  \]
commute?
\end{question}
Note that the analogous compatibility is discussed in \cite[Remark 9.6]{Cesnavicius-Koshikawa} if one replaces in the above square the right hand map with the classical Hyodo--Kato map. This follows from \cite[Proposition 9.2]{Cesnavicius-Koshikawa}, using Beilinson's construction of the Hyodo--Kato map \cite[Section 1]{BeilinsonPeriod}. So, the above compatibility could be obtained by a positive answer to the following Question, which we leave to a future work.
\begin{question}
    Let us write $\Psi_{\pi, \mathfrak X}^{\rm crys}$ for the crystalline Hyodo--Kato map \cite[Proposition-Definition 6.8]{ertl-yamada} (composed with the de Rham comparison, see \cite[Remark 6.9]{ertl-yamada}), and write $\Phi^{\rm rat}_{\pi, \mathfrak X}$ for the rationalization of $\Phi^{\rm int}_{\pi,\mathfrak X}$. Is it true that $\Phi^{\rm rat}_{\pi,\mathfrak X}=\Psi^{\rm crys}_{\pi, \mathfrak X}$, up to Frobenius twist? 
\end{question} 
Note that if $\varphi_{\rm lin}\colon\varphi^*C_{\rm HK}\to C_{\rm HK}$ is the linearized crystalline Frobenius, Example \ref{ex:Tate_ramified} shows that in the ramified case one would need to precompose $\Psi^{\rm crys}_{\pi, \mathfrak X}$ with $\varphi_{\rm lin}^m$ in order to get matching Frobenius structures.
\medskip

In a  different direction, the recent work \cite{AbhinandanYoucis} by Abhinandan and Youcis established an integral version of the Berthelot--Ogus comparison isomorphism between crystalline and de Rham cohomology of  perfect complexes of prismatic crystals on a smooth and proper formal scheme $\mathfrak {X}$ over a ramified base $\OK$ with the special fiber $X=\fX\otimes_{\OK} k$, where $\OK$ is as in the beginning of the introduction.
 More precisely, for $n\geq \lceil \log_p (\frac{e}{p-1}) \rceil$, they construct isomorphisms 
\[ R\Gamma_{\rm dR}^{(n)}({\fX/\OK,}\mathcal{E}) \xrightarrow{\sim} (\varphi^n)^*R\Gamma_{\rm crys}({X}/W, \mathcal{E}^{\rm crys}_k) \otimes_{W} \OK\] 
for $\mathcal{E}$ a prismatic crystal on $\mathfrak X$, where $W=W(k)$.
The left hand-side denotes the $n$-th twisted de Rham cohomology groups, defined by means of the de Rham point $\rho_{\rm dR} \colon \Spf(\OK) \to \OK^\Prism$,  composed with the  $n$-fold Frobenius $F^n_{\OK}$ of $ \OK^\Prism$, see \cite[Theorem A]{AbhinandanYoucis} (note that Lemma \ref{lem;CrysDRComparisontwisted} and  \cite[Proposition 4.7]{AbhinandanYoucis} imply $R\Gamma_{\rm dR}^{(m)}(\mathfrak X/\mathcal O_K)=R\Gamma_{\dR}^{\eta^m}(\fX/\OKpi)$ (cf. \eqref{twistedDR})).  We see their result as a good reduction analogue of Theorem \ref{thm:intro-integralHK},  although the comparison maps
are constructed by different methods and their equality is not
addressed here 
(note that in good reduction case, they get a the sharper ramification bound
$e\leq(p-1)p^n$). It is reasonable to expect that the diagram
(where we consider the case $\mathcal{E}=\sO_{\fX}$)
\[\begin{tikzcd}[column sep=large] \varphi_W^{m,*}R\Gamma_{\rm crys}(X/W) \otimes_W^{\mathbb L}\mathcal O_K \arrow[r,"\iota_{\rm AY}^{(m)}"] \arrow[d,equal] & R\Gamma_{\rm dR}^{(m)}(\mathfrak X/\mathcal O_K) \arrow[d,"\sim"'] \\ R\Gamma_{\log\text{-}\rm crys}(X/W^0) \widehat\otimes_{W,\varphi^m}^{\mathbb L}\mathcal O_K \arrow[r,"\rho_{\pi,m}^{\rm int}"'] & R\Gamma_{\log\text{-}\rm crys} (X_{\mathcal O_K^\pi}^{(m)}/\mathcal O_K^\pi), \end{tikzcd}\]
commutes for $\fX$ smooth and proper over $\OK$,
{where $\rho_{\pi,m}^{\rm int}$ is from Theorem \ref{thm:intro-integralHK},}
and the top (resp. right vertical) arrow is provided by \cite[Theorem A]{AbhinandanYoucis} (resp. \cite[Proposition 4.7]{AbhinandanYoucis}).

\medbreak

Another reasonable question is whether an analogue of Theorem \ref{thm:intro-integralHK} for coefficients holds or not. The most natural formulation of the question is in terms of logarithmic prismatic crystals, and a positive answer would give an integral refinement to the transfer Theorem of Ogus \cite{ogus_compositio}. We will address this question in a follow-up paper \cite{Revisiting_HK_part2}. 

\subsection*{Acknowledgments} This project began in 2023, during a visit of the second named author to the first named author at the University of Milano. The second author gratefully acknowledges generous financial support from Amnon Neeman during his visits, and thanks the Department of Mathematics at the University of Milano for its warm hospitality.

 We thank Mauro Porta for many interesting conversations about the subject of this paper,  Veronika Ertl for providing a very detailed feedback on a first draft, and Alberto Vezzani for a precious insight on the Frobenius operator on log rings. We also thank Tommy Lundemo, Alberto Merici, and Doosung Park for their interest in our work and for commenting on a preliminary version of this note, and Alice Garbagnati for providing references for Example \ref{ex:enriques}. The first named author was partially supported by the PRIN grant 20222B24AY ``The arithmetic of motives and $L$-functions''. The second author is supported by JSPS Grant-in-aid (B) \#20H01791 representative Shuji Saito.
\section{Log Dieudonn\'e algebras (after Yao)}

In this section, we present a slight modification of the theory of log Dieudonn\'e algebras developed by Yao. 
In this modified version, we relax the assumption made at the beginning of \cite[\S4]{Ya} to treat log Dieudonn\'e algebra over a base whose log structure may not be the Teichmuller lift of its reduction modulo $p$. Throughout this section we fix a perfect field $k$ of positive characteristic $p$. 
  
\subsection{Special pre-log algebras and de Rham complexes}\label{Special pre-log algebras}

Let $R$ be a $p$-complete and $p$-torsion free algebra over the ring $W(k)$ of Witt vectors of a perfect field $k$ with a fixed lift $F_R\colon R\to R$ of the Frobenius on $R/p$. 
We equip $R$ with a  pre-log structure $\alpha_R\colon N\to R$ satisfying the condition:
We have a decomposition $N=N_1\oplus N_2$ of monoids and 
\eq{eq;SN}{
\alpha_R(N_2-\{0\})\subset pR\text{ and } F_R(\alpha(n))=\alpha(p n)=\alpha(n)^p\text{ for } n\in N_1.}
Note that we don't require that the second formula holds for $n\in N_2$. 
Write $\sR=(R,N)$ for the associated pre-log ring.

\begin{rmk}
Write $F_N$ for the multiplication by $p$ map on the monoid $N$ (this will be the Frobenius on the monoid as in \cite{HK}). Note that, in general, the pair $(F_R,F_N)$ does not define a morphism of pre-log rings $\sR\to\sR$. This would be the case if $N_2=0$.
\end{rmk}

\begin{ex} The main examples of $p$-complete algebras $A$ equipped with  pre-log structures satisfying the assumption \eqref{eq;SN} are given respectively by $(W(k), \N \to W(k), 1 \mapsto p)$ and $(W(k), \N \to W(k), 1 \mapsto 0)$, i.e.,  the compactifying log structure on $W(k)$ induced by its closed point and the ``hollow'' log structure on the same ring. These are usually denoted by $W(k)^\times$ and $W(k)^0$ respectively. In both cases, the Frobenius is the standard Frobenius lift on the Witt vectors. 

     Other examples satysfying \eqref{eq;SN} include pre-log rings over $W(k)^\times$ or $W(k)^0$ equipped with horizontal log structures. E.g., one may take $\sR = (W(k)\langle T\rangle, \N\oplus \N)$ (the brackets denote $p$-adic completion) with pre-log structure given by $(1,0) \mapsto p$, $(0,1)\mapsto T$, and Frobenius  induced by $T\mapsto T^p$ (and the Frobenius lift on $W(k)$).  Yet another example is given by the ring $W(k)\langle T\rangle$ with pre-log structure given by $1\mapsto pT$. In this case $N_2=\N$ and $N_1=0$.
\end{ex}

\begin{rmk}
 Since the construction of the log de Rham complex is insensitive to passing to the associated log structure, we will mostly restrict our treatment to pre-log rings (instead of log rings), and work in the category of pre-log algebras. Given a pre-log ring $(A,P)$, we will denote by $(A,P)^a$ the associated log ring. See \ref{sec:global} below for a precise treatment of the globalization problem. 
\end{rmk}

\begin{defn}\label{def;Flogalgebra}
An $F$-pre-log algebra over $\sR$ is a pair $(\sA,F_A)$ consisting of a pre-log algebra 
$\sA=(A,\alpha\colon L \to A)$ over $\sR$ and a $F_R$-homomorphism $F_A\colon A\to A$ such that $F_A(a)\equiv a^p\mod pA$ (so, $F_A$ is a   lift of the absolute Frobenius of $A/p$).
 An $F$-pre-log algebra $(\sA,F_A)$ over $\sR$ is called \emph{special} if there is a decomposition $L=L^o\oplus N$ as monoids such that the structure map $N\to L$ is the inclusion into the second factor and that 
\eq{eq;def;Flogalgebra}{F_A(\alpha(l))=\alpha(l)^p\qfor l\in L^o.}
 A map $(\sA,F_A) \to (\sB,F_B)$ of $F$-pre-log algebras is a map $\sA\to \sB$ of pre-log rings which is compatible with $F_A$ and $F_B$. 
\end{defn}
 
\begin{rmk}\label{rmk;Flogalgebra-Frobenius}
Let $(\sA,F_A)$ be a special $F$-pre-log algebra over $\sR$.
While $F_A$ may not extend to a map of pre-log rings $\sA\to \sA$, a relative Frobenius $F_{\sA/\sR}$ is defined as follows:
Letting $\AFt 1=A\otimes_{R,F_R} R$, $F_A$ factors as 
$A \rmapo{-\otimes 1_R} \AFt {1} \rmapo{F_{A/R}} A$, 
where the former is the base change map and $F_{A/R}$ is an $R$-algebra map. By the assumption, we have a decomposition $L=L^o\oplus N$ and let
$\alpha^{(1)}:L\to \AFt 1$ be the map which on $L^o$ is the composite of $\alpha$ and the base change map $A\to \AFt 1$ and on $N$ is the composite of $\alpha_R: N\to R$ and the structure map of the $R$-algebra $\AFt 1$.
Then, the pair of $\sAFt 1=(\AFt 1,\alpha^{(1)}:L \to \AFt 1)$ and 
$F_{\AFt 1}=F_A\otimes F_R$ gives an $F$-pre-log algebra over $\sR$ and 
\[F_{\sA/\sR}=(F_{A/R},p/N) :\sAFt {1}\to \sA\]
is a map of $F$-pre-log algebras over $\sR$, where $p/N: L \to L$ is the map of monoids given by the multiplication $p$ on $L^o$ and the identity on $N$. 
Moreover, $F_{\sA/\sR}$ modulo $p$ is identified with 
the relative Frobenius map of pre-log rings (see \cite[(2.19]{HK})
\[\sA/p \otimes_{\sR/p,F_{\sR/p}} \sR/p \to \sA/p,\]
where $F_{\sR/p}$ is given by the Frobenius on $R/p$ and the multiplication by $ p$ on $N$.
\end{rmk}

\begin{rmk}\label{rmk;Flogalgebra}
\begin{itemize}
\item[(1)]
An example of a special $F$-pre-log algebra $(\sA,F_A)$ over $\sR$ is a free object, that is, an algebra of the form
\[ A = R[X_1, \ldots, X_n, Y_1, \ldots, Y_m], \quad \alpha_R = (\alpha_Y , \alpha_R) \colon \mathbb{N}^m\oplus N \to A
\]
where $\alpha_Y(e_i) = Y_i$ for $i=1, \ldots, m$, with the obvious choice of Frobenius lift mapping the variables to its $p$-th powers. 
\item[(2)]
On the other hand, note that in case $\sR=(W(k),\N^\times)$, the log-smooth algebra $(W(k)[x,y]/(xy-p),\N^2)$ over $\sR$ cannot be extended to a special $F$-pre-log algebra $(\sA,F_A)$ over $\sR$. 

\end{itemize}
\end{rmk}
\begin{rmk}\label{rmk:log_poly}
    We let $\mathrm{lPoly}_{\sR}$ be the full subcategory of $\mathrm{PreLog}_{\sR}$ of pre-log $\sR$-algebras spanned by free objects. It is easy to see that the subcategory $\mathrm{PreLog}_{\sR}^{\rm sfp}$ of strongly finitely presented objects in $\mathrm{PreLog}_{\sR}$ consists of retracts of free objects. In particular, we can consider   category of animated pre-log rings as   $\mathrm{Ani}(\mathrm{PreLog}_{\sR}) \simeq \Fun^\Pi(\mathrm{lPoly}_{\sR}^{\rm op}, \mathcal{S})$.
\end{rmk}
\begin{lem}\label{lem;Flogalgebra}
Let $A$ be a reduced $\F_p$-algebra and $\alpha:L\to W(A)$ be a pre-log structure such that 
$F_{W(A)}(\alpha(l))=\alpha(l)^p$ for $l\in L$, where $F_{W(A)}$ is the Witt vector Frobenius. 
Then, we have $\alpha=[-]\circ \alphab$, where 
$\alphab:L\to A$ is the composite of $\alpha$ and the projection $W(A)\to A$ and 
$[-]:A\to W(A)$ is the Teichmuller lift. 
\end{lem}
\begin{proof}
The proof is verbatim the same as \cite[Proposition 3.13]{Ya}.
\end{proof}
\medbreak

We say that a pre-log structure on $W(A)$ as in Lemma \ref{lem;Flogalgebra} is of Teichmuller type. Note that for such log structures, the pair $(F_{W(A)}, p)$ where $p$ denotes the multiplication by $p$ on $L$, defines a morphism of pre-log rings $(W(A),L)\to (W(A),L)$. 

\begin{defn}\label{def;logDA}(\cite[Definition 3.1]{Ya})
A (pre)-log Dieudonn\'e algebra over $\sR$ is a tuple
\[ \sA^*=((A^*,d:A^*\to A^{*+1}),F\colon A^*\to A^*,\alpha\colon L \to A^0,\delta\colon L\to A^1)\]
where
\begin{enumerate}
\item
$(A^*,d)$ is a strict cdga over $R$ (cohomologically graded),
\item
$F:A^*\to A^*$ is a graded algebra homomorphism,
\item
$\sA^0=(A^0,\alpha)$ is a (pre)-log algebra over $\sR$,
\item
$\delta$ is a map of monoids,
\end{enumerate}
satisfying the following conditions:
\begin{itemize}
\item[(i)]
$(A^*,d ,F)$ is a Dieudonn\'e algebra in the sense of \cite[Definition 3.1.2]{BLM}.
\item[(ii)]
$(A^0,\alpha,F_{|A^0})$ is an $F$-pre-log algebra over $\sR$.

\item[(iii)]
$\delta=F\delta$.
\item[(iv)]
$(d:A^0\to A^1,\delta)$ is a log derivation of $\sA^0$ over $\sR$ (in particular, $\delta(n)=0$ for every $n$ in $N$) and 
$L\rmapo{\delta} A^1\rmapo{d} A^2$ is $0$. 
\end{itemize}
We say that a log Dieudonn\'e algebra $\sA^*$ is special if so is the $F$-pre-log algebra $(\sA^0,F_{A^0} = F)$ over $\sR$.

A morphism $\sA^*\to \sB^*$ of log Dieudonn\'e algebras is a pair $(f,\psi)$ consisting of 
a morphism 
\[ f\colon (A^*,d_A,F_A)\to (B^*,d_B,F_B)\] of Dieudonn\'e algebras over $R$ and a morphism $\psi\colon L_A\to L_B$ of monoids over $N$ satisfying the conditions:
\begin{itemize}
\item[(i)]
The following diagram commutes:
\[\xymatrix{
L_A \ar[r]^-\psi\ar[d]^{\alphab_A} & L_B\ar[d]^{\alphab_B}\\
(A^0/p)_{\red} \ar[r]^-f &(B^0/p)_{\red}\\}\]
where $\alphab_A$ is the composite of $\alpha_A:L_A\to A^0$ and $A^0\to \redd{A^0/p}$.
\item[(ii)]
$f\circ \delta_A=\delta_B\circ \psi$.
\end{itemize}

We let $\DAlogS$ be the category of $p$-torsion free pre-log Dieudonn\'e algebra and let $\DAlogSs$ be the full subcategory spanned by special objects.
\end{defn}
\begin{rmk}\label{rmk;logDA}
The condition (i) is more relaxed than \cite[Definition 3.3(2)]{Ya}. 
\end{rmk}

\begin{defn}
    Let $\sA=(A,\alpha:L\to A)\in \mathrm{PreLog}_{\sR}$. We denote by $\omega_{\sA/\sR}^*$ the log de Rham complex of $\sA$ over $\sR$. It is a (strict) cdga over $A$. Following \cite[Variant 3.3.1]{BLM}, we denote by  $\widehat{\omega}_{\sA/\sR}^*$ the inverse limit $\varprojlim_n (\omega_{\sA/\sR}^* /p^n)$, and refer to it as the completed log de Rham complex. Note that we have $\widehat{\omega}_{\sA/\sR}^* /p \simeq {\omega}_{\sA/p/\sR/p}^*$.
\end{defn}

The argument of  \cite[Proposition 4.5]{Ya} gives that the log de Rham complex of a special $F$-pre-log algebra has a canonical Frobenius lift (note that as in \cite[Variant 3.3.1]{BLM}, the Frobenius lift induces a lift on the completed log de Rham complex, still denoted by $F$). We summarize this fact in the following Proposition (with some small modifications compared to \cite{Ya}). 
\begin{prop}\label{prop;omegaDAlog}
Let $(\sA=(A,\alpha:L\to A),F_A)$ be a $p$-torsion free special $F$-pre-log algebra over $\sR$. Assume that
$\omega^i_{\sA/\sR}$ is $p$-torsion free for every $i\geq 0$. Then, there exists a unique graded ring homomorphism
\[ F\colon \omegaSs {\sA} \to  \omegaSs {\sA}\]
which extends $F_A$ such that
\[ F(dx)=x^{p-1} dx+ d\big(\frac{F(x)-x^p}{p}\big)\qfor x\in A,\]
\[ F(d\log l)=d\log l\qfor l\in L.\]
Moreover, it satisfies $dF=pFd$ and this makes 
\[\uomegaSs{\sA}:=(\omegaSs \sA,d,F,\alpha:L\to A,d\log:L\to \omega^1_{\sA/\sR})\] an object of $\DAlogSs$. Similarly, the completed log de Rham complex $\widehat{\omega}_{\sA/\sR}^*$ gives rise to an object $\uomegaSshat \sA=(\widehat{\omega}_{\sA/\sR},d,F,\alpha:L\to A,d\log:L\to \widehat{\omega}_{\sA/\sR}^1)$ of  $\DAlogSs$, equipped with a canonical ``completion map'' $\uomegaSs \sA \to \uomegaSshat \sA$. 
\end{prop}
\begin{proof}The proof is elementary. 
Put
\[
\theta_A(x):=\frac{F_A(x)-x^p}{p}\in A.
\]
Regard $\omega^1_{\sA/\sR}$ as an $A$-module through $F_A$, and define
\[
D(x):=x^{p-1}dx+d\theta_A(x)\;\text{ for} x\in A,
\qquad
\Delta(l):=d\log l\;\text{ for } l\in L.
\]
Exactly as in the proof of \cite[Proposition~4.5]{Ya}, the explicit
formula for $\theta_A$ shows that $D$ is additive and satisfies
\[
D(xy)=F_A(x)D(y)+F_A(y)D(x).
\]
The only additional verification is that $(D,\Delta)$ is relative to
$\sR$ and satisfies the logarithmic relation. If $r\in R$, then
$dr=0$ and
\[
\theta_A(r)=\frac{F_R(r)-r^p}{p}\in R,
\]
so $D(r)=0$. Also $\Delta(n)=0$ for $n\in N$. If $l\in L^o$, then
specialness gives $F_A(\alpha(l))=\alpha(l)^p$, whence
\[
D(\alpha(l))
=
\alpha(l)^{p-1}d\alpha(l)
=
F_A(\alpha(l))\,d\log l.
\]
If $n\in N$, both $d\alpha(n)$ and $d\log n$ vanish relatively, and the
correction term $d\theta_A(\alpha(n))$ also vanishes because
$\theta_A(\alpha(n))$ comes from $R$. Hence
\[
D(\alpha(n))=0=F_A(\alpha(n))\Delta(n).
\]
Using $L=L^o\oplus N$ and the  Leibniz rule, the same identity
holds for every $l\in L$. Thus $(D,\Delta)$ is a logarithmic derivation
of $\sA$ over $\sR$ with values in the $F_A$-twist of
$\omega^1_{\sA/\sR}$.
The universal property of logarithmic differentials gives the required
$F_A$-semilinear map in degree one, and hence a unique graded ring map on
the logarithmic de Rham algebra. On the generators one has
\[
dF_A(x)
=d\bigl(x^p+p\theta_A(x)\bigr)
=p\bigl(x^{p-1}dx+d\theta_A(x)\bigr)
=pF(dx).
\]
Moreover, $dF(dx)=0=pF(d^2x)$ and
$dF(d\log l)=0=pF(d^2\log l)$. Since the de Rham algebra is generated
by $A$, the elements $dx$, and the elements $d\log l$, this proves
$dF=pFd$. The equality $F(d\log l)=d\log l$ is part of the
construction, and uniqueness follows from the same set of generators.
Finally, $F$ preserves $p^n\omega^*_{\sA/\sR}$ for every $n$, so it
extends uniquely to the $p$-adic completion. Note that the assumption that
$\omega^i_{\sA/\sR}$ is $p$-torsion free for every $i\geq 0$ is only needed to guarantee that the resulting object $\Omega^*_{\mathcal{A}/\mathcal{R}}$ (resp.~$\widehat{\Omega}^*_{\mathcal{A}/\mathcal{R}})$ is indeed an object of $\DAlogSs$ (the $p$-torsion free condition is required in the definition). 
\end{proof}

\begin{rmk}\label{rmk;prop;omegaDAlog}
Under the notation of Proposition \ref{prop;omegaDAlog},
we let $\omegaSs{\sA}$ denote $(\omegaSs \sA,d,F)$, which is the underlying Dieudonn\'e algebra of $\uomegaSs{\sA}$ forgetting log structures.
\end{rmk}

\medbreak
Recall that a morphism of pre-log rings $(R,P)\to(A,M)$ of characteristic $p$ is called of log-Cartier type (or simply of Cartier type) if it is integral \cite[Definition I.4.6.2]{ogu} and if the relative Frobenius \cite[Definition III.2.4.1]{ogu} is exact \cite[Definition I.2.1.15]{ogu}.

\begin{rmk}It turns out that an integral morphism of fine and saturated log schemes in characteristic $p$ is of Cartier type if and only if it is saturated in the sense of Tsuji \cite[Definition I.3.12]{tsuji}. 

The main reason why one cares about morphisms of Cartier type is the analogue of the classical Cartier isomorphism: we will need it only in the following form. 
Let $\sA=(A,L)$ be a pre-log ring and assume that $\Spec(\sA)$ is  log-smooth of Cartier type over $\sR/p=(R/p,N)$. Then, the Cartier isomorphism holds, i.e., there is a family of isomorphisms for each $i\geq 0$:
\eq{eq;Cartier}{
C^{-1}\colon \omega^i_{\AFt 1/(\sR/p)}=\omega^i_{\sA/(\sR/p)}\otimes_{R/p,F_R} R/p \simeq  H^i(\omega^*_{\sA/(\sR/p)}), }
where $\AFt 1=\sA\times_{\sR/p,F_R} \sR/p$. 
We refer to $C^{-1}$ as the inverse Cartier operator. 
\end{rmk}
\begin{rmk}\label{rmk:Cartier_smooth_flat}
    We will also use the following property of log smooth morphisms of Cartier type. Let $\sR = (R, N)$ be a $p$-complete, $p$-torsion free pre-log ring, and let $\sA=(A,L)$ be a log smooth algebra over $\sR$. If it is of Cartier type, then it is in particular integral (see \cite[ p.~347]{ogu}), hence by \cite[Theorem IV.4.3.5(1)] {ogu} (note that the Noetherian assumption is not used in the proof of (1) in loc.cit.), the underlying morphism $R\to A$ is flat. In particular, $A$ is $p$-torsion free. 
\end{rmk}
\begin{defn}\label{def;strictDAlog}
[See \cite[Definitions~3.8, 3.9 and 3.11]{Ya}.]
A pre-log Dieudonn\'e algebra
\[
\sA^*=(A^*,d,F,\alpha:L\to A^0,\delta:L\to A^1)
\]
over $\sR$ is \emph{saturated} if its underlying Dieudonn\'e algebra is
saturated in the sense of \cite[Definition~3.4.1]{BLM}. Let
\[
i_A\colon A^*\longrightarrow\Sat(A^*)
\]
be the saturation map of the underlying Dieudonn\'e algebra. We equip
$\Sat(A^*)$ with the same monoid $L$ and with maps
\[
\alpha_{\rm sat}:=i_A^0\circ\alpha,
\qquad
\delta_{\rm sat}:=i_A^1\circ\delta.
\]
The resulting pre-log Dieudonn\'e algebra is denoted by $\Sat(\sA^*)$.

A saturated pre-log Dieudonn\'e algebra is \emph{strict} if its
underlying saturated Dieudonn\'e algebra is strict in the sense of
\cite[Definition~3.5.6]{BLM}. If
\[
\rho_A\colon A^*\longrightarrow W(A^*)
=
\varprojlim_r A^*/\mathrm{Fil}^rA^*
\]
is the strictification map of the underlying Dieudonn\'e algebra, we
equip $W(A^*)$ with the same monoid $L$ and with maps
\[
\alpha_W:=\rho_A^0\circ\alpha,
\qquad
\delta_W:=\rho_A^1\circ\delta.
\]
We write $W(\sA^*)$ for the resulting strict pre-log Dieudonn\'e
algebra and set
\[
\WSat(\sA^*):=W(\Sat(\sA^*)).
\]
\end{defn}

\begin{lem}\label{lem;log-saturation-strictification}
The constructions $\Sat(-)$ and $W(-)$ are respectively left adjoint to
the inclusions of saturated and strict pre-log Dieudonn\'e algebras.
They preserve special objects. Hence $\WSat(-)$ restricts to a
left adjoint
\begin{equation}\label{eq:WSat_left_adjoint}
\WSat(-)\colon
\DAlogSs\longrightarrow(\DAlogSs)^{\rm str}.
\end{equation}
\end{lem}

\begin{proof}
The underlying adjunctions are those of Bhatt--Lurie--Mathew. Let
$(f,\psi)\colon\sA^*\to\sB^*$ be a morphism with $\sB^*$ saturated.
The underlying morphism of Dieudonn\'e algebras extends uniquely to
$\Sat(A^*)\to B^*$. Since the monoid is unchanged and
$\alpha_{\rm sat}=i_A^0\alpha$, $\delta_{\rm sat}=i_A^1\delta$, the
extended map still satisfies  compatibility with $\delta$. This proves the saturation adjunction.
The proof for strictification is identical, using $\rho_A$ in place of
$i_A$.

If $\sA^*$ is special and $l\in L^o$, then
\[
F(\alpha_{\rm sat}(l))
=i_A^0(F(\alpha(l)))
=i_A^0(\alpha(l)^p)
=\alpha_{\rm sat}(l)^p,
\]
and the same calculation with $\rho_A$ proves that strictification
preserves the condition of being special.
\end{proof}

We set (see \S\ref{Special pre-log algebras} for $F_R$)
\[\Rperf=\colim(R\rmapo{F_R} R \rmapo{F_R} R\rmapo{F_R} \cdots).\]
The map $F_R$ induces a ring map $\Rperf\to \Rperf$, which is denoted also by $F_R$.  

\begin{lem}\label{lem;WSatomegaV}
Let $ (\sA=(A,\alpha:L \to A),F_A)$ be a special $F$-pre-log algebra over $\sR$ such that 
$\sR\to \sA$ is log-smooth. Assume that $\sA/p$ is  of Cartier type over $\sR/p$, and that $A$ is $p$-torsion free. 
\begin{itemize}
\item[(1)]
The natural map $\omega^*_{\sA/\sR}\to \WSat(\omegaSs{\sA})$ induces a quasi-isomorphism (cf. Remark \ref{rmk;prop;omegaDAlog})
\[\omega^*_{\sA/\sR}/p^r\otimes_R \Rperf \simeq \WSat(\omegaSs{\sA})/V^r\qfor r>0,\]
where $\WSat(\omegaSs{\sA})/V^r= \WSat(\omegaSs{\sA})/\im(V^r)+\im(dV^{r})$ (cf. \cite[2.5.1]{BLM}).
\item[(2)]
The above map induces isomorphisms
\[ \omega^*_{\sA/\sR}/p\otimes_R \Rperf \simeq \widehat{\omega}_{\sA/\sR}^*/p \otimes_R \Rperf\simeq  \WSat(\omegaSshat{\sA})/V.\]
\end{itemize}
\end{lem}
\begin{proof}

For $n\geq 1$, define $\AFt n$ and $F_{\AFt n}:\AFt n\to \AFt n$ inductively by $\AFt n=\AFt{n-1}\otimes_{R,F_R} R$ and $F_{\AFt n}=F_{\AFt {n-1}}\otimes F_R$. 
Then,  $F_{\AFt n}$ factors as 
\[\AFt n \rmapo{-\otimes 1_R} \AFt {n+1} \rmapo{F_{\AFt n/R}} \AFt n,\] 
where the former is the base change map and $F_{\AFt n/R}$ is a homomorphism over $R$. By the assumption, we have a decomposition $L=L^o\oplus N$ and let
$\alpha^{(n)}:L\to \AFt n$ be the map which on $L^o$ is the composite of $\alpha$ and the base change map $A\to \AFt n$ and on $N$ is the composite of $\alpha_R: N\to R$ and the structure map of the $R$-algebra $\AFt n$.
Then, the pair of $\sAFt n=(\AFt n,\alpha^{(n)}:L \to \AFt n)$ and $F_{\AFt n}$
gives an $F$-pre-log algebra over $\sR$ and 
\[F_{\sAFt n/\sR}=(F_{\AFt n/R},p/N) :\sAFt {n+1}\to \sAFt n\]
is a map of $F$-pre-log algebras over $\sR$, where $p/N: L \to L$ is the map of monoids given by the multiplication $p$ on $L^o$ and the identity on $N$. 

By the same proof of \cite[Th.8.3]{bo} (for this we use the assumption that $\mathcal{A}/p$ is of Cartier type), we have a quasi-isomorphism
\[(F_{\sAFt{n}/\sR})_*: \omega^*_{\sAFt {n+1}/\sR}\simeq \eta_p(\omega^*_{\sAFt {n}/\sR}).\]
In view of \cite[Proposition 2.4.6]{BLM} (note that since $A$ is $p$-torsion free, and $\sA$ is log smooth over $\sR$, the log differentials are locally free $A$-modules, hence $p$-torsion free) 
these induce quasi-isomorphisms for $r>0$
\[\upsilon_n:  \omega^*_{\sAFt {n}/\sR}/p^r\simeq \eta_p(\omega^*_{\sAFt {n-1}/\sR})/p^r
\simeq \eta_p^2(\omega^*_{\sAFt {n-2}/\sR})/p^r \simeq \cdots \simeq \eta_p^n(\omega^*_{\sA/\sR})/p^r.\]
They fit into a commutative diagram
\[ \xymatrix{
 & \omega^*_{\sAFt {1}/\sR}/p^r \ar[r]\ar[d]_\simeq^{\upsilon_1} & 
\omega^*_{\sAFt {2}/\sR}/p^r \ar[r]\ar[d]_\simeq ^{\upsilon_2} & \cdots \\
\omega^*_{\sA/\sR}/p^r \ar[r]^-{\alpha_F}
& \eta_p(\omega^*_{\sA/\sR}))/p^r \ar[r]^-{\alpha_F}
& \eta_p^2(\omega^*_{\sA/\sR})/p^r \ar[r]^-{\alpha_F}& \cdots \\}\]
where $\alpha_F$ is defined as \cite[2.1.4]{BLM}. 
Noting the base change isomorphism
\[ \omega^*_{\sAFt {n+1}/\sR} \simeq \omega^*_{\sAFt {n}/\sR}\otimes_{R,F_R} R,\]
this yields a quasi-isomorphism
\[ \omega^*_{\sA/\sR}\otimes_R \Rperf/p^r=\omega^*_{\sA/\sR}/p^r\otimes_{R}\Rperf/p^r
\simeq \Sat(\omegaSs{\sA})/p^r\simeq \WSat(\omegaSs{\sA})/V^r,\]
where the last quasi-isomorphism follows from \cite[2.7.3]{BLM}. This proves (1).

Next, we prove (2). Note that,  thanks to Remark \ref{rmk:Cartier_smooth_flat}, the assumption  implies that 
$\omega^*_{\sA/\sR}$ is $p$-torsion free and 
$\omega^*_{\sA/\sR}/p=\omega^*_{\sA/p/(\sR/p)}$.
We have the following commutative diagram 
\[\xymatrix{
\omega^*_{\sA/p/(\sR/p)}\otimes_R \Rperf\ar[r]^-\simeq\ar[d]_\simeq^{C^{-1}} & 
\widehat{\omega}^*_{\sA/\sR}/p\otimes_R \Rperf\ar[r]\ar[d]^{F}  &  \WSat(\omegaSshat{\sA})/V \ar[d]^F_{(**)} \\
 H^*({\omega}^*_{\sA/p/(\sR/p)}\otimes_R \Rperf) \ar[r]^\simeq  & H^*(\widehat{\omega}^*_{\sA/\sR}/p\otimes_R \Rperf) 
  \ar[r]^-{(*)} &  H^*(\WSat(\omegaSshat{\sA})/V)\\}\]
where $C^{-1}$ is the limit of isomorphisms (cf. \eqref{eq;Cartier})
\[ C^{-1}: \omega^*_{\AFt n/p/\sR/p}\simeq H^*({\omega}^*_{\AFt{n-1}/p/\sR/p}),\]
and the middle and right vertical maps are induced by the Frobenius $F$ given by Proposition \ref{prop;omegaDAlog} (see \cite[\S2.4]{BLM}, and note that the completion is harmless since we are working modulo $p$). 
We remark that the horizontal arrows are morphisms of $\Rperf$-modules, while the vertical maps are $F_A$-semilinear (cfr. with \cite[Example 7.6.5]{BLM}).
The commutativity of the left square is checked by using the formulae in Proposition \ref{prop;omegaDAlog} (see \cite[Example 3.3.5]{BLM}).  
The map $(*)$ is an isomorphism since 
$\widehat{\omega}^*_{\sA/\sR}/p\otimes_S \Rperf\to \WSat(\omegaSshat{\sA})/V$ is a quasi-isomorphism by (1).
Finally, the map $(**)$ is an isomorphism by \cite[2.7.1 and 2.7.2]{BLM}. 
This completes the proof of the lemma.
\end{proof}

\begin{prop}\label{prop;omegaUniversal}
Let $(\sA=(A,\alpha:L\to A),F_A)$ be as in Proposition \ref{prop;omegaDAlog} and 
$B^*$ be an object of $\DAlogSs$. Assume that the underlying Dieudonn\'e complex of $B^*$ is strict.  
Then, the canonical map
\[ \Hom_{\DAlogS}(\widehat{\uomega}^*_{\sA/\sR}, B^*)\to \Hom((\redd {A/p},L),(\redd{B^0/p},L_B)),\]
induced by reduction modulo $p$ on the degree zero part, is an isomorphism, where the right hand side is the set of homomorphisms of pre-log algebras over $(\sR/p)_{\rm red}$, i.e.~the set of pairs $(\fb,\psi)$ consisting of morphisms $\fb\colon   \redd {A/p}\to\redd{B^0/p}$ of $(R/p)_{\rm red}$-algebras and a map
$\psi:L\to L_B$ of monoids over $N$ such that the following diagram commutes:
\eq{eq1;omegaUniversal}{
\xymatrix{
L \ar[r]^-\psi\ar[d]^{\alphab} & L_B\ar[d]^{\alphab_B}\\
(A/p)_{\red} \ar[r]^-{\fb} &(B^0/p)_{\red}\\}}
\end{prop}
\begin{proof}
Given a pair $(\fb,\psi)$ as above, we want to extend it to a morphism
$f\colon {\uomega}^*_{\sA/\sR}\to B^*$ in $\DAlogS$. Note that since $B^*$ is a strict Dieudonné algebra, then it is automatically $V$-complete, hence $p$-adically complete, and therefore if the morphism $f$ exixts it automatically factors through the $p$-adic completion $\widehat{\uomega}^*_{\sA/\sR}$.  Since $V(1)=p$, the quotient  $B^0/VB^0$ is an $\F_p$-algebra and there is a natural surjective ring homomorphism $B^0/p \to B^0/VB^0$. By
\cite[Lemmma 3.6.1]{BLM} the ring $B^0/VB^0$ is reduced, hence the previous map factors as $(B^0/p)_{\red} \to B^0/VB^0$. This map is readily seen to be injective using the relation $FV(y)=py$, so that it is an isomorphism. 

Next, combining this with    \cite[Proposition 3.6.3]{BLM}, we obtain a natural isomorphism
\[\Hom_F(A,B^0)\simeq\Hom(\redd{A/p},B^0/VB^0),\]
where the left hand side is the set of $F$-homomorphisms $A\to B^0$ over $R$
and the left hand side is the set of homomorphisms $\redd{A/p}\to B^0/VB^0$ over $R/p$.
More precisely, for a given map $\fb\colon \redd{A/p} \to B^0/VB^0$ over $R/p$, there exists a unique $F$-homomorphism $f\colon A\to B^0$ over $A$ which fits into the commutative diagram
\eq{eq2;omegaUniversal}{ \xymatrix{
A\ar[r]^-f \ar[d]^{\upsilon_A} & B^0\ar[d]^{\upsilon_B}_\simeq\\
W(\redd{A/p}) \ar[r]^-{W(\fb)} \ar[d] & W(B^0/VB^0)\ar[d]\\
\redd{A/p} \ar[r]^-{\fb} & B^0VB^0
}
}
where $\upsilon_A$ is the $F$-homomorphism induced by the universality of 
$W(\redd{A/p})$ as the cofree $\delta$-ring on $\redd{A/p}$ and similarly for 
$\upsilon_B$, which is an isomorphism by \cite[Proposition 3.6.2]{BLM} (the map $W(\redd{A/p})\to B^0$  induced by composition with the inverse of $v_B$ is precisely the one provided by \cite[Proposition 3.6.3]{BLM}).
Note that $f$ is automatically a map of $R$-algebras. Indeed, let $\iota_A\colon R\to A$ and $\iota_B\colon R\to B^0$ be the structure maps. Both $\iota_B$ and $f\circ \iota_A$ are $F$-homomorphisms $R\to B^0$, and by \cite[Proposition 3.6.3]{BLM} again applied to $R$ (which is $p$-torsion free and equipped with a Frobenius lift $F_R$) such a map is determined by the reduction $(R/p)_{\rm red}\to B^0/VB^0$. Now, the two reductions agree because $\overline{f}$ is a morphism of $(R/p)_\mathrm{red}$-algebras. Hence $f\iota_A=\iota_B$ (and in particular $d_B\circ f$ kills $\iota_A(R)$)

By the assumption, we have 
\[L=L^0\oplus N=L^0\oplus N_1\oplus N_2\text{ and }
L_B=L_B^0\oplus N=L_B^0\oplus N_1\oplus N_2.\]

By \eqref{eq;SN} and \eqref{eq;def;Flogalgebra}, we have
$F_A(\alpha(l))=\alpha(l)^p$ (resp. $F_B(\alpha_B(l))=\alpha_B(l)^p$) for $l\in L^o\oplus N_1$ (resp. $l\in L_B^o\oplus N_1$).
Hence, by Lemma \ref{lem;Flogalgebra}, the following diagrams are commutative:  
\eq{eq3;omegaUniversal}{
\xymatrix{
L^o\oplus N_1\ar[r]^-\alpha\ar[d]^{\alphab} & A\ar[d]^{\upsilon_A}\\
\redd{A/p}\ar[r]^-{[-]} &W(\redd{A/p})\\}\quad
\xymatrix{
L_B^o\oplus N_1\ar[r]^-{\alpha_B}\ar[d]^{\alphab_B} & B^0\ar[d]^{\upsilon_B}\\
B^0/VB^0\ar[r]^-{[-]} &W(B^0/VB^0)\\} }
Now, we extend $f\colon A\to B^0$ to a morphism $f\colon \omegaSs \sA \to B^*$ of cdga's over $A$. For this, we construct the map $f^1$ which makes the right and outer squares of the following diagram
\eq{eq3.5;omegaUniversal}{
\xymatrix{
L = N\oplus L^o \ar[r]^-{\alpha} \ar@/^18pt/[rr]^-{\delta_A}\ar[d]^{\psi} & A \ar[r]^- d\ar[d]^f
& \omegaS{\sA}^1\ar[d]^{f^1}\\ 
L_B=N \oplus L_B^o \ar[r]^-{\alpha_B} \ar@/_18pt/[rr]_-{\delta_B} & B^0 \ar[r]^-{d_B}& B^1\\}}
 commute (beware that the left square may not commute). 
Define $d_f=d_B\circ f:A \to B^1$ and $\delta_f=\delta_B\circ \psi:L\to B^1$.
It is easy to check that $d_f$ is a derivation over $A$. Thus, it suffices to show that 
$(d_f,\delta_f)$  is a log derivation over $\sR$.
For $n\in N$, we have $\delta_B(n)=0$ and $\psi(n)=n$ by definition so that $\delta_f(n)=0$.
Since $f$ is a map of $R$-algebras, we have $f(\alpha(n))=\alpha_B(n)$ so that $d_f(\alpha(n))=d_B(f(\alpha(n)))=d_B(\alpha_B(n))=0$.
It remains to show 
\eq{eq4;omegaUniversal}{
f(\alpha(l))\cdot \delta_f(l) = d_f(\alpha(l))\overset{\text{def}}{=} d_B(f(\alpha(l)))\qfor l\in L.}
We may assume $l\in L^o$. Write 
$\psi(l)=b\cdot n =b\cdot n_1 \cdot n_2$ with $b\in L_B$, $n\in N$, $n_i\in N_i$ for $i=1,2$. 
Then, we have
\begin{multline*}
\upsilon_B(f(\alpha(l)))\overset{(*1)}{=} W(\fb)(\upsilon_A(\alpha(l))\overset{(*2)}{=}  
W(\fb)([\alphab(l)])=[\fb(\alphab(l))]\overset{(*3)}{=}[\alphab_B(\psi(l))]\\
=[\alphab_B(b, n_1,n_2)]= 
 \begin{cases} 
[\alphab_B(b,n_1)]\overset{(*4)}{=}  \upsilon_B(\alpha_B(\psi(l)))  \;, &  n_2=0,\\ 
[\alphab_B(b,n_1)]\cdot [\alphab_B(n_2)]\overset{(*5)}{=} 0\;, &n_2\not=0.  
\end{cases}\end{multline*} 
where $(*1)$ follows from \eqref{eq2;omegaUniversal}, $(*2)$ from \eqref{eq3;omegaUniversal}, $(*3)$ follows from \eqref{eq1;omegaUniversal},
$(*4)$ follows from \eqref{eq3;omegaUniversal}, and $(*5)$ holds since
$\alphab_B(n_2)=0$ if $n_2\not=0$ thanks to \eqref{eq;SN}. 
Since $v_B$ is an isomorphism,  we get
\[ f(\alpha(l))= \begin{cases} 
  \alpha_B(\psi(l))  \;, &  n_2=0, \\ 
 0\;, & n_2\not=0. 
\end{cases}\]
Hence, if $n_2\not=0$, \eqref{eq4;omegaUniversal} holds as the both terms are $0$.
If $n_2=0$, then
\[ f(\alpha(l))\cdot \delta_f(l)=\alpha_B(\psi(l))\cdot \delta_B(\psi(l))=
d_B(\alpha_B(\psi(l))) =d_B(f(\alpha(l)))=d_f(\alpha(l)).\]
This prove \eqref{eq4;omegaUniversal} so that we get the map $f^1$ as in  \eqref{eq3.5;omegaUniversal}.

Since $\omegaSs {\sA}$ ia free cdga generated by $\omega^1_{\sA/\sR}$ over $A$, we can extend $f^1$ to a cdga morphism $f\colon \omegaSs {\sA} \to B^*$, and  by  construction we have $f^1\circ \delta_A=\delta_B\circ \psi$ (this is the proven commutativity of the outer square in \eqref{eq3.5;omegaUniversal}). It remains to prove that $f$ is compatible with $F$.
Since $f\colon A\to B^0$ is compatible with $F$ by the construction, it suffices to check this on $dx$ and $\delta_A(l)=d\log l$ for $x\in A$ and $l\in L$.
For $dx$, we have
\[ pf(F(dx))=f(pF(dx))=f(dF(x))= d(f(F(x)))=d(F(f(x)))=pFd(f(x))=pFf(dx).\]
Thus, we get $fF(dx))=Ff(dx)$ since $B^*$ is $p$-torsion free.

Finally we check $fF\delta_A=Ff \delta_A$. 
We have
\[ fF\delta_A=f \delta_A=\delta_B\psi=F\delta_B\psi=F f\delta_A,\]
where the first and third equalities follow from $F\delta_A=\delta_A$ and $F\delta_B=\delta_B$ (see Definition \ref{def;logDA}(iii)), and the second and the last equalities follow from the commutativity of the outer square of \eqref{eq3.5;omegaUniversal}.
This proves the desired equality.

Finally, note that injectivity of the reduction map is immediate. Indeed, let $f, g\colon \widehat{\Omega}^*_{\mathcal{A}/\mathcal{R}}\to B^*$ be two morphisms in $\DAlogS$. inducing the same pair $(\overline{f}, \psi)$. In degree zero, $f^0,g^0\colon A\to B^0$ are $F$-homomorphisms with the same reduction, hence by \cite[Proposition 3.6.3]{BLM} again we obtain $f^0=g^0$. In degree 1, log differentials are generated by the elements $dx$ for $(x\in A)$ and $d\log l = \delta_A(l)$, $(l\in L)$. Since $f$ and $g$ are morphisms of cdga's compatible with $d$ and satisfy $f^1\delta_A = \delta_B \psi = g^1 \delta_A$ by condition (ii) of Definition \ref{def;logDA}, we get \[f^1(dx)=d_Bf^0(x)=d_Bg^0(x)=g^1(dx), \quad f^1(\delta_A(l))=\delta_B(\psi(l))=g^1(\delta_A(l)), \]
so that $f^1=g^1$. Since the algebras are generated by degree 1 elements, it is enough to conclude that $f=g$ in all degrees. The two maps then agree also on $p$-adic completion by continuity. 
\end{proof}

\begin{rmk}\label{rmk;WSatomegaIndep}
Notice that the proof of Proposition \ref{prop;omegaUniversal} is delicate.
If the $F$-homomorphism $f$ in \eqref{eq2;omegaUniversal} satisfies,
$\alpha_B\circ \psi= f\circ \alpha$ (in other words, if the left hand square in \eqref{eq3.5;omegaUniversal} is also commutative), the existence of a morphism $f\colon \omegaSs \sA \to B^*$ of cdga's over $A$ extending $f\colon A\to B^0$ is a formal consequence of the universal property of $\omegaSs \sA$. 
In our case, we need to prove \eqref{eq4;omegaUniversal} by an explicit computation.
\end{rmk}
 
\begin{cor}\label{cor;WSatomegaIndep}
Let $(\sA=(A,\alpha\colon L \to A),F)$ and $(\sA'=(A',\alpha'\colon L' \to A'),F')$ be two special $F$-pre-log algebras over $\sR$ such that the pre-log algebras $\sA$ and $\sA'$ are log-smooth of Cartier type over $\sR$ .
Let $\fb\colon A/p \simeq A'/p$ be an isomorphism of $R/p$-algebras and 
$\psi: L\simeq L'$ be an isomorphism of monoids such that the diagram
\[\xymatrix{
L \ar[r]^-\psi\ar[d]^{\alphab} & L'\ar[d]^{\alphab'}\\
A/p \ar[r]^-{\fb} & A'/p\\}.\]
 commutes. Let $\uomega^*_{\sA/\sR}$ and $\uomega^*_{\sA'/\sR}$ be the associated 
objects of $\DAlogSs$  provided by Proposition \ref{prop;omegaDAlog}.
Then, there exists a unique isomorphism in $\DAlogS$:
\[ (\fb,\psi)\colon \WSat( \widehat{{\uomega}}^*_{\sA/\sR}) \simeq \WSat( \widehat{\uomega}^*_{\sA'/\sR}) \]
which makes the following diagram commutative:
\[\xymatrix{
 \WSat( \widehat{\uomega}^*_{\sA/\sR})^0/V \ar[d]^\simeq \ar[r]^-{(\fb,\psi)} & \WSat( \widehat{\uomega}^*_{\sA'/\sR})^0/V\ar[d]^\simeq\\
 A/p\otimes_R \Rperf\ar[r]^-{\fb\otimes \id_{\Rperf}} & A'/p\otimes_R \Rperf\\}\] 
where the vertical isomorphisms come from Lemma \ref{lem;WSatomegaV}(2).
\end{cor}
\begin{proof}
This follows from Proposition \ref{prop;omegaUniversal} applied to the case
$(\fb,\psi)=(\id_{A/p},\id_L)$ using the isomorphism
\begin{equation} \label{eq:key_comparison_WSat}\Hom_{\DAlogS}(\widehat\uomega^*_{\sA/\sR}, \WSat( \uomega^*_{\sA'/\sR}))\simeq
\Hom_{\DAlogS}(\WSat( \widehat\uomega^*_{\sA/\sR}), \WSat( \uomega^*_{\sA'/\sR}))\end{equation}
given by the fact that $\WSat(-)$ is a left adjoint as in \eqref{eq:WSat_left_adjoint} to the inclusion of the full subcategory of strict log Dieudonn\'e algebras. 
Moreover, we have isomorphisms
\[ (\WSat(\uomega^*_{\sA'/\sR})^0/p)_{\rm red} \simeq (A'/p\otimes_R \Rperf)_{\red} \simeq (\WSat(\widehat{\uomega}^*_{\sA'/\sR})^0/p)_{\rm red}\] 
by Lemma \ref{lem;WSatomegaV}(2) so that we can use again Proposition \ref{prop;omegaUniversal}  to add a $p$-adic completion to the right hand term of the Hom groups in \eqref{eq:key_comparison_WSat}, and get the second part of the claim. 
\end{proof}

\begin{rmk}
Consider the case $A=A'$ and $L=L'$, but not necessarily $\alpha=\alpha'$ nor $F=F'$.  The following diagram of the underlying chain complexes may not commute
\[
\begin{tikzcd}
    \widehat{\uomega}^*_{\sA/\sR} \ar[d, "\rho"] \ar[rrd, "\rho'"] \\
    \WSat(\widehat{\uomega}^*_{\sA/\sR}) \ar[rr, "{(\id_{A/p},\id_L)}"'] && \WSat( \widehat{\uomega}^*_{\sA'/\sR})
\end{tikzcd}
\]
where $\rho$ and $\rho'$ are the natural map of underlying chain complexes induced by the identity on $\WSat(-)$.
\end{rmk}

\begin{thm}\label{thm;mathcaWomega}
Put $(\sR/p)_{\red}:=(R/\sqrt{pR},N)$
 and let $\DA_R^{\str}$ be the category of strict Dieudonn\'e algebras over $R$ from \cite{BLM}.
There is a functor
\[
\mathcal W\uomega^*_{-/\sR}\colon
\mathrm{lPoly}_{(\sR/p)_{\red}}
\longrightarrow
\DA_R^{\str}
\]
characterized as follows. For
$\sP_0=(P_0,L_P)\in\mathrm{lPoly}_{(\sR/p)_{\red}}$, let $\sP$ be its
standard $p$-complete polynomial $F$-lift over $\sR$, and put
\[
\mathbf W_{\sR}(\sP_0)
:=
\WSat\bigl(\widehat{\uomega}^*_{\sP/\sR}\bigr)
\in(\DAlogSs)^{\str}.
\]
Then $\mathcal W\uomega^*_{\sP_0/\sR}$ is the underlying strict
Dieudonn\'e algebra of $\mathbf W_{\sR}(\sP_0)$.

Let
\[
q_{\sP}\colon
\widehat{\uomega}^*_{\sP/\sR}
\longrightarrow
\mathbf W_{\sR}(\sP_0)
\]
be the  unit map, and let
\[
e_{\sP_0}\colon
\sP_0
\longrightarrow
\left(
 (\mathbf W_{\sR}(\sP_0)^0/p)_{\red},L_P
\right)
\]
be its reduction on degree zero. For a morphism
$u\colon\sP_0\to\sQ_0$, the morphism
$\mathbf W_{\sR}(u):\mathbf W_{\sR}(\sP_0)\to \mathbf W_{\sR}(\sQ_0) $ 
is the unique strict logarithmic Dieudonn\'e
algebra morphism satisfying
\begin{equation}\label{eq;esP0}
(\mathbf W_{\sR}(u)^0/p)_{\red}\circ e_{\sP_0}
=
e_{\sQ_0}\circ u.
\end{equation}
Precomposing with reduction gives, in particular, a functor on
$\mathrm{lPoly}_{\sR/p}$.
\end{thm}

\begin{proof}Proposition \ref{prop;omegaUniversal} 
     immediately implies that the functor $\sP \mapsto \WSat(\widehat{\uomega}^*_{\sP/\sR})$ on $\mathrm{lPoly}_{\sR}^F$ only depends functorially on the reduced part of the reduction modulo $p$ of $\sP$, where $\mathrm{lPoly}_{\sR}^F$  is the category of polynomial $F$-pre-log algebras over $\sR$  as in Remark \ref{rmk;Flogalgebra}(1). More precisely,
for a morphism $u\colon\sP_0\to\sQ_0$, apply
Proposition~\ref{prop;omegaUniversal} to the map of reduced pre-log
algebras $e_{\sQ_0}\circ u$. It gives a unique morphism
\[
\widehat{\uomega}^*_{\sP/\sR}
\longrightarrow
\mathbf W_{\sR}(\sQ_0).
\]
The adjunction \eqref{eq:WSat_left_adjoint} extends it uniquely to a
strict logarithmic Dieudonn\'e morphism
$\mathbf W_{\sR}(\sP_0)\to\mathbf W_{\sR}(\sQ_0)$. The identity and
composition laws follow from uniqueness. This also shows directly that
the construction is independent of the chosen standard presentation of
the polynomial lift.
\end{proof}
\begin{rmk}\label{rmk;multiplicativity} Let $\mathcal{D}(R)^{\wedge}_p$ be the $p$-complete $\infty$-category of $R$-modules and  let
$\mathrm{CAlg}(-)$ denote the $\infty$-category of commutative algebra objects with respect to the standard symmetric monoidal structure on  $\mathcal{D}(R)^{\wedge}_p$.
    Let
\[
U_{\rm alg}\colon
\DA_R^{\str}
\longrightarrow
\CAlg\bigl(\mathcal D(R)^\wedge_p\bigr)
\]
be the functor which forgets the Frobenius and regards the underlying
commutative differential graded $R$-algebra as a commutative algebra
object of the derived $p$-complete category.
The functor of
Theorem~\ref{thm;mathcaWomega} canonically lifts to a functor
\[
U_{\rm alg}\circ\mathcal W\uomega^*_{-/\sR}\colon
\mathrm{lPoly}_{\sR/p}
\longrightarrow
\CAlg\bigl(\mathcal D(R)^\wedge_p\bigr).
\]
Since $\CAlg\bigl(\mathcal D(R)^\wedge_p\bigr)$ admits sifted colimits, and the forgetful functor to $p$-complete $R$-modules preserves them,  sifted left Kan extension gives a canonical functor
\begin{equation}\label{eq;CAlg-animation-Womega}
L\mathcal W\uomega^*_{-/\sR}\colon
\Ani(\PreLog_{\sR/p})
\longrightarrow
\CAlg\bigl(\mathcal D(R)^\wedge_p\bigr).
\end{equation}
The underlying $\mathcal D(R)^\wedge_p$-valued functor is the
animation of $U\circ\mathcal W\uomega^*_{-/\sR}$, 
where $U$ is the composite of $U_{\rm alg}$ and the obvious forgetful functor $
\CAlg\bigl(\mathcal D(R)^\wedge_p\bigr) \to \mathcal D(R)^\wedge_p$. 
Indeed, a strict Dieudonn\'e algebra over $R$ is, in particular, a
commutative differential graded $R$-algebra, and a morphism of strict
Dieudonn\'e algebras preserves the multiplication and the unit. Its
underlying complex is derived $p$-complete. Hence the functor
$\mathcal W\uomega^*_{-/\sR}$ canonically factors through
$\CAlg(\mathcal D(R)^\wedge_p)$.  See \cite[Corollary~7.6.2, Proposition~7.6.3 and
Examples~7.6.5, 7.6.7]{BLM}.
\end{rmk}

\subsection{Change of log structures of the base}
Let $(\sA=(A,\alpha\colon L \to A),F)$ be a special $F$-pre-log algebras over $\sR$.
Proposition \ref{prop;omegaUniversal} shows that maps 
from $\widehat{\omega}^*_{\sA/\sR} $ to special log Dieudonné algebras are appropriately detected  on the reduced special fibre.
This allows us in particular to change the pre-log structure $\alpha$ on $A$ as follows. 

Consider another pre-log structure $\hal_R\colon N\to R$ with the same monoid $N$ satisfying \eqref{eq;SN} and put $\uhS=(R,\hal_R)$.
We assume 
\eq{eq:talpha}{
\alpha_R\equiv \hal_R \mod \sqrt{pR}.}
We also consider the log algebra $\utS=(R,\triv)$ with the trivial log structre.
Let $(\sA=(A,\alpha\colon L=L^o\oplus N \to A),F)$ be a $p$-torsion free special $F$-pre-log algebra over $\sR$ and $\iota\colon R\to A$ be the structure map.
Then,
\begin{align}\label{eq:special_lift} 
\widetilde{\sA}=(A,\hal\colon L\to A,F)\qaq
 \sA^{\triv}=(A,\alpha_{|L^o}\colon  {L^o } \to A,F) \end{align}
are $p$-torsion free special $F$-pre-log algebras over $\uhS$ and $\utS$ respectively, where
$\hal_{|L^o}=\alpha_{|L^o}$ and $\hal_{|N}=\iota\circ \hal_R$.
Note that the relative log de Rham complexes of these log algebras are canonically identified
\eq{eq;omegaSSS}{
\omega^*_{\sA/\sR}\simeq \omega^*_{\widetilde{\sA}/\uhS}\simeq 
\omega^*_{\sA^{\triv}/\utS}}
by the natural morphisms induced by the identity of $A$. Indeed, the pair consisting of the differential $\tilde{d}\colon A\to  \omega^1_{\widetilde{\sA}/\uhS}$ together with 
$\tilde{\delta}\colon L\to \omega^1_{\widetilde{\sA}/\uhS}$  
gives in particular a log derivation of $\sA$ over $\sR$, thus induces a unique  map $\omega^*_{\sA/\sR}\to \omega^*_{\uhA/\uhS}$. Exchanging the roles of $\widetilde{\sA}/\uhS$ and $\sA/\sR$, we get the map in the inverse map in the opposite direction. Similar argument works for $\sA^{\triv}/\utS$. We stress that \eqref{eq;omegaSSS} holds only at the level of underlying cdga's with Frobenius (that is, Dieudonné algebras), not as log-objects (as log Dieudonné algebras, they live in different categories). 

\begin{const}\label{hat_constr} Let $\sB^*=(B^*,d,F_B,\alpha_B\colon L_B=L_B^o\oplus N \to B^0,\delta_B\colon {L_B}\to B^1)$
be a $p$-torsion free special log Dieudonn\'e algebra over $\sR$ as defined in Definition \ref{def;logDA} and $\iota_B:R\to B^0$ be the structure map. Then,
\[\uhB^*=(B^*,d,F_B,\hal_B\colon L_B\to B^0,\delta_B)
\;\text{ in } \DA^{\log}_{\widetilde{\sR}} \]
and
\[ \utBs=(B^*,d,F_B,(\alpha_B)_{|L_B^o}\colon L_B^o\to B^0,\delta_B)\;\text{ in } \DA^{\log}_{\sR^{\triv}} \]
are $p$-torsion-free special log Dieudonn\'e algebras over
$\uhS$ and $\utS$, respectively, where
$(\hal_B)_{|L_B^o}=(\alpha_B)_{|L^o}$ and $(\hal_B)_{|N}=\iota_B\circ \hal_R$.
\end{const}

Let $\fb\colon   \redd {A/p}\to\redd{B^0/p}$ and $\psi:L\to L_B$ be a pair as in Proposition \ref{prop;omegaUniversal} such that the diagram \eqref{eq1;omegaUniversal}
commutes. Note that the commutativity also holds for 
$(\widetilde{\alpha},\widetilde{\alpha}_B)$ and $(\alpha_{|L^o},(\alpha_B)_{|L^o})$ thanks to \eqref{eq:talpha}.
Thus, assuming that the underlying Dieudonn\'e complex $(B^*,d,F_B)$ of $\sB^*$ is strict, the proposition implies that we have the induced maps
\begin{align*}
    f\colon  \uomegaSshat {\sA} \to \sB^*\; &\text{ in }\DAlogS,\\ 
    \hf\colon  \widehat{\uomega}^*_{\uhA/\uhS}\to \uhB^*\; &\text{ in }\DAlog_{\uhS},\\
    \tf\colon \widehat{\uomega}^*_{\utA/\utS}\to \utBs\; &\text{ in }\DAlog_{\utS}.
\end{align*}
 Finally, let $\uf\colon  \omegaSshat {\sA} \to B^*$, $\uhf\colon  \widehat{\omega}^*_{\uhA/\uhS}\to B^*$ and
$\utf\colon  \widehat{\omega}^*_{\utA/\utS}\to B^*$ be the induced maps in $\DA_R$ of the underlying Dieudonné algebras over $R$
(noting that the underlying Dieudonné algebra of $\sB^*$, of $\uhB^*$ and of $\utBs$ is always $(B^*,d,F_B)$).  

 \begin{lem}\label{lem;IndepBase}Following the notation above, let $\Gamma^*_{A/R}$ be the common Dieudonné algebra over $R$ underlying the log Dieudonné algebras along the identification \eqref{eq;omegaSSS}. Then, for any $p$-torsion free special log Dieudonné algebra $\sB^*$ with the underlying Dieudonné algebra $(B^*,d,F_B)$, we have an identification
$\uf=\uhf=\utf$ as elements of $\Hom_{\DA_R} ( \Gamma^*_{A/R},  B^*) $.
\end{lem}
\begin{proof}
This follows immediately from the construction of $f$, $\hf$ and $\tf$.
\end{proof}

\begin{cor}\label{main_cor_WOmega_agrees}
Let $\sR=(R,\alpha_R\colon N\to R)$ and
$\widetilde{\sR}=(R,\widetilde{\alpha}_R\colon N\to R)$ be two
pre-log rings satisfying \eqref{eq;SN} and \eqref{eq:talpha}. Thus $(\sR/p)_{\rm red}
=
(\widetilde{\sR}/p)_{\rm red}$. 
Then there is a canonical natural isomorphism
\begin{equation}\label{eq:kappa-polynomial-base-log}
\varkappa\colon
\mathcal W\uomega^*_{-/\sR}
\xrightarrow{\ \sim\ }
\mathcal W\uomega^*_{-/\widetilde{\sR}}, \qquad \text{in}\qquad \Fun\left(
\mathrm{lPoly}_{(\sR/p)_{\rm red}},
\DA_R^{\str}
\right)
\end{equation}
More precisely, let
$\sP_0\in\mathrm{lPoly}_{(\sR/p)_{\rm red}}$, let $\sP$ be its
standard $p$-complete polynomial $F$-lift over $\sR$, and let
$\widetilde{\sP}$ be the corresponding lift over
$\widetilde{\sR}$. Write
\[
j_{\sP}\colon
\widehat\omega^*_{\sP/\sR}
\xrightarrow{\ \sim\ }
\widehat\omega^*_{\widetilde{\sP}/\widetilde{\sR}}
\]
for the completed version of the canonical identification
\eqref{eq;omegaSSS}. If
\[
q_{\sP}\colon
\widehat{\uomega}^*_{\sP/\sR}
\longrightarrow
\mathbf W_{\sR}(\sP_0),
\qquad
q_{\widetilde{\sP}}\colon
\widehat{\uomega}^*_{\widetilde{\sP}/\widetilde{\sR}}
\longrightarrow
\mathbf W_{\widetilde{\sR}}(\sP_0)
\]
are the unit maps, then
\begin{equation}\label{eq:kappa-unit-normalization}
\varkappa(\sP_0)\circ q_{\sP}
=
q_{\widetilde{\sP}}\circ j_{\sP}
\end{equation}
after forgetting the log structures.
\end{cor}
\begin{proof}
Let $\sP_0$ be a polynomial object as in the statement. By
Construction~\ref{hat_constr}, we may regard
$\mathbf W_{\sR}(\sP_0)$ as a strict logarithmic Dieudonn\'e algebra
over $\widetilde{\sR}$ by changing only the value of the base summand
$N$. We denote this object by
$\widetilde{\mathbf W_{\sR}(\sP_0)}$.
The identification \eqref{eq;omegaSSS} gives an isomorphism of
logarithmic Dieudonn\'e algebras over $\widetilde{\sR}$
\[
\widetilde{
 \widehat{\uomega}^*_{\sP/\sR}
}
\xrightarrow{\ \sim\ }
\widehat{\uomega}^*_{\widetilde{\sP}/\widetilde{\sR}}.
\]
Since saturation and strictification leave the monoid unchanged and
are functorial on the underlying Dieudonn\'e algebra, this induces an
isomorphism
\[
\widetilde{\mathbf W_{\sR}(\sP_0)}
\xrightarrow{\ \sim\ }
\mathbf W_{\widetilde{\sR}}(\sP_0)
\]
of strict logarithmic Dieudonn\'e algebras over
$\widetilde{\sR}$. We define $\varkappa(\sP_0)$ to be its underlying
isomorphism in $\DA_R^{\str}$. Naturality of the unit of the adjunction
\eqref{eq:WSat_left_adjoint} gives
\eqref{eq:kappa-unit-normalization}.
It remains to check functoriality. Let
\[
u\colon\sP_0\longrightarrow\sQ_0
\]
be a morphism of polynomial objects. Consider the two composites
\[
\varkappa(\sQ_0)\circ\mathbf W_{\sR}(u),
\;
\mathbf W_{\widetilde{\sR}}(u)\circ\varkappa(\sP_0)\; :\;
 \mathbf W_{\sR}(\sP_0) \to \mathbf W_{\widetilde{\sR}}(\sQ_0) .
\]
By the adjunction \eqref{eq:WSat_left_adjoint}, it is enough to compare
their precompositions with 
 
$q_{\sP} :\widehat{\uomega}^*_{\sP/\sR} \to\mathbf W_{\sR}(\sP_0)  $.
By Theorem~\ref{thm;mathcaWomega}, the latter morphisms
are characterized by their respective reduced morphisms
$\sP_0\to (\mathbf W_{\sR}(\sQ_0)^0/p)_{\red}$.
By Construction~\ref{hat_constr}, we have a canonical identification
$(\mathbf W_{\sR}(\sQ_0)^0/p)_{\red} =(\mathbf W_{\widetilde{\sR}}(\sQ_0)^0/p)_{\red}$ through which the reduced morphisms are identified with  
$e_{\sQ_0}\circ u.$ and $(\mathbf W_{\sR}(u)^0/p)_{\red}\circ e_{\sP_0}$ respectively.
By \eqref{eq;esP0}, these are equal so that the two displayed composites agree,
which proves naturality. 
\end{proof}
\begin{rmk}
    More generally, for three pre-log structures
\[
\sR_i=(R,\alpha_i\colon N\to R),
\qquad i=0,1,2,
\]
with the same reduced special fibre, the resulting comparisons satisfy
the strict identities
\[
\varkappa_{ii}=\id,
\qquad
\varkappa_{jk}\circ\varkappa_{ij}
=
\varkappa_{ik}
\]
in the $1$-category $\DA_R^{\str}$. Indeed, for three base log structures, the canonical identifications of the
completed relative log de Rham complexes satisfy
\[
j_{ii,\sP}=\id,
\qquad
j_{jk,\sP}\circ j_{ij,\sP}=j_{ik,\sP}.
\]
The corresponding identities for the $\varkappa_{ij}$ follow from
\eqref{eq:kappa-unit-normalization} and the uniqueness in the
adjunction \eqref{eq:WSat_left_adjoint}.
\end{rmk}

 \begin{rmk}\label{rmk;main_cor_WOmega_agrees}
 In view of Remark \ref{rmk;multiplicativity}, the animation of the functor $\mathcal{W}\uomega_{-/\sR}^*$ from Theorem \ref{thm;mathcaWomega} induces a functor 
\begin{equation}\label{eq;LWOmegaAni}
L\mathcal{W}\uomega^*_{-/\sR}\colon \mathrm{Ani} (\mathrm{PreLog}_{\sR/p})\to 
\mathrm{CAlg}(\mathcal{D}(R)^{\wedge}_p).
\end{equation}

In view of \cite[7.3.4 and 7.4.7]{BLM}, Corollary \ref{main_cor_WOmega_agrees} implies that there is a natural equivalence
\begin{equation}\label{eq;LWOmegaAnicomparison} 
\varkappa: L\mathcal{W}\uomega^*_{-/\sR} \simeq L\mathcal{W}\uomega^*_{-/\widetilde{\sR}}
\;\text{ in } \Fun(\mathrm{Ani} (\mathrm{PreLog}_{\sR/p}),\mathrm{CAlg}(\mathcal{D}(R)^{\wedge}_p)).
\end{equation}
 Let $h: R\to R'$ be a ring map with a Frobenius lift $F_{R'}$ on $R'$ compatible with $F_R$. Let $\sR'=(R',\alpha_{R'}: N\to R')$ and 
     $\widetilde{\sR}'=(R',\widetilde{\alpha}_{R'}:N\to R')$ be the pre-log rings with the induced pre-log structure from those of $\sR$ and $\widetilde{\sR}$ respectively.
       Then, $\alpha_{R'}$ and $\widetilde{\alpha}_{R'}$ satisfy \eqref{eq;SN} so that we have a natural equivalence
    \[  \varkappa': L\mathcal{W}\uomega^*_{-/\sR'} \simeq L\mathcal{W}\uomega^*_{-/\widetilde{\sR}'}
\;\text{ in } \Fun(\mathrm{Ani} (\mathrm{PreLog}_{\sR'/p}),\mathrm{CAlg}(\mathcal{D}(R)^{\wedge}_p))
       \]
       such that $h^*(\varkappa)=h_*(\varkappa')$, where        
  \[ h^*:  \Fun(\mathrm{Ani} (\mathrm{PreLog}_{\sR/p}),\mathrm{CAlg}(\mathcal{D}(R)^{\wedge}_p))\to
  \Fun(\mathrm{Ani} (\mathrm{PreLog}_{\sR/p}),\mathrm{CAlg}(\mathcal{D}(R)^{\wedge}_p)),\]
  \[ h_*:  \Fun(\mathrm{Ani} (\mathrm{PreLog}_{\sR'/p}),\mathrm{CAlg}(\mathcal{D}(R)^{\wedge}_p))\to
  \Fun(\mathrm{Ani} (\mathrm{PreLog}_{\sR/p}),\mathrm{CAlg}(\mathcal{D}(R)^{\wedge}_p))\]
  are the functors induced by the base change along $h$.
  \end{rmk}
     
      \section{Comparison with log crystalline complexes}
      
      In this section we review the construction of log crystalline cohomology
      (or Hyodo--Kato cohomology) over various bases
      using the saturated log de Rham--Witt complex. A version of these results is in
      \cite[Section~5]{Ya}, but we record the details
      here since some necessary verifications were omitted, and since only the log
      structure $[\alpha]\colon N\to W(k)$ given by the
      Teichm\"uller lift of a log point $N\to k$
      on $W(k)$ was considered in \cite{Ya}. Note that our argument is  slightly different from that of \cite{Ya}. 
      \medbreak
      
      Throughout this section all monoids and log structures are fine, and pre-log
      rings are understood through their associated fine log schemes. For a
      $p$-adic fine log PD-base $\sR=(R,N)$, we use the notation
      \[
      R\Gamma_{\rm crys}(-/\sR)
      :=R\varprojlim_n
      R\Gamma_{\rm crys}\bigl(-/(\sR/p^n,pR/p^n,\gamma_{\rm can})\bigr).
      \]
      We write $\widehat\otimes_R^{\mathbb L}$ for the derived $p$-completed tensor
      product.
      
      \subsection{The crystalline--de Rham comparison}
      
      \begin{const}\label{const;Kato-comparison}
      	Let $\sR=(R,\alpha_R\colon N\to R)$ be as in \eqref{eq;SN}, and let
      	$(\sA=(A,\alpha\colon L\to A),F_A)$ be a $p$-torsion-free special
      	$F$-pre-log algebra over $\sR$. Assume that $\sA$ is log smooth over $\sR$,
      	and put
      	\[
      	\sA_0=(A/p,L).
      	\]
      	Applying \cite[Theorem~6.4]{katolog} over $R/p^n$ and passing to the derived
      	inverse limit gives a natural multiplicative equivalence
      	\begin{equation}\label{eq;omegalogcryscomparison}
      		c_{\sA}\colon
      		R\Gamma_{\rm crys}(\sA_0/\sR)
      		\xrightarrow{\ \sim\ }
      		\widehat\omega^*_{\sA/\sR}
      		\qquad\text{in    } \;
            \mathrm{CAlg}(\mathcal{D}(R)^\wedge_p). 
      	\end{equation}
      	At finite level this is given by the equivalence of complexes of sheaves
      	\begin{equation}\label{eq;omegalogcryscomparison2}
      		R(u^{\log}_{\sA_0/(\sR/p^n)})_*
      		\mathcal O_{\sA_0/(\sR/p^n),\rm crys}
      		\simeq
      		\omega^*_{(\sA/p^n)/(\sR/p^n)}
      	\end{equation}
      	on $(\Spec(A/p))_{\acute et}$ 
        from \cite[Theorem~6.4]{katolog}.
      \end{const}
      
      We record explicitly the functoriality with respect to a morphism of special
      fibres.
      
      \begin{lem}\label{lem;Kato_fixed_base_lift_independence}
      	Let $\underline T$ be a $p$-adic fine log
      	PD-base, and let
      	$u\colon\underline P_0\to\underline Q_0$ be a morphism of affine fine
      	pre-log algebras over  the special fibre $\underline T_0$ of $\underline T$.
        Let $\underline P$ and
      	$\underline Q$ be affine log-smooth liftings
        of $\underline P_0$ and $\underline Q_0$
        over $\underline T$ respectively,
        and let
      	\[
      	\widetilde u_1,\widetilde u_2\colon
      	\underline P\longrightarrow\underline Q
      	\]
      	be two liftings of $u$ over $\underline T$. Then
      	\[
      	\widehat\omega^*(\widetilde u_1)
      	\simeq
      	\widehat\omega^*(\widetilde u_2)
      	\quad\text{in }\mathcal D(T)^\wedge_p.
      	\]
      	More precisely, if $c_{\underline P}$ and $c_{\underline Q}$ denote Kato's
      	comparison equivalences \eqref{eq;omegalogcryscomparison}, 
        then, 
      	\begin{equation}\label{eq;Kato-lift-functoriality}
      		\widehat\omega^*(\widetilde u_i)\circ c_{\underline P}
      		\simeq
      		c_{\underline Q}\circ u^*\quad\text{ for }
           i=1,2.
      	\end{equation}
      	These identifications are compatible with composition.
      \end{lem}
      
      \begin{proof}
      	It is enough to work modulo $p^n$ and then pass to the derived inverse limit.
      	The pullback $u^*$ is the site-theoretic pullback supplied by the functoriality
      	of the logarithmic crystalline topos in \cite[(5.9)]{katolog}. The comparison
        \eqref{eq;omegalogcryscomparison2} 
      	constructed by the de Rham resolution in its
      	proof \cite[(6.9)]{katolog}, identifies the map induced by any chosen lifting
      	$\widetilde u_i$ with this site-theoretic pullback. Since the site-theoretic pullback depends
      	only on $u$, the two maps agree in the derived category.
    Compatibility with composition follows from the
      	transitivity of the crystalline stratification used in the construction of
      	Kato's comparison.
      \end{proof}
      
      In particular, for a morphism $(f,\psi)\colon\sA\to\sB$ of log-smooth
      special $F$-pre-log algebras over $\sR$, the diagram
      \begin{equation}\label{omega-logcrysCD0}
      	\xymatrix{
      		R\Gamma_{\rm crys}(\sA_0/\sR)\ar[d]^-{(\bar f,\psi)^*}\ar[r]^-{c_{\sA}}
      		&\widehat\omega^*_{\sA/\sR}\ar[d]^{(f,\psi)^*}\\
      		R\Gamma_{\rm crys}(\sB_0/\sR)\ar[r]^-{c_{\sB}}
      		&\widehat\omega^*_{\sB/\sR}
      	}
      \end{equation}
      commutes in $\mathcal D(R)^\wedge_p$.
      
      \subsection{Divided Frobenius and strict Dieudonn\'e algebras}
      We begin with the following  Lemma.

      \begin{lem}\label{lem;log-dR-Frobenius-base-change}Let
      	$(\sA=(A,\alpha\colon L\to A),F_A)$ be a $p$-torsion-free special
      	$F$-pre-log algebra over $\sR$ such that   $\sA$ is log smooth over $\sR$. 
      	There is a canonical equivalence of completed cdga's
      	\[
      	\widehat\omega^*_{\sA/\sR}
      	\widehat\otimes^{\mathbb L}_{R,F_R}R
      	\xrightarrow{\ \sim\ }
      	\widehat\omega^*_{\sAFt 1/\sR},
      	\]
      	where $\sAFt1=(A^{(1)}=A\otimes_{R,F_R}R,\alpha^{(1)}:L\to A^{(1)},F_{\AFt 1}=F_A\otimes F_R)$
      	is the $F$-pre-log algebra over $\sR$ introduced in Remark~\ref{rmk;Flogalgebra-Frobenius}.
      \end{lem}
      
      \begin{proof}
First, note that the assumption on $\sA$ implies that $\sAFt1$ is log smooth over $\sR$. 
      	Let
\[
\iota\colon A\longrightarrow A^{(1)}=A\otimes_{R,F_R}R,
\qquad a\longmapsto a\otimes1.
\]
Write $(d^{(1)},d^{(1)}\log)$ for the standard 
relative logarithmic
derivation of $\sAFt1$ over $\sR$ as in Proposition \ref{prop;omegaDAlog},
 and consider the maps
\[
D\colon A\longrightarrow\omega^1_{\sAFt1/\sR},
\qquad
\Delta\colon L\longrightarrow\omega^1_{\sAFt1/\sR}
\]
defined by
\[
D=d^{(1)}\circ\iota,
\qquad
\Delta(l)=d^{(1)}\log l.
\]
Since $\iota(r)=1\otimes F_R(r)$ for $r\in R$, one has $D(r)=0$, and
$\Delta(n)=0$ for $n\in N$. If $l\in L^o$, then
$\alpha^{(1)}(l)=\iota(\alpha(l))$, hence
\[
D(\alpha(l))={d^{(1)}(\alpha^{(1)}(l))=\alpha^{(1)}(l)d^{(1)}\log l =} \iota(\alpha(l))\Delta(l).
\]
For $n\in N$, both sides vanish, since $\iota(\alpha(n))$ comes from
the base and $\Delta(n)=0$. Thus, using the decomposition $L=L^o\oplus N$,
$(D,\Delta)$ is a logarithmic derivation of $\sA$ over $\sR$, where
$\omega^1_{\sAFt1/\sR}$ is viewed as an $A$-module via $\iota$.
The
universal property of logarithmic differentials therefore gives an $F_R$-semilinear morphism of
cdga's and, after derived $p$-completion, a canonical map
\[
\beta\colon
\widehat\omega^*_{\sA/\sR}
\widehat\otimes^{\mathbb L}_{R,F_R}R
\longrightarrow
\widehat\omega^*_{\sAFt1/\sR}.
\]
Now note that, modulo $p$, the pair
\[
F_{\sR/p}=(\overline F_R,\;{\text{$p$ on $N$} })\colon\sR/p\longrightarrow\sR/p
\]
is a morphism of pre-log rings, and the square
\[
\begin{tikzcd}
\sR/p \arrow[r,"F_{\sR/p}"] \arrow[d]
&
\sR/p \arrow[d]
\\
\sA/p \arrow[r,"{(\overline\iota,\tau)}"']
&
\sAFt1/p
\end{tikzcd}
\]
is cocartesian, where $\tau$ is the identity on $L^o$ and multiplication
by $p$ on $N$. The reduction of $\beta$ is the corresponding
base-change map: this is clear on $L^o$, while on $N$ both logarithmic
differentials vanish. Base change for logarithmic differentials and
exterior powers therefore shows that $\beta$ is an equivalence modulo
$p$. Since its source and target are derived $p$-complete, derived
Nakayama implies that $\beta$ is an equivalence.
      \end{proof}
       Assume now that $\sA_0 = \sA/p$ is moreover of Cartier type over $\sR/p$. Then
      $\sA\to\sR$ is integral (this is part of the condition of being of Cartier type) and log smooth. By
      Remark~\ref{rmk:Cartier_smooth_flat}, $A$ and the terms of
      $\omega^*_{\sA/\sR}$ are $p$-torsion-free.
      Assume also from now on that $F_R$ is an automorphism. 
      Note now that pullback  by the relative Frobenius from Remark \ref{rmk;Flogalgebra-Frobenius} 
      satisfies
      \[
      F_{\sA/\sR}^*\bigl(\omega^i_{\sAFt1/\sR}\bigr)
      \subseteq p^i\omega^i_{\sA/\sR}.
      \]
      Indeed, $dF_A(x)$ is divisible by $p$ for $x\in A$,
      $F_{\sA/\sR}^*(d\log l)=p\,d\log l$ for $l\in L^o$, and the relative
      logarithmic differentials of elements of $N$ vanish. Dividing by $p^i$ in
      degree $i$ therefore gives a natural map
      \[
      \widetilde F_{\sA/\sR}\colon
      \widehat\omega^*_{\sAFt1/\sR}
      \longrightarrow
      L\eta_p\bigl(\widehat\omega^*_{\sA/\sR}\bigr).
      \]
      Modulo $p$, this is the inverse logarithmic Cartier map, after the usual
      identification of $\eta_p(C)/p$ with $H^*(C/p)$ equipped with its Bockstein differential
      (cf. \cite[2.4.5]{BLM}.
      Since $\sA_0$ is of Cartier type over $\sR/p$, we obtain using Lemma \ref{lem;log-dR-Frobenius-base-change} a
      natural equivalence
      \begin{equation}\label{eq;crysfrobomega}
      	\widehat\omega^*_{\sA/\sR}
      	\widehat\otimes^{\mathbb L}_{R,F_R}R
      	\xrightarrow{\ \sim\ }
      	L\eta_p\bigl(\widehat\omega^*_{\sA/\sR}\bigr)
      	\qquad\text{in }\;
         \mathrm{CAlg}(\mathcal{D}(R)^\wedge_p). 
             \end{equation}
      
      Transporting \eqref{eq;crysfrobomega} through
      \eqref{eq;omegalogcryscomparison} equips
      $R\Gamma_{\rm crys}(\sA_0/\sR)$ with a canonical
      $(F_{R})_*\circ L\eta_p$-fixed-point structure. We write
      \[
      \mathbf D_{\rm crys}(\sA_0/\sR)\in\DA_R^{\str}
      \]
      for the corresponding strict Dieudonn\'e algebra under
      \cite[Corollary~7.4.9, Remark 7.6.4 and Examples~7.6.5, 7.6.7]{BLM}. Thus, if
      $U\colon\DA_R^{\str}\to \mathrm{CAlg}(\mathcal{D}(R)^\wedge_p). $ is the forgetful functor, there is an equivalence
      in $\mathrm{CAlg}(\mathcal{D}(R)^\wedge_p)$.
      \[
      U\mathbf D_{\rm crys}(\sA_0/\sR)
      \simeq R\Gamma_{\rm crys}(\sA_0/\sR).
      \]
      
      Let
       \begin{equation}\label{eq;comegaUWSat}
      q_{\sA}\colon
      \widehat\omega^*_{\sA/\sR}
      \longrightarrow
      U\WSat(\widehat\omega^*_{\sA/\sR})
      \end{equation}
      be the natural map. The Cartier criterion and
      \cite[Corollary~2.8.5]{BLM} show that $q_{\sA}$ is an equivalence.
      Compatibility of Kato's comparison with divided relative Frobenius gives an
      isomorphism
      \begin{equation}\label{eq;Theta-crys}
      	\Theta_{\sA}\colon
      	\WSat(\widehat\omega^*_{\sA/\sR})
      	\xrightarrow{\ \sim\ }
      	\mathbf D_{\rm crys}(\sA_0/\sR)
      	\qquad\text{in }\DA_R^{\str},
      \end{equation}
      normalized by the canonical homotopy
      \begin{equation}\label{eq;Theta-normalization}
      	U(\Theta_{\sA})\circ q_{\sA}\circ c_{\sA}
      	\simeq\id_{R\Gamma_{\rm crys}(\sA_0/\sR)}.
      \end{equation}

      \begin{rmk}
      	The object $\mathbf D_{\rm crys}(\sA_0/\sR)$ is defined intrinsically by $R\Gamma_{\rm crys}(\sA_0/\sR)$ as a fixed-point object 
        under the equivalence of
      	\cite[Corollary~7.4.9, Remark 7.6.4 and Examples~7.6.5, 7.6.7]{BLM}. In particular, it is
      	independent of any choice of a chain-level representative.
      \end{rmk}

      \begin{thm}\label{cor;thm;omega-logcrys}
      	Let $(\sR=(R,N),F_R)$ be as in \eqref{eq;SN}, and assume that $F_R$ is an
      	automorphism. Then there is a natural equivalence
      	\[
      	U\circ\mathcal W\uomega^*_{-/\sR}
      	\xrightarrow{\ \sim\ }
      	R\Gamma_{\rm crys}(-/\sR)
      	\quad\text{in }
      	\Fun\bigl(\mathrm{lPoly}_{\sR/p},
        \mathrm{CAlg}(\mathcal{D}(R)^\wedge_p) 
        \bigr),
      	\]
      where $U\colon\DA_R^{\str}\to\mathrm{CAlg}(\mathcal{D}(R)^\wedge_p). $ is the forgetful functor.
      \end{thm}
      
      \begin{proof}
      	Since $F_R$ is an automorphism, the absolute Frobenius on $R/p$ is an
      	automorphism; in particular, $R/p$ is perfect and reduced. For
      	$\mathcal P_0=(P_0, L_P)\in\mathrm{lPoly}_{\sR/p}$, choose its standard special
      	$F$-lift $\mathcal P$ over $\sR$. 
        The isomorphism
      	\eqref{eq;Theta-crys} and Theorem~\ref{thm;mathcaWomega} give an 
      	objectwise equivalence
        \[
U\mathbf W_{\sR}(\mathcal P_0)
\xrightarrow[\sim]{\,U(\Theta_{\mathcal P})\,}
U\mathbf D_{\rm crys}(\mathcal P_0/\sR)
\simeq
R\Gamma_{\rm crys}(\mathcal P_0/\sR).
\] 
We now verify naturality. Let
      	$u\colon\mathcal P_0\to\mathcal{Q}_0$ be a morphism of polynomial
      	pre-log algebras.
        The relative Frobenius maps fit into a commutative
diagram of pre-log algebras over $\sR/p$
\[
\begin{tikzcd}
\mathcal P_0^{(1)}
\arrow[r,"u^{(1)}"]
\arrow[d,"F_{\mathcal P_0/(\sR/p)}"']
&
\mathcal Q_0^{(1)}
\arrow[d,"F_{\mathcal Q_0/(\sR/p)}"]
\\
\mathcal P_0
\arrow[r,"u"']
&
\mathcal Q_0 .
\end{tikzcd}
\]
Functoriality of the logarithmic crystalline topos
\cite[(5.9)]{katolog} therefore gives a commutative square of
crystalline pullback maps. Under the de Rham resolutions constructed
in \cite[(6.9)]{katolog}, its vertical maps identify with the divided
relative Frobenius maps defining \eqref{eq;crysfrobomega}. Hence, the compatibility of crystalline pullback with
divided relative Frobenius gives a morphism
\[
\mathbf D_{\rm crys}(u)\colon
\mathbf D_{\rm crys}(\mathcal P_0/\sR)
\longrightarrow
\mathbf D_{\rm crys}(\mathcal Q_0/\sR)
\]
in $\DA_R^{\str}$ whose underlying morphism (after applying $U$) is
$u_{\rm crys}^*$. Set
\[
\Lambda(u):=
\Theta_{\mathcal Q}^{-1}
\circ\mathbf D_{\rm crys}(u)
\circ\Theta_{\mathcal P}\colon
\mathbf W_{\sR}(\mathcal P_0)
\longrightarrow
\mathbf W_{\sR}(\mathcal Q_0).
\]
    Now   choose a pre-log lift
      	$\widetilde u\colon\mathcal P\to\mathcal Q$ over $\sR$, with underlying monoid map $\psi\colon L_P\to L_Q$. The functoriality in
Lemma~\ref{lem;Kato_fixed_base_lift_independence}, together with the
divided-Frobenius compatibility used in the construction of
$\Theta_{\mathcal P}$ and $\Theta_{\mathcal Q}$, shows that
$\Lambda(u)\circ q_{\mathcal P}$ (with the monoid map
$\psi$) is a morphism of logarithmic Dieudonn\'e algebras (the compatibility
with the logarithmic derivations follows from
$\widehat\omega^*(\widetilde u)(d\log l)=d\log\psi(l)$ for $ l\in L_P$). 
After forgetting the log and the Frobenius structures, this is
given by the following canonical equivalence
\[
\begin{aligned}
U(\Lambda(u))\circ q_{\mathcal P}
&\simeq
U(\Theta_{\mathcal Q})^{-1}
 \circ u_{\rm crys}^*
 \circ c_{\mathcal P}^{-1} \\
&\simeq
U(\Theta_{\mathcal Q})^{-1}
 \circ c_{\mathcal Q}^{-1}
 \circ\widehat\omega^*(\widetilde u) \\
&\simeq
q_{\mathcal Q}\circ\widehat\omega^*(\widetilde u),
\end{aligned}
\]
where the first and last equivalences use
\eqref{eq;Theta-normalization}, and the middle one uses
\eqref{eq;Kato-lift-functoriality}.
By Proposition~\ref{prop;omegaUniversal},
$\Lambda(u)\circ q_{\mathcal P}$ is therefore the unique logarithmic
Dieudonn\'e morphism classified by
$e_{\mathcal Q_0}\circ u$, where $e_{\mathcal Q_0}$ is as in Theorem~\ref{thm;mathcaWomega}.
By the construction in
Theorem~\ref{thm;mathcaWomega}, the same is true for $\mathbf W_{\sR}(u)\circ q_{\mathcal P}$, hence  $\Lambda(u)\circ q_{\mathcal P}
=
\mathbf W_{\sR}(u)\circ q_{\mathcal P}.$
Since the logarithmic structure on
$\mathbf W_{\sR}(\mathcal P_0)$ is induced by
$q_{\mathcal P}$, the adjunction
\eqref{eq:WSat_left_adjoint} gives $\Lambda(u)=\mathbf W_{\sR}(u)$. 
Equivalently,
\[
\mathbf D_{\rm crys}(u)\circ\Theta_{\mathcal P}
=
\Theta_{\mathcal Q}\circ\mathbf W_{\sR}(u),
\]
which proves naturality. All the maps involved are compatible with multiplication, so
the resulting equivalence takes values in
$\mathrm{CAlg}(\mathcal D(R)^\wedge_p)$. 
      \end{proof}



      \subsection{Changing the base log structure when Frobenius is invertible}

      Let
      \[
      \sR_i=(R,\alpha_i\colon N\to R),\qquad i=0,1,2,
      \]
      be pre-log structures satisfying \eqref{eq;SN} with respect to the same
      decomposition $N=N_1\oplus N_2$. In this subsection we assume that $F_R$ is
      an automorphism and that the reductions of the $\alpha_i$ modulo $p$ agree.
      Write $\sR_0^{\rm sp}=(R/p,N)$ for their common special fibre.
      We can translate our main result, stated for 
      the animation $L\mathcal{W}\Omega_{-/\mathcal{R}}^*$ of the functor $\mathcal{W}\Omega_{-/\mathcal{R}}^*$ from Theorem~\ref{thm;mathcaWomega},       
      in terms of crystalline complexes as follows. 
      \begin{cor}\label{cor:main_cor}
      	For every pair $i,j$ there is a canonical natural equivalence
      	\begin{equation}\label{eq:automorphism-transfer}
      		\Phi_{ij}\colon
      		dR\Gamma_{\rm crys}(-/\sR_i)
      		\xrightarrow{\ \sim\ }
      		dR\Gamma_{\rm crys}(-/\sR_j)
      	\end{equation}
      	in
      	\[
      	\Fun\bigl(\Ani(\PreLog_{\sR_0^{\rm sp}}), \mathrm{CAlg}(\mathcal{D}(R)^\wedge_p). \bigr).
      	\]
      	These equivalences satisfy
      	\[
      	\Phi_{ii}=\id,
      	\qquad
      	\Phi_{jk}\circ\Phi_{ij}\simeq\Phi_{ik},
      	\]
      	with the resulting higher cocycle compatibilities.
      	
      	On a polynomial object $\underline P_0$
        over $\sR_0^{\rm sp}$, 
        let $\underline P_i$ be the standard
      	lift over $\sR_i$, let
      	\[
      	c_{i,\underline P}\colon
      	R\Gamma_{\rm crys}(\underline P_0/\sR_i)
      	\xrightarrow{\sim}
      	\widehat\omega^*_{\underline P_i/\sR_i}
      	\]
      	be Kato's comparison \eqref{eq;omegalogcryscomparison}, and let
      	\[
      	j_{ij,\underline P}\colon
      	\widehat\omega^*_{\underline P_i/\sR_i}
      	\xrightarrow{\sim}
      	\widehat\omega^*_{\underline P_j/\sR_j}
      	\]
      	be the canonical identification of the underlying completed Dieudonn\'e
      	algebras as in \eqref{eq;omegaSSS}. 
        Then
      	\begin{equation}\label{eq:automorphism-transfer-Kato-formula}
      		\Phi_{ij}(\underline P_0)
      		\simeq
      		c_{j,\underline P}^{-1}\circ
      		j_{ij,\underline P}\circ
      		c_{i,\underline P}.
      	\end{equation}
      	
      	The construction has the following properties:
      	\begin{enumerate}
      		\item If $h\colon R \to R'$ is a morphism of $p$-complete rings such that $R'$ admits a Frobenius lift, which is an automorphism, and such that $h$ commutes with it, then  Kato's base-change equivalences identify
      		\[
      		h^*(\Phi_{ij})\simeq\Phi_{h\alpha_i,h\alpha_j}.
      		\]
      		These identifications are transitive under further
Frobenius-compatible base change.
      		\item
		Let $\sT=(T,\beta\colon N\to T)$ 
		be a $p$-adic fine log PD-base. Suppose that a given ring map $h\colon R\longrightarrow T$
		defines morphisms of log PD-bases
		\[
		h_i\colon\sR_i\longrightarrow\sT,
		\qquad
		h_j\colon\sR_j\longrightarrow\sT.
		\]
	(or, equivalently, suppose that 	$h\circ\alpha_i=\beta=h\circ\alpha_j$). 	For a polynomial object $\underline P_0$, let
		$\underline P_{0,T}$ denote its base change to the special
		fibre of $\sT$. Then the diagram
		\begin{equation}\label{eq:automorphism-transfer-specialization}
		\begin{tikzcd}
			R\Gamma_{\rm crys}(\underline P_0/\sR_i)
			\widehat\otimes^{\mathbb L}_{R,h}T
			\arrow[rr,
				"\Phi_{ij}(\underline P_0)
				\widehat\otimes^{\mathbb L}_{R,h}T"]
			\arrow[d,"\mathrm{BC}_{i,h}"']
			&&
			R\Gamma_{\rm crys}(\underline P_0/\sR_j)
			\widehat\otimes^{\mathbb L}_{R,h}T
			\arrow[d,"\mathrm{BC}_{j,h}"]
			\\
			R\Gamma_{\rm crys}(\underline P_{0,T}/\sT)
			\arrow[rr,"\id"']
			&&
			R\Gamma_{\rm crys}(\underline P_{0,T}/\sT)
		\end{tikzcd}
		\end{equation}
		commutes in $\mathcal D(T)^\wedge_p$, where $\mathrm{BC}_{i,h}$ (resp. $\mathrm{BC}_{j,h}$) denotes the base-change isomorphism. By animation, the same property extends to $\mathrm{Ani}(\mathrm{\PreLog}_{\mathcal{R}^{\rm sp}_0})$.
      	\end{enumerate}
      \end{cor}

      \begin{proof}
	Let 
  $U\colon\DA_R^{\str}\to \mathrm{CAlg}(\mathcal{D}(R)^\wedge_p). $  
  be again the forgetful functor, write
	$\Theta_i=\Theta_{-/\sR_i}$
    (cf. \eqref{eq;Theta-crys}),
    and let
	\[
	\varkappa_{ij}\colon
	\mathcal W\uomega^*_{-/\sR_i}
	\xrightarrow{\sim}
	\mathcal W\uomega^*_{-/\sR_j}
	\]
	be the comparison provided by
	Corollary~\ref{main_cor_WOmega_agrees}. Define
	\[
	\Phi_{ij}
	:=
	U(\Theta_j)\circ U(\varkappa_{ij})\circ U(\Theta_i)^{-1}
	\]
	on polynomial objects, and then apply sifted left Kan
	extension.\footnote{Note that sifted left Kan extension preserves algebra objects (see Remark \ref{rmk;multiplicativity}). }
    This is a natural equivalence by
	Theorem~\ref{cor;thm;omega-logcrys}.
	For a standard polynomial lift, let
	\[
	q_{i,\underline P}\colon
	\widehat\omega^*_{\underline P_i/\sR_i}
	\longrightarrow
	U\WSat\bigl(
	\widehat\omega^*_{\underline P_i/\sR_i}
	\bigr)=\mathcal W\uomega^*_{\underline P_i/\sR_i}  
	\]
	be the  unit map. By construction
    of $\mathcal W\uomega^*_{-/\sR_i}$, 
	\[
	U(\varkappa_{ij})\circ q_{i,\underline P}
	=
	q_{j,\underline P}\circ j_{ij,\underline P}.
	\]
	Together with \eqref{eq;Theta-normalization}, this gives
	\eqref{eq:automorphism-transfer-Kato-formula}.

	The last assertion of
Corollary~\ref{main_cor_WOmega_agrees} gives
	\[
	\varkappa_{ii}=\id,
	\qquad
	\varkappa_{jk}\circ\varkappa_{ij}=\varkappa_{ik}
	\]
	in the $1$-category of strict Dieudonn\'e algebras. Since
	these are 1-categorical identities, their images give the stated
	cocycle identities and the resulting higher coherences.

	Property {\rm(1)} follows by applying Kato's base-change
	theorem to
	\eqref{eq:automorphism-transfer-Kato-formula}: the maps
	$c_{i,\underline P}$ and $j_{ij,\underline P}$ base-change
	to the corresponding maps over $R'$. Transitivity is
	inherited from Kato's base-change isomorphisms, and sifted
	left Kan extension gives the assertion on the animated
	category.

	For {\rm(2)}, let $\underline P_{i,T}$ and
	$\underline P_{j,T}$ be the base change to $\sT$ of the two standard lifts 
    $\underline P_{i}$ and
	$\underline P_{j}$ respectively.
    Since
	\[
	h\circ\alpha_i=\beta=h\circ\alpha_j,
	\]
	they are canonically the same pre-log algebra over $\sT$;
	denote it by $\underline P_T$. Under Kato's base-change
	equivalences, both $c_{i,\underline P}$ and
	$c_{j,\underline P}$ therefore become
	\[
	c_{\underline P_T}\colon
	R\Gamma_{\rm crys}(\underline P_{0,T}/\sT)
	\xrightarrow{\sim}
	\widehat\omega^*_{\underline P_T/\sT},
	\]
	while $j_{ij,\underline P}$ becomes the identity. Hence the
	base change of
	\eqref{eq:automorphism-transfer-Kato-formula} is
	\[
	c_{\underline P_T}^{-1}\circ\id\circ c_{\underline P_T}
	=\id,
	\]
	which is precisely the commutativity of
	\eqref{eq:automorphism-transfer-specialization}. 
	Sifted left Kan extension gives the assertion on the animated
	category.
\end{proof}
    
      In order to correctly recover log crystalline cohomology from its animated version, we need the following Lemma. 
      \begin{lem}\label{lem;leftKanCrys}
      	Let $\sR=(R,\alpha_R\colon N\to R)$ be as in \eqref{eq;SN} and write $\sR/p=(R/p,N)$. The restriction on the subcategory $\mathrm{lSm}^{\rm Cart}_{\sR/p}$ of affine fine log-smooth algebras of Cartier type of the log-crystalline cohomology functor
      	\[
      	R\Gamma_{crys}(-/\sR) \colon \PreLog_{\sR/p} \to \mathcal{D}(R)^\wedge_p
      	\]
      	is the left Kan extension of its restriction to $\mathrm{lPoly}_{\sR/p}$. In particular, derived log crystalline cohomology and underived crystalline cohomology coincide for log smooth algebras of Cartier type.  
         Moreover, the algebra structure of $R\Gamma_{crys}(X/\sR)$ for $X\in \mathrm{lSm}^{\rm Cart}_{\sR/p}$ is left Kan extended from that for polynomial objects. 
     	\begin{proof}This is basically well-known, and follows the same argument of the non-logarithmic counterpart. We give a proof for the convenience of the reader. 
      		Considering the restriction of $R\Gamma_{crys}(-/\sR)$ to  polynomial pre-log algebras and the resulting left Kan extension, (here, we use the fact that $\mathcal{D}(R)^\wedge_p$ admits colimits)
      		we obtain a natural comparison map from derived log crystalline cohomology $dR\Gamma_{crys}(-/\sR)$ to $R\Gamma_{crys}(-/\sR)$. Both functors satisfy étale descent: this is clear for the underived version. As for derived log crystalline cohomology, it is enough to note that, since $dR\Gamma_{crys}(-/\sR)$ takes value in $p$-complete $\sR$-modules, we can check descent by reduction modulo $p$. Since by construction derived log crystalline cohomology is obtained by Kan extension from polynomials pre-log algebras over $\sR/p$, on any such algebra 
      		$\sP$,  we can choose a polynomial algebra $\widetilde{\sP}$ over $\sR$ as its smooth lift. Then, the derived log crystalline cohomology $R\Gamma_{\text{crys}}(\sP/\sR)$ is exactly computed by the $p$-completion of $\omega_{\widetilde{\sP}/\sR}^{*}$ as \eqref{eq;omegalogcryscomparison}. 
      		Reducing modulo $p$ again, this is precisely $\omega_{\sP/(\sR/p)}^{*}$, thus $dR\Gamma_{crys}(-/\sR)/p$ is $L\omega^*_{-/(\sR/p)}$, which satisfies étale descent. At this point, the claim of the Lemma follows again from the fact that         for log smooth pre-log algebras $\sA$ of Cartier type over $\sR/p$,
      		the comparison map 
      		$L\omega^*_{\sA/(\sR/p)}\to \omega^*_{\sA/(\sR/p)}$ is an equivalence (see for example \cite[Proposition 4.6]{BLPO-HKR}).
            Finally, the last statement follows from Remark \ref{rmk;multiplicativity}). 
                 	\end{proof}
      \end{lem}
      
      \begin{cor}\label{cor:main_cor2}
      	With the assumptions of Corollary~\ref{cor:main_cor}, the comparisons
      	$\Phi_{ij}$ extend uniquely to natural equivalences in 
        $\Fun\bigl(\Ani(\PreLog_{\sR_0^{\rm sp}}), \mathrm{CAlg}(\mathcal{D}(R)^\wedge_p). \bigr)$ 
      	\[\Phi_{ij} :  
      	R\Gamma_{\rm crys}(-/\sR_i)
      	\xrightarrow{\ \sim\ }
      	R\Gamma_{\rm crys}(-/\sR_j)
      	\]
      	on the category of fine, quasi-compact and quasi-separated log-smooth schemes
      	of Cartier type over the common special fibre. They satisfy the same cocycle
      	and base-change compatibilities as in Corollary~\ref{cor:main_cor}.
      \end{cor}
      
      \begin{proof}
      	For charted affine objects (affine, log affine) this follows from Lemma~\ref{lem;leftKanCrys}
      	and Corollary~\ref{cor:main_cor}. The general case follows from strict
      	\'etale descent, using strict \'etale local charts. All compatibilities extend
      	because the constructions and the descent identifications are natural.
      \end{proof}

\section{Globalizing the saturated de Rham Witt complex}\label{sec:global}
In this Section we explain how to globalize the construction of  the saturated log de Rham Witt complex to general quasi-coherent log schemes (i.e., log schemes that, Zariski locally, admit a chart). The procedure is general, but for simplicity we restrict ourselves to the case where the base pre-log ring is
\[\sR =(W(k), [\alpha]\colon N\to W(k)),\]
where $[\alpha]$ is the Teichmuller lift of a log point 
$(k,N)=(k,\alpha:N\to k)$. Note $\sR/p=(k,N)$.
This is done in order to use some result from \cite{Ya}. Note that the globalization procedure in \cite{Ya} makes use of Lemma \cite[A.1]{Ya}, for which a counterexample was constructed by Tsuji (see \cite[Example 3.6.1]{achingeretal}). We use  instead the globalization criterion of
\cite[Section~2.4]{BLMP}.

\begin{prop}\label{WWYcomparison}
The functor of Theorem~\ref{thm;mathcaWomega}:
\[
\mathcal W\uomega^*_{-/\sR}\colon
\mathrm{lPoly}_{(k,N)}\longrightarrow\DA_{W(k)}^{\str}
\]
agrees with Yao's saturated log de Rham--Witt functor \cite[Definition 4.1]{Ya},
$\sP_0\mapsto\mathcal W^Y\omega^*_{\sP_0/(k,N)}$.
\end{prop}
    \begin{proof}
       When $\sR$ has log structure given by the Teichmuller lift, \cite[Theorem 4.20]{Ya} implies that there is a quasi-isomorphism $\widehat{\omega}_{\sP/\sR}^* \simeq \mathcal{W}^Y\omega_{\sP/p/(k,N)}^*$ for any log polynomial algebra $\sP$ over $\sR$ lifting $\sP/p$ over $\sR/p=(k,N)$ and  $\widehat{\omega}^*_{\sP/\sR}\to \WSat(\uomegaSshat{\sP})$ is a quasi-isomorphism by \cite[Corollary 2.8.5]{BLM}.  This completes the proof in view of \cite[7.3.4]{BLM}. 
    \end{proof}

We write
\[
L\mathcal W\uomega^*_{-/\sR}\colon
\mathrm{Ani}(\mathrm{PreLog}_{(k,N)})
\longrightarrow
\CAlg(\mathcal D(W(k))^\wedge_p)
\]
for the sifted left Kan extension of
$\mathcal W\uomega^*_{-/\sR}$ from polynomial objects.
 
\begin{rmk}\label{rmk:modp_isdeRham} Consider the canonical map $\mathcal{D}(W(k))^{\wedge}_p \to \mathcal{D}(k)$ given by the reduction modulo $p$.
There is a canonical natural equivalence
\[
L\mathcal W\uomega^*_{-/\sR}/p
\simeq
L\omega^*_{-/(k,N)}
\]
of functors
$\mathrm{Ani}(\mathrm{PreLog}_{(k,N)})\to\mathcal D(k)$, where the
right-hand side is the derived log de Rham complex. Indeed, both sides
preserve sifted colimits, so it is enough to check the claim on
$\mathrm{lPoly}_{(k,N)}$. For a polynomial object $\sP_0$, the canonical
map
\[
\omega^*_{\sP_0/(k,N)}
\longrightarrow
\mathcal W_1\uomega^*_{\sP_0/\sR}
\]
is an isomorphism by \cite[Theorem~5.8]{Ya}, while
\[
\mathcal W\uomega^*_{\sP_0/\sR}/p
\longrightarrow
\mathcal W_1\uomega^*_{\sP_0/\sR}
\]
is a quasi-isomorphism by \cite[Corollary~2.7.2]{BLM}; see also
\cite[Corollary~5.9]{Ya}.
\end{rmk}

\begin{lem}[{See \cite[Theorem 4.33]{Ya}}]\label{lem:descent} For every $\sA=(A,L)\in \mathrm{PreLog}_{(k,N)}$,
    the functor $L\mathcal{W}\uomega^*_{-/\sR}$. restricted to the  category of $A$-algebras, given as $B\mapsto L\mathcal{W}\uomega^*_{(B,L)/\sR}$, satisfies étale descent.
    \begin{proof}
        Since $L\mathcal{W}\uomega^*_{-/\sR}$ takes value in a category of $p$-completed modules, and since this category is stable, we can check the étale descent by working modulo $p$. By Remark \ref{rmk:modp_isdeRham}, the functor now agrees with $B\mapsto \omega^*_{ (B,L)/(k,N)}$ and the statement is clear. 
    \end{proof}
\end{lem}
\begin{lem}\label{lem:log_inv}
    For every $(A,L)\in \mathrm{PreLog}_{(k,N)}$, the natural map $L\mathcal{W}\uomega^*_{(A,L)/\sR} \to L\mathcal{W}\uomega^*_{(A,L)^a/\sR}$ is an equivalence, where $(A,L)^a$ denotes the associated log structure (in other words, the functor $L\mathcal{W}\uomega^*_{-/\sR}$ restricted to ordinary pre-log rings is invariant under logification). 
    \begin{proof} Once again using Remark \ref{rmk:modp_isdeRham}, it suffices to show the lemma after reducing modulo $p$. Then, it follows from the analogous statement for the derived log de Rham complex $L\omega^*_{-/(k,N)}$ which is proved, for example, in \cite[Lemma 3.11]{BLPO-HKR}. 
    \end{proof}
\end{lem}
For the globalization we will use the following general result taken from \cite[Section 2.4]{BLMP}. Let $\mathcal{C}$ be an $\infty$-category which is both complete and cocomplete, and let $\mathcal{F}\in \PSh(\mathrm{PreLog}^{op}, \mathcal{C})$. We say that $\mathcal{F}$ is \emph{globalizable} if the unit morphism 
\[\alpha_{\mathcal{F}}\colon \mathcal{F}\to \Gamma_* L_{s\acute{e}t}\Gamma^* \mathcal{F
}\]
is an equivalence, where $L_{s\acute{e}t}$ is the strict étale sheafification functor from $\PSh(\mathrm{lSch}^{qcoh}, \mathcal{C})$ to $\mathrm{Shv}_{s\acute{e}t}(\mathrm{lSch}^{qcoh}, \mathcal{C})$, and $(\Gamma_*,\Gamma^*)$ is the adjoint pair induced by the global section functor $\Gamma(-)$, sending a log scheme $\sX = (X, \mathcal{M}_X)$ to the pre-log ring $(\Gamma(X,\mathcal{O}), \Gamma(X,\mathcal{M}_X))$ (see \cite[Definition 2.19]{BLMP}). 
Equivalently, there exists a strict étale sheaf $\mathcal{G}$ on quasi-coherent log schemes such that its restriction to the spectrum of any pre-log ring agrees with $\mathcal{F}$. 
\begin{prop}[{\cite[Proposition 2.20]{BLMP}}]
\label{prop_global}    With the above notation,
assume that $\sF$ satisfies the following conditions:
\begin{enumerate}
\item[(i)] $\sF$ preserves filtered colimits as a functor $\mathrm{PreLog}\to \sC$.
\item[(ii)] For every pre-log ring $(A,M)$,
the presheaf
\[
(A\to B)\mapsto \sF(B,M)
\]
on the opposite category of $A$-algebras is an \'etale sheaf.
\item[(iii)] For every pre-log ring $(A,M)$ such that $A$ is strictly local, the induced morphism $\sF(A,M)\to \sF(A,M^a)$ is an isomorphism,
i.e., we have the ``logification invariance'' when the underlying ring is strictly local.
\end{enumerate}
Then $\sF$ is globalizable.
\end{prop}

\begin{thm}\label{thm;sWomega-glabel}
    There exists a unique strict étale sheaf $L\mathcal{W}\uomega^*_{-/\sR}$ of $p$-complete $W(k)$-modules on $\mathrm{lSch}^{qcoh}_{(k,N)}$ such that for each $\sX$ and each étale local chart $(A,L)\in \PreLog_{(k,N)}$ of $\sX$, the natural map  $L\mathcal{W}\uomega^*_{(A,L)/\sR} \to \Gamma(U, L\mathcal{W}\uomega^*_{/\sR})$ is an equivalence, where $U\in \mathrm{lSch}^{qcoh}_{(k,N)}$ is associated to the pre-log scheme $(\Spec(A),L)$. 
    \begin{proof}
By definition,
$L\mathcal W\uomega^*_{-/\sR}$ is a sifted left Kan extension. It
therefore preserves sifted colimits, and in particular filtered
colimits. This verifies condition~\textup{(i)} of
Proposition~\ref{prop_global}. Condition~\textup{(ii)} is
Lemma~\ref{lem:descent}, and condition~\textup{(iii)} follows from
Lemma~\ref{lem:log_inv}. Proposition~\ref{prop_global} now gives the
claimed strict-\'etale sheaf and its uniqueness.
\end{proof}
\end{thm}

\section{An integral Hyodo-Kato isomorphism}\label{integralHK}
Let $K$ be a complete discrete valuation field of mixed characteristic $(0,p)$
with valuation ring $\sO_K$ and residue field $k$.
We assume that $k$ is perfect, normalize the valuation by
$v_K(\pi)=1$, and write $e=v_K(p)$ for the absolute ramification
index.
In what follows, we fix a prime element $\pi\in\sO_K$.
Let $W=W(k)$ be the ring of Witt vectors and $\phi:W\to W$ the Witt-vector
Frobenius.
Consider the pre-log rings
\[
\uW^0:=(W,\N\to W;1\mapsto0),\qquad
\uk^0=(k,\N\to k;1\mapsto0),\qquad
\OKpi=(\OK,\N\to\OK;1\mapsto\pi).
\]
Note that $\uk^0=\uW^0/p=\OKpi/(\pi)$.

For $m\geq0$ with $e\leq p^m$, put
\[
\OK^{\pi^{p^m}}
:=(\OK,\N\to\OK;1\mapsto\pi^{p^m}).
\]
Let
\[
g_m:\uk^0\longrightarrow\uk^0
\]
be induced by $\phi^m$ on $k$ and the identity on $\N$, and let
\[
h_m:\OK^{\pi^{p^m}}\longrightarrow\OKpi
\]
be the identity on the underlying ring and multiplication by $p^m$ on $\N$.
Since $\pi^{p^m}\in p\OK$, the $p^m$-power map on $\OK/p$ factors through
$\OK/(\pi)=k$. Hence, the  $m$-fold Frobenius
\[
F_m:\uk^0\longrightarrow\OKpi/p
\]
is the composite of $g_m$, the natural map
$\uk^0\to\OK^{\pi^{p^m}}/p$, and $h_m$ modulo $p$.
For a log scheme $\sX\to\Spec(\uk^0)$, put
\begin{equation}\label{eq;XOKm}
	\FttOK\sX m
	:=\sX\times_{\uk^0,F_m}\OKpi/p.
\end{equation}
We also write
\begin{equation}\label{eq;Xm}
\Ftt\sX m:=\sX\times_{\uk^0,g_m}\uk^0.
\end{equation}

For a fine $p$-adic log scheme
$\fX\to\Spec(\OKpi)$ as below, let
$\widehat\omega^*_{\fX/\OKpi}$ denote the $p$-adically completed
relative log de Rham complex, regarded as an object of
$\mathcal D(\sX_{\acute et},\OK)^\wedge_p$, where
$\sX=\fX\times_{\OKpi}\uk^0$.  We set
\begin{equation}\label{eq;sheafwise-Leta-dR}
	R\Gamma_{\dR}^{\eta^m}(\fX/\OKpi)
	:=
	R\Gamma\!\left(
	\sX_{\acute et},
	L\eta_p^m\widehat\omega^*_{\fX/\OKpi}
	\right).
\end{equation}


\begin{lem}\label{lem;CrysDRComparisontwisted}
	Let $m\geq0$ with $e\leq p^m$, and let
	$\tsX\to\Spec(\OKpi)$ be a fine, quasi-compact and quasi-separated
	log-smooth morphism of Cartier type. Put
	\[
	  \sX=\tsX\times_{\OKpi}\uk^0.
	\]
	Then, there is a natural equivalence
	\[
	\xi_{\sX}:R\Gamma_{\rm crys}(\FttOK\sX m/\OKpi)
	\xrightarrow{\ \sim\ }
	R\Gamma_{\dR}^{\eta^m}(\tsX/\OKpi)
	\]
	in $\mathcal D(\OK)^\wedge_p$ so that, for any morphism
	$f:\tsY\to\tsX$ of such log schemes over $\Spec(\OK^\pi)$, the diagram
	\[
	\xymatrix{
		R\Gamma_{\rm crys}(\FttOK\sX m/\OKpi)
		\ar[r]^-{\xi_{\sX}}\ar[d]^{f^*}
		&R\Gamma_{\dR}^{\eta^m}(\tsX/\OKpi)\ar[d]^{f^*}\\
		R\Gamma_{\rm crys}(\FttOK\sY m/\OKpi)
		\ar[r]^-{\xi_{\sY}}
		&R\Gamma_{\dR}^{\eta^m}(\tsY/\OKpi)
	}
	\]
	commutes.
\end{lem}
\begin{proof}
	Let  $\tsXp=\tsX\times_{\OKpi}\OKpi/p$.
	 Then, the   factorization of the  Frobenius above gives a canonical
	identification
	\[
	\FttOK\sX m
	\simeq
	\tsXp\times_{\OKpi/p,F^m_{\OKpi/p}}\OKpi/p.
	\]
	The divided-relative-Frobenius construction in the proof of
	\cite[Proposition~2.24]{HK}, and specifically the quasi-isomorphism of
	\cite[Lemma~2.25]{HK}, identifies 
the pullback by relative Frobenius of the crystalline complex on the logarithmic crystalline topos with its $\eta_p$-subcomplex. 
Iterating the construction $m$ times and passing to the derived inverse
	limit over the finite PD-bases gives a natural equivalence in the derived
	category of sheaves on $\sX_{\acute et}$,
	\[
	R u^{\log}_{\FttOK\sX m/\OKpi,*}\mathcal O_{\rm crys}
	\xrightarrow{\ \sim\ }
	L\eta_p^m
	R u^{\log}_{\tsXp/\OKpi,*}\mathcal O_{\rm crys}.
	\]
	  Kato's 
	crystalline--de Rham comparison \cite[(6.4)]{katolog},
	 applied at every finite level and then
	$p$-adically completed, identifies
	\[
	R u^{\log}_{\tsXp/\OKpi,*}\mathcal O_{\rm crys}
	\xrightarrow{\ \sim\ }
	\widehat\omega^*_{\tsX/\OKpi}.
	\]
	Applying $L\eta_p^m$ in the derived category of sheaves and then taking
	derived global sections gives the asserted equivalence. Both constructions
	are functorial, completing the proof. 
\end{proof}

\begin{thm}\label{thm.integralHK}Fix a uniformizer $\pi$ as above.
	Assume $e\leq p^m$.
	For a fine, quasi-compact and quasi-separated
	log-smooth morphism $\sX\to\Spec(\uk^0)$ of Cartier type, there is a natural
	equivalence  in $\mathrm{CAlg}(\mathcal{D}(\OK)^\wedge_p)$  
	\[
	\rho_{\pi,m}^{\rm int}\colon R\Gamma_{\rm crys}(\sX/\uW^0)
	\widehat\otimes^{\mathbb{L}}_{W,\phi^m}\OK
	\simeq
	R\Gamma_{\rm crys}(\FttOK\sX m/\OKpi),
	\]
	functorial for morphisms of such log schemes over $\uk^0$.
\end{thm}

\begin{rmk}\label{rem.integralHK}
	
	If $\OK=W$, $\pi=p$, and $m=0$, Theorem~\ref{thm.integralHK} gives the
	unramified comparison
	\[
	R\Gamma_{\rm crys}(\sX/W^0)
	\xrightarrow{\ \sim\ }
	R\Gamma_{\rm crys}(\sX/W^\times),
	\qquad
	W^\times=(W,\N\ni1\mapsto p).
	\]
	This of course  is given directly by 
	Corollary~\ref{cor:main_cor2} applied over $W$, together with the identification with log-crystalline cohomology.
	
\end{rmk}
 \begin{rmk}\label{rmk:uniformizer-dependence}
The naturality in Theorem~\ref{thm.integralHK} is  in $\sX$ for a
fixed uniformizer $\pi$. The proof below shows that
$\rho_{\pi,m}^{\rm int}$ is independent of  auxiliary lifts of $\pi$ chosen in
$A_{\inf}(\OC)$ with the ring $\OC$ of integers in the completion of an algebraic closure $C$ of $K$
and on the \v{C}ech nerve, but the map itself still depends on a chosen uniformizer $\pi$. The canonical unramified comparison of Remark \ref{rem.integralHK}, defined directly from the two
parameters $0$ and $p$ is, on the other hand, canonical (there are no other choices involved).
\end{rmk}
The proof is obtained by descent. To this end, we will need to use  the base-change compatibilities  recorded in Corollary~\ref{cor:main_cor2}. In particular, if two
base log structures become equal after a morphism of log PD-bases, the
base-changed comparison is canonically the identity.  
		In order to get the proof, we first base change to the ring of integers $\mathcal{O}_C$ of the completion of an algebraic closure of $K$.

\begin{thm}\label{thm.integralHK-OC}Fix a uniformizer $\pi$ as above.
	Assume $e\leq p^m$.
	For a fine, quasi-compact and quasi-separated
	log-smooth morphism $\sX\to\Spec(\uk^0)$ of Cartier type, there is a natural
	equivalence
    in $\mathrm{CAlg}(\mathcal{D}(\OC)^\wedge_p)$  
    (depending on the choice of $\pi$)
	\[
	R\Gamma_{\rm crys}(\sX/\uW^0)
	\widehat\otimes^{\mathbb{L}}_{W,\phi^m}\OC
	\simeq
	R\Gamma_{\rm crys}(\FttOK\sX m/\OKpi)
	\widehat\otimes^{\mathbb{L}}_{\OK}\OC,
	\]
	functorial in $\sX$.
\end{thm}

We need some preparations for the proof. Put
\[
R=\Ainf(\OC)=W(\OC^\flat),
\]
with its Witt-vector Frobenius $F_R$, which is an automorphism, and let $\theta\colon R\to\OC$
be Fontaine's map.
By the assumption $e\leq p^m$, we can write $\pi^{p^m}=p\,u,$ for $u\in \OK$ (note that $u=u_\pi$ depends here on the choice of $\pi$). 

Choose an element $a\in R$ with $\theta(a)=u$ (it exists because
$\theta$ is surjective).   
Consider the pre-log rings
\begin{equation}\label{eq;SNPD}
	\sR^0=(R,\alpha^0:\N\to R),
	\qquad
	\sRp=(R,\alpha^p:\N\to R),
\end{equation}
where $\alpha^0(1)=0$ and 
$\alpha^p(1)=pa$. 
These pre-log rings satisfy \eqref{eq;SN} and
\begin{equation}\label{eq;SNSN}
	\sRp/p=\sR^0/p.
\end{equation}
By  Corollary~\ref{cor:main_cor2}, for every fine
log-smooth morphism $\sY\to\sR^0/p$ of Cartier type, there is a functorial
equivalence
\begin{equation}\label{eq;Cryscomparison1}
\Phi_{0,p}=\Phi_{\sR^0,\sR^p}:	R\Gamma_{\rm crys}(\sY/\sR^0)
	\xrightarrow{\ \sim\ }
	R\Gamma_{\rm crys}(\sY/\sRp).
\end{equation}

Consider the pre-log rings
\[
\OC^0=(\OC,\N\ni1\mapsto0),
\qquad
\OC^{\pi^{p^m}}=(\OC,\N\ni1\mapsto\pi^{p^m}),
\qquad
\OC^\pi=(\OC,\N\ni1\mapsto\pi).
\]
The map $\theta$ induces a morphism of pre-log rings
\[
\theta:\sRp\longrightarrow\OC^{\pi^{p^m}},
\]
and the identity on $\OC$ and the multiplication by $p^m$ on $\N$ induces
\[
h_m:\OC^{\pi^{p^m}}\longrightarrow\OC^\pi.
\]
Let
\[
\mu:\uk^0=\uW^0/p\longrightarrow\OC^0/p
=\OC^{\pi^{p^m}}/p
\]
be induced by $W\to\OC$.
For a log scheme $\sX\to\Spec(\uk^0)$, let (cf. \eqref{eq;Xm})  
\[
\XRm
=
\Ftt\sX m\times_{\uk^0}\sR^0/p
=
\Ftt\sX m\times_{\uk^0}\sRp/p.
\]
Then there is a canonical identification
\begin{equation}\label{eq1;claim;thm.integralHK}
	\XRm\times_{\sRp,\theta}\OC^{\pi^{p^m}}
	\simeq
	\Ftt\sX m\times_{\uk^0,\mu}\OC^{\pi^{p^m}}/p.
\end{equation}

Moreover, we have a canonical identification
\begin{equation}\label{eq;XmXOKm}
	\left(
	\Ftt\sX m\times_{\uk^0,\mu}\OC^{\pi^{p^m}}/p
	\right)
	\times_{\OC^{\pi^{p^m}},h_m}\OC^\pi
	\simeq
	\FttOK\sX m\times_{\OKpi/p}\OC^\pi/p,
\end{equation}
noting that $\uk^0\rmapo{g_m} \uk^0\rmapo{\mu} \OC^{\pi^{p^m}}/p\rmapo{h_m}\OC^{\pi}/p$ coincides with $\uk^0 \rmapo{F_m} \OK^{\pi}/p \to \OC^\pi/p$.

Below, all scalar extensions  are derived and
	$p$-completed.  This is essential for the maps $\theta$ and $\theta_n$ below,
	which need not be flat.

\begin{proof}[Proof of Theorem~\ref{thm.integralHK-OC}]
	Kato's base-change theorem and \eqref{eq;Cryscomparison1} give natural
	equivalences
	\[
	\begin{aligned}
		R\Gamma_{\rm crys}(\sX/\uW^0)
		\widehat\otimes^{\mathbb{L}}_{W,\phi^m}\OC
		&\overset{(*1)}{\simeq}
		R\Gamma_{\rm crys}(\Ftt\sX m/\uW^0)
		\widehat\otimes^{\mathbb{L}}_{W}\OC\\
		&\simeq
		R\Gamma_{\rm crys}(\Ftt\sX m/\uW^0)
		\widehat\otimes^{\mathbb{L}}_{W}R
		\widehat\otimes^{\mathbb{L}}_{R,\theta}\OC\\
		&\overset{(*2)}{\simeq}
		R\Gamma_{\rm crys}(\XRm/\sR^0)
		\widehat\otimes^{\mathbb{L}}_{R,\theta}\OC\\
		&\overset{(*3)}{\simeq}
		R\Gamma_{\rm crys}(\XRm/\sRp)
		\widehat\otimes^{\mathbb{L}}_{R,\theta}\OC\\
		&\overset{(*4)}{\simeq}
		R\Gamma_{\rm crys}\!\left(
		\Ftt\sX m\times_{\uk^0,\mu}\OC^{\pi^{p^m}}/p
		\,/\,\OC^{\pi^{p^m}}
		\right)\\
		&\overset{(*5)}{\simeq}
		R\Gamma_{\rm crys}(\FttOK\sX m/\OKpi)
		\widehat\otimes^{\mathbb{L}}_{\OK}\OC.
	\end{aligned}
	\]
	Here $(*1)$ is base change along the morphism $g_m:W^0\to W^0$ and then
	along $W\to\OC$; $(*2)$ is base change along $W^0\to\sR^0$; $(*3)$ is the $F_R$-automorphism-case comparison \eqref{eq;Cryscomparison1}; and $(*4)$ is base
	change along $\theta:\sRp\to\OC^{\pi^{p^m}}$, using
	\eqref{eq1;claim;thm.integralHK}. Finally, using \eqref{eq;XmXOKm},
	$(*5)$ is the composite of base
	change along
	\[
	h_m:\OC^{\pi^{p^m}}\longrightarrow\OC^\pi
	\]
	and Kato base change along $\OKpi\to\OC^\pi$. All applications of base
	change are instances of \cite[Theorem~6.10]{katolog}.

	If $a'\in R$ is another element with $\theta(a')=u$, the cocycle identity
	of Corollary~\ref{cor:main_cor2} expresses the comparison
	$\Phi_{0,pa'}$ as the comparison $\Phi_{0,pa}$ followed by $\Phi_{pa,pa'}$. After base
	change along $\theta$, the latter becomes the identity because both log
	parameters specialize to $\pi^{p^m}$ (see Corollary \ref{cor:main_cor}(2)).
	 Hence the displayed equivalence is
	canonically independent of the choice of $a$.
	
	This proves the theorem.
\end{proof}

We now need to descend Theorem \ref{thm.integralHK-OC} from $\mathcal{O}_C$ to $\mathcal{O}_K$. The map $\sO_K\to \OC$ is faithfully flat, but  the self-intersections $B_n = \OC\widehat\otimes^{\mathbb{L}}_{\OK}\cdots
		\widehat\otimes^{\mathbb{L}}_{\OK}\OC$ ($n+1$ times), are not perfectoid anymore, and the argument of Theorem \ref{thm.integralHK-OC} does not immediately apply. This can be easily fixed, thanks to the following construction. 
\begin{lem}\label{lem;Cech-Fontaine-rings}
	For $n\geq0$, let $B_n$ be as above. 
	Then $B_n$ is discrete, $p$-complete and $p$-torsion-free, and Frobenius on
	$B_n/p$ is surjective. Put
	\[
	B_n^\flat:=\varprojlim_{x\mapsto x^p}B_n/p,
	\qquad
	A_n:=W(B_n^\flat).
	\]
	The ring $A_n$ is $p$-complete and $p$-torsion-free, its Frobenius is an
	automorphism, and there is a natural surjective Fontaine map
	\[
	\theta_n:A_n\longrightarrow B_n.
	\]
	The constructions $B_n^\flat$, $A_n$, and $\theta_n$ are functorial in the
	cosimplicial ring $B_\bullet$; the induced maps between the $A_n$ are
	Frobenius-compatible, and $\theta_n$ is compatible with the structural maps
	$W(k)\to A_n$ and $W(k)\to B_n$.
\end{lem}
\begin{proof}As we have observed, the map $\OK\to\OC$ is faithfully flat, since $\OK$ is a discrete valuation
	ring and $\OC$ is torsion-free over it. Hence the ordinary tensor products coincide with the derived
	tensor products. Their derived $p$-completions are discrete and $p$-completely
	flat over $\OK$ (in particular, they are $p$-torsion-free). These completions
	are the precisely the rings $B_n$.
	Modulo $p$ one has
	\[
	B_n/p\simeq
	\underbrace{
		(\OC/p)\otimes_{\OK/p}\cdots\otimes_{\OK/p}(\OC/p)
	}_{n+1\text{ factors}}.
	\]
	Frobenius on $\OC/p$ is surjective, and this immediately says that  Frobenius on $B_n/p$ is surjective.
	
	The inverse limit $B_n^\flat$ is therefore a perfect $\F_p$-algebra, so the Frobenius on $A_n=W(B_n^\flat)$ is an automorphism. Since $k$ is
	perfect, the structural map $k\to B_n/p$ lifts canonically to a
	Frobenius-compatible map $k\to B_n^\flat$, and hence to a
	Frobenius-compatible map $W=W(k)\to A_n$. The standard construction of Fontaine gives a natural ring map $\theta_n\colon A_n\to B_n$, and the composition with $W(k)\to A_n$ is immediately seen to be the structural morphism. Note that the   reduction modulo $p$ of $\theta_n$ is the
	projection $B_n^\flat\to B_n/p$, which is surjective. Successively lifting
	modulo $p^r$ and using $p$-adic completeness proves that $\theta_n$ is
	surjective as well (this is in fact well-known).
\end{proof}

\begin{proof}[Proof of Theorem~\ref{thm.integralHK}]
	
	In what follows, for $c\in \OK$, we let $c$ denote $c\otimes 1\otimes \cdots\otimes 1\in B_n$.
	We now deduce the result from Theorem~\ref{thm.integralHK-OC}. For each $n$, choose
	\[
	a_n\in A_n
	\quad\text{with }
	\theta_n(a_n)=u,
	\]
	and equip $A_n$ with the two pre-log structures
	\[
	\sA_n^0=(A_n,\N\ni1\mapsto0),
	\qquad
	\sA_n^p=(A_n,\N\ni1\mapsto pa_n).
	\]
	We also write
	\[
	B_n^{\pi^{p^m}}=(B_n,\N\ni1\mapsto\pi^{p^m}),
	\qquad
	B_n^\pi=(B_n,\N\ni1\mapsto\pi).
	\]
  	Then $\theta_n$ induces a morphism
	$\sA_n^p\to B_n^{\pi^{p^m}}$, and the identity on $B_n$ together with
	multiplication by $p^m$ on $\N$ gives
	$h_m:B_n^{\pi^{p^m}}\to B_n^\pi$.
	The proof of Theorem~\ref{thm.integralHK-OC}, with
	$(R,\theta,\OC,a)$ replaced by $(A_n,\theta_n,B_n,a_n)$, gives a natural
	equivalence
	\begin{equation}\label{eq;proofthm.integralHK}
		E_n(a_n):
		R\Gamma_{\rm crys}(\sX/\uW^0)
		\widehat\otimes^{\mathbb{L}}_{W,\phi^m}B_n
		\xrightarrow{\ \sim\ }
		R\Gamma_{\rm crys}(\FttOK\sX m/\OKpi)
		\widehat\otimes^{\mathbb{L}}_{\OK}B_n.
	\end{equation}

	We first check that this equivalence is independent of the choice of $a_n$.
	Let $a_n'$ be another lift of $u$. By the cocycle identity of
	Corollary~\ref{cor:main_cor2},
	\[
	\Phi_{0,pa_n'}
	\simeq
	\Phi_{pa_n,pa_n'}\circ\Phi_{0,pa_n}.
	\]
	After base change along $\theta_n$, both  parameters $pa_n$ and $pa_n'$
	become $pu=\pi^{p^m}$. The final base-change assertion of
	Corollary~\ref{cor:main_cor2}  and Corollary \ref{cor:main_cor}(2)
	therefore identifies the base
	change of $\Phi_{pa_n,pa_n'}$ with the identity. 
    Hence
	\[
	E_n(a_n)\simeq E_n(a_n')
	\]
	canonically, and these identifications satisfy the cocycle condition.
		Let $\delta:B_n\to B_r$ be a face or degeneracy map. Functoriality in
	Lemma~\ref{lem;Cech-Fontaine-rings} gives a Frobenius-compatible map
	\[
	\delta_{\inf}:A_n\longrightarrow A_r
	\]
	with $\theta_r\circ\delta_{\inf}=\delta\circ\theta_n$. The base-change
	compatibility in Corollary~\ref{cor:main_cor2} gives
	\[
	E_n(a_n)\widehat\otimes^{\mathbb{L}}_{B_n,\delta} B_r 
	\simeq
	E_r(\delta_{\inf}(a_n)).
	\]
	Since $\theta_r(\delta_{\inf}(a_n))=u$, choice-independence 
	in Theorem~\ref{thm.integralHK-OC} 
	identifies the
	right-hand side with $E_r(a_r)$. All the remaining arrows in the construction
	of $E_n$ are Kato base-change maps, and their transitivity is compatible with
	the cosimplicial structure. 
	The cocycle identity therefore shows that these
	identifications are compatible with all composites of faces and degeneracies.
	Thus the equivalences \eqref{eq;proofthm.integralHK} form an equivalence of
	cosimplicial descent diagrams.
	
 Faithfully flat
	descent for derived $p$-complete modules
	\cite[Appendix~D.6.3.3]{SAG}, applied in the functor category,
	therefore descends the system $E_\bullet$ to the natural equivalence of
	Theorem~\ref{thm.integralHK}. Since equivalences may be checked after the
	$p$-completely faithfully flat base change $\OK\to\OC$, the descended
	transformation is an equivalence.
	
\end{proof}

\begin{cor}\label{cor.integralHK}
	Let the notation be as in Theorem~\ref{thm.integralHK}. Let
	$\tsX$ be a fine, quasi-compact and
	quasi-separated log-smooth scheme of Cartier type over $\OKpi$, and let $\sX=\tsX\times_{\OKpi}\uk^0.$ If 
      $e\leq p^m$, there is a
		natural equivalence
		\[
\Phi_{\pi, \mathfrak X }^m\colon		R\Gamma_{\rm crys}(\sX/W^0)
		\widehat\otimes^{\mathbb{L}}_{W,\phi^m}\OK
		\simeq
		R\Gamma_{\dR}^{\eta^m}(\tsX/\OKpi),
		\]
		functorial for morphisms over $\OKpi$.
\end{cor}
\begin{proof}
    This follows from
	Lemma~\ref{lem;CrysDRComparisontwisted} and
	Theorem~\ref{thm.integralHK}.
\end{proof}
Finally, we record the dependency on the choice of a uniformizer $\pi$ of $\OK$. This is similar to the classical dependency in \cite{HK}. 
\begin{prop}[Change of uniformizer]\label{prop:change-uniformizer}
Let $\pi$ and $\pi'$ be two uniformizers of $\sO_K$, and let $m\geq0$
satisfy $e\leq p^m$. Then there is a canonical natural equivalence
\[
T_{\pi,\pi'}^{(m)}\colon
R\Gamma_{\rm crys}
\bigl(
 \sX_{\OK^\pi}^{(m)} 
/\sO_K^\pi\bigr)
\xrightarrow{\ \sim\ }
R\Gamma_{\rm crys}
\bigl(
 \sX_{\OK^{\pi'}}^{(m)} 
/\sO_K^{\pi'}\bigr)
\]
such that
\[
\rho_{\pi',m}^{\rm int}
\simeq
T_{\pi,\pi'}^{(m)}\circ\rho_{\pi,m}^{\rm int}.
\]
These equivalences satisfy
\[
T_{\pi,\pi}^{(m)}=\id,
\qquad
T_{\pi',\pi''}^{(m)}\circ T_{\pi,\pi'}^{(m)}
\simeq
T_{\pi,\pi''}^{(m)}.
\]
\end{prop}
\begin{proof} This is just a direct consequence of the construction. We spell this out for the convenience of the reader. 
Put
\[
u_\pi=\frac{\pi^{p^m}}p,
\qquad
u_{\pi'}=\frac{(\pi')^{p^m}}p,
\]
and choose $a_\pi,a_{\pi'}\in A_{\inf}$ 
with $\theta(a_\pi)=u_\pi,
\theta(a_{\pi'})=u_{\pi'}$. 
Working over $A_{\inf}$, Corollary~\ref{cor:main_cor} gives the equivalence  $\Phi_{pa_\pi,pa_{\pi'}}$ (with the obvious choice of the notation). 
If we specialize it along $\theta$, and apply the Kato base-change
equivalences, and repeat  the same \v{C}ech descent as in the
proof of Theorem~\ref{thm.integralHK}, we get the equivalence
$T_{\pi,\pi'}^{(m)}$.
The cocycle identity
\[
\Phi_{0,pa_{\pi'}}
=
\Phi_{pa_\pi,pa_{\pi'}}\circ\Phi_{0,pa_\pi}
\]
translates to give
\[
\rho_{\pi',m}^{\rm int}
\simeq
T_{\pi,\pi'}^{(m)}\circ\rho_{\pi,m}^{\rm int}.
\]
If $a_\pi'$ is another lift of $u_\pi$, then the specialization of
$\Phi_{pa_\pi,pa_\pi'}$ is the identity by
Corollary~\ref{cor:main_cor}(2). The same argument applies to the
$\pi'$-endpoint. Hence $T_{\pi,\pi'}^{(m)}$ is independent of all
the other choices. The identity and cocycle assertions follow from the strict
cocycle identities for the equivalence over $A_{\inf}$ and the transitivity of
Kato base change.
\end{proof} 
\section{The monodromy operator on de Rham cohomology and examples}

Let $k$ be a perfect field with $\ch(k)=p>0$ and consider the pre-log rings
\[ W^0=(W(k), \N\ni1\mapsto0)\qaq\Wp=(W(k), \N\ni1\mapsto p)\qaq
 \klog=(k, \N\ni1\mapsto0).\]
In this Section, we discuss some integral phenomena that can be captured on the relative log de Rham cohomology $R\Gamma_{\rm dR}(\fX/\Wp)$ of
a log smooth scheme $\fX$ of Cartier type over $\Wp$
 by transporting the structure from
 $R\Gamma_{crys}(\sX/W^0)$ with $\sX=\fX\otimes_{W^0} \klog$.


\begin{const}\label{ConstructionMonodromy}
We begin by recalling some material from \cite{HK} and \cite{tsuji_Faltings},
explaining the construction of the monodromy operator $N$ on log crystalline cohomology. 
Consider the pre-log rings
\[ \fS=(\Wu,\N\ni 1\mapsto u)\qaq
\fSFt=(W[[u_1,u_2]],\N^2\ni (n_1,n_2) \mapsto u_1^{n_1} u_2^{n_2}).\]

Note that $\fSFt$ is the $p$-completed self coproduct of $\fS$ over $W$ with the trivial log structure. Let $p_1,p_2: \fS \to \fSFt$ be the two co-projections (the two coordinate inclusions on underlying rings).

Consider a map $\xi: \fS\to \Wp$ of pre-log rings whose map on underling rings maps $u$ to $p$ and the map on monoids is the identity on $\N$.
Let $\sD=(D, \N\ni 1\to u)$ be the $p$-adic completion of the log PD envelope of $\fS$ with respect to $\Ker(\xi)$ (cf. \cite[(2.16)]{HK}).
The map $\xi$ lifts to a map $\xi_\sD: \sD\to \Wp$.
We view $\sD$ as a PD-thickening of $\klog$ via the composite of $\xi_\sD$ and 
the quotient map $\Wp\to \klog$.

Next, we consider a map $\xiFt: \fSFt\to \Wp$ of pre-log rings whose map on underling rings maps $u_i$ to $p$ ($i=1,2$) and the map on monoids is $\N^2\to \N;(n_1,n_2)\mapsto n_1+n_2$. 
Let $\sDFt=(\DFt, (\N^2\to \N)^{\mathrm{ex}})$ be the $p$-adic completion of the log PD envelope of $\fSFt$ with respect to $\Ker(\xiFt)$ (cf. \cite[(5.3)]{katolog} and \cite[(2.16)]{HK}). 
The maps $p_1,p_2: \fS \to \fSFt$ induce $p_1,p_2: \sD \to \sDFt$.
By \cite[(6.5)]{katolog}, we have $D$-algebra isomorphism
\[\DFt \simeq D\{v\},\]
where $D\{v\}$ is the $p$-completed PD polynomial algebra over $D$ in $v=\frac{p_1(u)}{p_2(u)}-1$ (note $\frac{p_1(u)}{p_2(u)}\in (\DFt)^\times$ since $(n,-n)\in (\N^2\to \N)^{\mathrm{ex}}$, where $(-)^{\mathrm ex}$ denotes the exactification of a map of monoids).
Let $f: \DFt \to D$ be the $D$-linear map given by
\[ \DFt \simeq D\{v\}=\widehat{\bigoplus}_{n\geq 0} D\cdot v^{[n]} \to D\cdot v^{[1]} \simeq D.\]
Let $\theta\colon \sD\to W^0$ be the morphism of log PD bases defined by mapping $u^{[n]}$ to $0$ for all $n> 0$ (and given by the identity on the chart $\N$).
For a log smooth integral morphism $\sX=(X,M_X)\to \Spec(\klog)$, we consider the following $D$-linear morphism $\sN$
\[
\mathcal{N}\colon R\Gamma_{crys}(\sX/\sD) \longrightarrow \theta_* R\Gamma_{crys}(\sX/W^0)
\]
\begin{align*}
    R\Gamma_{crys}(\sX/\sD) \xrightarrow{p_1} &R\Gamma_{crys}(\sX/\sD)\widehat{\otimes}^\bL_{D} \DFt\simeq  R\Gamma_{crys}(\sX/\sDFt) \\
    & \simeq R\Gamma_{crys}(\sX/\sD)\widehat{\otimes}^\bL_{D} \DFt \xrightarrow{f} R\Gamma_{crys}(\sX/\sD) \xrightarrow{q_\theta} \theta_* R\Gamma_{crys}(\sX/W^0)
\end{align*}
where the first (resp. second) equivalence is the base change along $p_1$ (resp. $p_2$)
which is an equivalence by \cite[(6.10)]{katolog},  and  $q_\theta$ is obtained by adjunction from the equivalence  
\[  R\Gamma_{crys}(\sX/\sD) \otimes^\bL_{D,\theta} W\simeq R\Gamma_{crys}(\sX/W^0)\]
provided by \cite[(6.10)]{katolog}. Applying again adjuction and Kato's base-change formula, we obtain then a map
\[ N\colon R\Gamma_{crys}(\sX/\sD) \otimes^\bL_{D,\theta} W \simeq R\Gamma_{crys}(\sX/W^0) \to R\Gamma_{crys}(\sX/W^0)\]
which is $W$-linear.  
On the cohomology, this operator coincides with the monodromy operator defined in \cite{HK} (see. \cite[(6.7)]{katolog}, \cite[\S4.3]{tsuji_Faltings} and \cite[6.40 and 6.45]{Du-Moon-Shimizu}). Thus, we get an additional structure on $R\Gamma_{crys}(\sX/W^0)$ as an object of $D^{\phi,N}(W)$ (see Definition \ref{def;DphiN} below),
where the action of $\phi$ is induced by the relative Frobenius of $\sX$ over $\uk^0$,
which is an isogeny by \cite[(2.24)]{HK}.
\end{const}

\begin{defn}\label{def;DphiN}
A $(\phi,N)$-module over $W$ is a triple $(M, \phi_M,N_M)$, where $M$ a $W$-module,  
$\phi_M: M\to M$ a $\phi$-semilinear isogeny (i.e. $\phi_M\otimes_W K_0$ is an automorphism), and $N_M: M\to M$   is a $W$-linear endomorphism such that 
$N_M \phi_M = p\phi_M N_M$. 
Write $D^{\phi,N}(W)$ for the  derived category
of $(\phi,N)$-modules over $W$. See e.g. \cite{deg_niziol} for details. 
\end{defn}
\begin{rmk}
	After inverting $p$, a $\phi$-semilinear isogeny is equivalently an  isomorphism $\phi^*M_{K_0}\to M_{K_0}$. With the convention $M(-1)=(M,p \phi_M)$,  the relation $N \phi=p \phi N$ says exactly that $N\colon M\to M(-1)$ is $\phi$-equivariant. 
\end{rmk}

Let $K$ be a complete discrete valuation field of mixed characteristic $(0,p)$ with the valuation ring $\sO_K$ and residue field $k$ and $W=W(k)$ and $K_0=W[1/p]$.
Let $\lSmCar_{\OK}$ be the category of log smooth schemes of Cartier type over $\OKpi=(\OK,\N\to \OK; 1\to \pi)$.

\begin{thm}
Let $\lSmCarpr_{\OK}$ be the full subcategory of proper objects in
$\lSmCar_{\OK}$, and let $m\geq0$ satisfy $e\leq p^m$. For
$\fX\in\lSmCarpr_{\OK}$, put
\[
\sX:=\fX\times_{\OKpi}\uk^0.
\] 
Then the functor
\[
R\Gamma_{\dR}^{\eta^m}(-/\OKpi)\colon
(\lSmCarpr_{\OK})^{\rm op}\longrightarrow\mathcal D(\OK)^\wedge_p
\]
is naturally equivalent to the composite
\[
(\lSmCarpr_{\OK})^{\rm op}
\xrightarrow{\ \fX\mapsto R\Gamma_{\rm crys}(\sX/W^0)\ }
D^{\phi,N}(W)
\xrightarrow{\ -\widehat\otimes^{\mathbb L}_{W,\phi^m}\OK\ }
\mathcal D(\OK)^\wedge_p.
\]
Here the $(\phi,N)$-structure on crystalline cohomology is the one of
Construction~\ref{ConstructionMonodromy}.
\end{thm}
\begin{proof}
Construction~\ref{ConstructionMonodromy} provides the functor to
$D^{\phi,N}(W)$, and Corollary~\ref{cor.integralHK}(2) identifies its
scalar extension with the sheafwise complex
$R\Gamma_{\dR}^{\eta^m}(-/\OKpi)$.
\end{proof}
 

\begin{ex}\label{ex:enriques}
   In this example we work over the ring of Witt vectors $W=W(k)$, where $k$ is an algebraically closed field of characteristic $p=2$. Our goal is to construct a projective, vertical, log smooth scheme $\fX$ over $W^\times$ and two non trivial $2$-torsion classes, one constructed in log crystalline cohomology over $W^0$ of $\sX=\fX\times _{W^\times} k^0$ and another in relative log de Rham cohomology of $\fX$ over $W^\times$. Our main Theorem shows that, in particolar, the torsion subgroups of these two groups are isomorphic (so, this is an independent sanity check -- the classes should match by construction, although we do not verify this directly). Note that the constructed class is annihilated by the integral monodromy operator.  

        Choose a classical Enriques surface $V$ over $k$ whose K3-cover is normal. Such surfaces exist: they form an open dense locus in the moduli space of classical Enriques surfaces in characteristic $2$, see \cite[Remark~1.3.9]{EnriquesBook}. Also, we have \cite[Corollary~1.4.9]{EnriquesBook},
        \[
        H^0(V,\mathcal T_V)=0,
        \]
         the canonical sheaf satisfies $\omega_V\not\simeq\mathcal O_V$ and $\omega_V^{\otimes 2}\simeq\mathcal O_V$. By \cite[Theorem~1.4.13]{EnriquesBook}, $H^2_{\rm crys}(V/W)\simeq W^{\oplus 10}\oplus k$:
the free rank is $b_2(V)=10$, and the torsion submodule is thus a one-dimensional
$k$-vector space.
By \cite[Corollary~1.4.14]{EnriquesBook} this torsion is \emph{divisorial} in the
sense of Illusie \cite{IllusiedRW}: it is generated by crystalline Chern classes of
numerically trivial line bundles. Since $\operatorname{Pic}^\tau(V)=\Z/2$ is generated
by $\omega_V$, the class $\alpha_V:=c_1^{\rm crys}(\omega_V)$ must be non-zero and $\alpha_V$ generates  the torsion submodule of
$H^2_{\rm crys}(V/W)$.

        Fix a very ample divisor $H$ on $V$. For $n$ sufficiently large, choose a smooth connected curve $   C\in |nH|$
        such that $H^1(V,\mathcal O_V(C))=0$. Put
        \[
        \mathcal L:=\omega_V|_C,
        \qquad
        \mathcal N:=\mathcal N_{C/V}=\mathcal O_C(C).
        \]
        The line bundle $\mathcal L$ is a non-trivial $2$-torsion line bundle of degree zero. Indeed, if $\mathcal L$ were trivial, then the exact sequence
        \[
        0\longrightarrow \omega_V(-C)\longrightarrow\omega_V
        \longrightarrow\omega_V|_C\longrightarrow0
        \]
        and Serre duality
        \[
        H^1(V,\omega_V(-C))^\vee\simeq H^1(V,\mathcal O_V(C))=0
        \]
        would imply that $H^0(V,\omega_V)\to H^0(C,\omega_V|_C)$ is surjective. This is impossible because $H^0(V,\omega_V)=0$, whereas $H^0(C,\mathcal O_C)=k$.
Next, consider the  surface
\[
\pi\colon S=\mathbb P(\mathcal O_C\oplus\mathcal N)\longrightarrow C
\]
and the section $C\subset S$ corresponding to the quotient
$\mathcal O_C\oplus\mathcal N\twoheadrightarrow\mathcal O_C$.   This embeds $C$ in $S$ with normal bundle $\mathcal{N}_{C/S}\simeq\mathcal{N}^{\vee}$.   

       After this operation, consider the projective scheme $Y$ obtained as a glueing of $S$ and $V$ along the common $C$ (this is essentially the same as the double construction along an effective Cartier divisor discussed in \cite{BindaKrishna}). By construction, this is in fact a normal crossing scheme, of dimension $2$ over $k$, with exactly 2 smooth components meeting transversally in $C$, since
\[
\mathcal N_{C/V}\otimes\mathcal N_{C/S}
\simeq
\mathcal N\otimes\mathcal N^\vee
\simeq\mathcal O_C. 
\]   It is in fact a $d$-semistable (or ''degenerate'' semistable) projective variety in the sense of Friedman, see \cite[Section 1]{KawNam}.  Projectivity can be checked directly: for $m\gg0$ the ample line bundles
$\mathcal O_V(mC)$ and $\mathcal O_S(1)\otimes\pi^*\mathcal N^{m}$ have
isomorphic restrictions $\mathcal N^{m}$ to $C$, hence glue to a line bundle on
$Y$ which is ample, being ample on each component. 

    We can now equip $Y$ with the standard log structure of a normal crossing variety, so that it is integral (this is automatic), and log smooth over the standard log point $\klog$ (with  chart   étale-locally given by $\N^2\to k[x,y,z]/xy$). We remark that the existence of a global log structure on $Y$ is \cite[Proposition 1.1]{KawNam} in characteristic zero, or \cite[Theorem 11.7]{FKato_Def} in positive characteristic.  
    
    We next show that it admits a projective log smooth lifting over $W(k)$.
        Let $\nu:V\amalg S\to Y$ be the normalization. There is an exact sequence
        \[
        0\longrightarrow\mathcal T_Y^0\longrightarrow
        \nu_*\bigl(\mathcal T_V(-\log C)\oplus\mathcal T_S(-\log C)\bigr)
        \longrightarrow\mathcal T_C\longrightarrow0.
        \]
        We first claim that
        \[
        H^2(V,\mathcal T_V(-\log C))=0.
        \]
        By Serre duality, its dual is
        \[
        H^0\bigl(V,\Omega_V^1(\log C)\otimes\omega_V\bigr).
        \]
        The residue sequence gives
        \[
        0\longrightarrow\Omega_V^1\otimes\omega_V
        \longrightarrow\Omega_V^1(\log C)\otimes\omega_V
        \longrightarrow\omega_V|_C\longrightarrow0.
        \]
        Since $\omega_V^{\otimes2}\simeq\mathcal O_V$, we have
        $\Omega_V^1\otimes\omega_V\simeq\mathcal T_V$, and therefore
        \[
        H^0(V,\Omega_V^1\otimes\omega_V)=0.
        \]
        Also $H^0(C,\omega_V|_C)=H^0(C,\mathcal L)=0$, because $\mathcal L$ is a non-trivial line bundle of degree zero. This proves the claim.
    For the ruled component $S$, the morphism of pairs $(S,C)\to C$ gives an exact sequence
        \[
        0\longrightarrow\mathcal T_{S/C}(-C)
        \longrightarrow\mathcal T_S(-\log C)
        \longrightarrow\pi^*\mathcal T_C\longrightarrow0.
        \]
        On every fibre of $\pi$, the first term restricts to $\mathcal O_{\mathbb P^1}(1)$. Hence
        \[
        R^1\pi_*\mathcal T_{S/C}(-C)=0,
        \qquad
        R^1\pi_*\mathcal O_S=0.
        \]
        It follows that
        \[
        H^2(S,\mathcal T_S(-\log C))=0
        \]
        and that the restriction map
        \[
        H^1(S,\mathcal T_S(-\log C))
        \longrightarrow H^1(C,\mathcal T_C)
        \]
        is surjective. The above normalization sequence therefore gives $   H^2(Y,\mathcal T_Y^0)=0.$ 
Let
\[
\sX=(Y,\mathcal M_Y)\longrightarrow\klog
\]
be the resulting log-smooth morphism. Following
\cite[Definition~5.1]{FKato_Def}, let
\[
\mathcal T^{\log}_{\sX/\klog} = \operatorname{Der}_{\klog}(\sX,\mathcal O_Y)
\]
denote the sheaf of logarithmic derivations of $\sX$ over
$\klog$. By \cite[Proposition~5.2]{FKato_Def}, there is a natural
identification
\[
\mathcal T^{\log}_{\sX/\klog} = \operatorname{Der}_{\klog}(\sX,\mathcal O_Y)
\simeq
\mathcal Hom_{\mathcal O_Y}
\left(
\Omega^1_{\sX/\klog},
\mathcal O_Y
\right).
\]To apply F.~Kato's log deformation theory, it remains to prove
\[
H^2\left(
Y,
\operatorname{Der}_{\klog}(\sX,\mathcal O_Y)
\right)=0.
\]
        The chosen $d$-semistable trivialization also gives an exact sequence  (this can be checked locally using coordinates to describe the sheaf  of log differentials, see \cite{KawNam})
\begin{equation}\label{eq:log-tangent-T0}
0\longrightarrow
\mathcal T^{\log}_{\sX/\klog}
\longrightarrow
\mathcal T_Y^0
\xrightarrow{\ \lambda\ } i_*O_C
\longrightarrow0
\end{equation}
for $i\colon C\to Y$ the inclusion.  Since the obstruction to lifting $\sX$ log smoothly lies in
$H^2(Y,\mathcal T^{\log}_{\sX})$  by \cite[Proposition~8.6]{FKato_Def}, 
it remains to check that
\[
H^1(Y,\mathcal T_Y^0)
\longrightarrow
H^1(C,\mathcal O_C)
\]
is surjective. This can be reduced as before looking at the normalization, and using this time the component $S$ (we leave the details to the reader).  

This gives  a compatible system of log smooth liftings over
$W_n^\times$, and hence a proper formal log smooth lifting
$\widehat{\fX} \to\Spf(W^\times)$ of $\sX$.
    To algebraize it, observe that $ H^2(Y,\mathcal O_Y)=0$. 
        Indeed, the normalization sequence
        \[
        0\longrightarrow\mathcal O_Y\longrightarrow
        \mathcal O_V\oplus\mathcal O_S\longrightarrow\mathcal O_C\longrightarrow0
        \]
        together with $H^1(V,\mathcal O_V)=H^2(V,\mathcal O_V)=0$ and
        $R\pi_*\mathcal O_S\simeq\mathcal O_C$ shows that the map
        $H^1(S,\mathcal O_S)\to H^1(C,\mathcal O_C)$ is an isomorphism and hence that the displayed group vanishes. As a consequence, the ample line bundle constructed above lifts through the formal deformation. By Grothendieck existence theorem, $\widehat{\fX} $ algebraizes to a projective, vertical, log smooth scheme $\fX/W^\times$
        with special fibre $\sX$.

        We now construct the torsion class. The line bundles $\omega_V$ on $V$ and $\pi^*\mathcal L$ on $S$ have canonically isomorphic restrictions to $C$. They therefore glue to a line bundle $\mathcal M$ on $Y$ satisfying
        \[
        \mathcal M|_V\simeq\omega_V,
        \qquad
        \mathcal M|_S\simeq\pi^*\mathcal L.
        \]
        After rescaling one of the chosen trivializations of the squares, they agree on $C$, and hence $ \mathcal M^{\otimes2}\simeq\mathcal O_Y$.
        Let
        \[
        \alpha:=c_1^{\log\text{-}\rm crys}(\mathcal M)
        \in H^2_{\log\text{-}\rm crys}(\sX/W^0).
        \]
        Then $2\alpha=0$.
    It remains to check that $\alpha$ is non-zero. For the strict normal crossing surface $Y=V\cup_C S$, the weight spectral sequence of Nakkajima \cite{nakkajimaTokyo} is
        \[
        E_1^{-k,q+k}
        =
        \bigoplus_{j\geq\max(0,-k)}
        H_{\rm crys}^{q-2j-k}
        \bigl(Y^{(2j+k+1)}/W\bigr)(-j-k)
        \Longrightarrow
        H^q_{ \rm crys}(\sX/W^0).
        \]
        The image of $\alpha$ in
        \[
        E_1^{0,2}=H^2_{\rm crys}(V/W)\oplus H^2_{\rm crys}(S/W)
        \]
        is
        \[
        \alpha_0=(c_1^{\rm crys}(\omega_V),0).
        \]
        Here $c_1^{\rm crys}(\pi^*\mathcal L)=0$, because it is $2$-torsion whereas $H^2_{\rm crys}(S/W)$ is $W$-torsion free by the projective bundle formula.
        By \cite[Theorem~10.1]{nakkajimaTokyo}, the outgoing differential
        \[
        d_1:E_1^{0,2}\longrightarrow E_1^{1,2}=H^2_{\rm crys}(C/W)
        \]
        is the alternating restriction map. It kills $\alpha_0$, since
        $c_1^{\rm crys}(\mathcal L)=0$ in the torsion-free module
        $H^2_{\rm crys}(C/W)$. The incoming differential
        \[
        d_1:E_1^{-1,2}=H^0_{\rm crys}(C/W)(-1)
        \longrightarrow E_1^{0,2}
        \]
        is the alternating Gysin map. The $S$-component of the image of its generator is, up to sign,
        $c_1^{\rm crys}(\mathcal O_S(C))$, which is non-zero because $\mathcal O_S(C)$ restricts to $\mathcal O_{\mathbb P^1}(1)$ on every fibre of $\pi$, and which belongs to the torsion-free module $H^2_{\rm crys}(S/W)$. Hence $\alpha_0$ is not a boundary. Since $Y$ has only two irreducible components, there are no columns with $|k|\geq2$, so no higher differential can enter or leave $E_r^{0,2}$. Thus $\alpha_0$ gives a non-zero class on the $E_\infty$-page. Compatibility of crystalline Chern classes with the weight complex shows that this is the associated graded class of $\alpha$. Hence,
        \[
        0\neq\alpha\in H^2_{\rm crys}(\sX/W^0)[2].
        \]
 By Corollary~\ref{cor.integralHK}(1), under the integral comparison
        \[
        H^2_{\log\text{-}\rm crys}(\sX/W^0)
        \simeq H^2_{\rm dR}(\fX/W^\times)
        \]
        the class $\alpha$ gives therefore a 2-torsion class in de Rham cohomology.
        Finally, $H^2(Y,\mathcal O_Y)=0$ also implies that $\mathcal M$ lifts to a line bundle $\widetilde{\mathcal M}$ on $\fX$. Note that $c_1^{\rm dR}(\widetilde{\mathcal M})$   can be seen  to be a (non-trivial) 2-torsion class by direct inspection, reducing modulo 2 (note that we do not verify
here that the transferred class $\alpha$ agrees with
$c_1^{\rm dR}(\widetilde{\mathcal M})$, although it should be possible after proving a compatibility between the transfer map and $c_1(-)$).
        
        The equality $N(\alpha)=0$ can also be checked directly from Construction~\ref{ConstructionMonodromy}: the crystalline Chern class of $\mathcal M$ over the auxiliary PD base $\sD$ has equal pullbacks along $p_1$ and $p_2$, by base-change functoriality of Chern classes. Hence its expansion with respect to the $p_2$-identification has zero coefficient of $v^{[1]}$, and therefore
        \[
        N(\alpha)=0.
        \]
        We have thus obtained a non-zero integral torsion invariant cycle in
        $H^2_{\rm dR}(\fX/W^\times)$, which disappears after inverting $2$.
        \end{ex}

\begin{ex}[Tate curves]\label{ex:Tate_ramified} In the following example we show that in the ramified setting, the integral Hyodo--Kato isomorphism $\Phi_{\pi}^{\rm int}$  from Corollary \ref{cor.integralHK} identifies a lattice inside the de Rham cohomology of the generic fiber of a formal semistable family which is genuinely different from the standard lattice coming from the log de Rham cohomology of the family. All the computations are standard.

To see this, let us consider the field $K = \Q_p(\pi)$ where $\pi^p=p$. We have $\OK = \Z_p[\pi]$ with residue field $\F_p$. Let $E = E_p$ be the Tate curve of parameter $q=p$ defined over $K$, that is the rigid analytic torus $\G_{m,K}/q^\Z \to \mathrm{Spa}(K)$. It admits a standard semistable formal model over $\OK$ given by the proper formal scheme $\mathfrak E$, obtained by gluing $p$-charts with local equations $\Spf \OK\langle    v_i,w_i\rangle /(v_iw_i-\pi)$ ($0\leq i\leq p-1$), see e.g., \cite[Section 5]{ertl-yamada}. By \cite[Section 5]{ertl-yamada}, the gluing gives a proper log smooth log scheme of Cartier type $\mathfrak E \to \Spec(\OK^\pi)$, with special fiber 
\[ E_0 = C_0\cup \ldots C_{p-1}, \quad C_i \simeq \P^1_{\F_p}, \quad C_i\cap C_j = {x_{ij}}\]
  Each $C_i$ meets $C_j$ exactly in the point $x_{ij}$ transversally.  The log structure on $\mathfrak E$ is locally given by $\N^2\to \OK\langle    v_i,w_i\rangle /(v_iw_i-\pi)$, $(1,0)\mapsto v_i$ and $(0,1)\mapsto w_i$.

  To compute the Hyodo--Kato cohomology of $\mathcal{E} = (E_0, \mathcal{M}_{E}) \to \Spec(\F_p^0)$ over $\Z_p^0$ we use again the spectral sequence of Mokrane--Nakkajima, with the identifications for the cohomology of curves provided by Coleman--Iovita in \cite{coleman_iovita}. 

More precisely, let $\Gamma$ be the dual graph of $E_0$. By  \cite[Chapter~I, \S1, pp.~4--5]{coleman_iovita}, its vertices are the
irreducible components of $E_0$, and its oriented edges are the two
ordered branches over every node (so, they correspond to intersection points). Choosing one orientation of each edge, write
\[
v_i=[C_i],\qquad
e_i\colon v_i\longrightarrow v_{i+1}.
\]
So we have  the chain complex 
\[
C_1(\Gamma,\mathbb Z_p)
 =\bigoplus_{i\in\mathbb Z/p\mathbb Z}\mathbb Z_p e_i\xrightarrow{\partial} \bigoplus_{j\in \Z/p\Z} \Z_p v_j,
\qquad
\partial(e_i)=v_{i+1}-v_i.
\]
A generator for $H_1$ and the corrisponding class in $H^1$ class are then given by
\[
\gamma=e_0+\cdots+e_{p-1}
 \in H_1(\Gamma,\mathbb Z_p),
\qquad
a_\Gamma\in H^1(\Gamma,\mathbb Z_p),
\qquad
\langle a_\Gamma,\gamma\rangle=1.
\]
Let $M_{HK}=H^1_{\rm crys}(\mathcal{ E} /\Z_p^0)$. Using the weight spectral sequence as in Example \ref{ex:enriques} we obtain a short exact sequence of $\varphi$-modules over $\Z_p$
\[ 0\to H^1(\Gamma,\Z_p) \to M_{HK} \to H_1(\Gamma,\Z_p)(-1)\to 0.\]
This follows from the explicit description in \cite[Theorem 10.1]{nakkajima} as
\[E_1^{0,0}=H^0_{\rm crys}(Y^{(1)}/\Z_p) = \bigoplus v_j \Z_p, \quad E_1^{1,0} = H^0_{\rm crys}(Y^{(2)}/\Z_p) = \bigoplus e_i \Z_p,\]
and $E_1^{1,0}=\bigoplus H^1_{\rm crys}(\P^1/\Z_p)=0$. The differential $E_1^{0,0}\to E_1^{1,0}$ is identified with the boundary map in the cohomology of $\Gamma$, giving $E_2^{1,0}=H^1(\Gamma,\Z_p)$. Similarly, we have 
\[E_{1}^{-1,2} = H^0_{\rm crys}(Y^{(2)}/\Z_p)(-1) = \bigoplus e_i\Z_p(-1),\] 
\[E_1^{0,2}=H^2_{\rm crys}(Y^{(1)}/\Z_p) = \bigoplus H_{\rm crys}^2(C_i/\Z_p) \simeq \bigoplus \Z_p(-1)\]
and the differential is identified with $\partial$, giving $E_2^{-1,2}=H_1(\Gamma,\Z_p)(-1)$.  We let $a_{HK}$ be the image of $a_\Gamma$ in $M_{HK}$, generating a rank 1 submodule of $M_{HK}$. If we write $\varphi_M$ for the Frobenius on $M_{HK}$, note that $\varphi_M(a_{HK}) = a_{HK}$ (this follows again from the above computation, noting that Frobenius is the identity on $H^0_{\rm crys}(-/\Z_p)$ on $\P^1$ and on points). Similarly, note that on the quotient $H_1(\Gamma,\Z_p)(-1)$, the Frobenius acts as multiplication by $p$. 

Let now $b_0$ be any element in $M_{HK}$ lifting $\gamma \in H_1(\Gamma,\Z_p)$. By the previous discussion, we get $\varphi_M(b_0)-p b_0 \in H^1(\Gamma,\Z_p) = a_{HK}\Z_p$, so $\varphi_M(b_0)-p b_0  = c a_{HK}$ for some $c\in \Z_p$. Setting $b_{HK}=b_0+\frac{c}{p-1}a_{HK}$, we obtain another generator satisfying $\varphi_M(b_{HK})=p b_{HK}$. In conclusion, $M_{HK}$ is free of rank 2 as $\Z_p$-module on generators $(a_{HK},b_{HK})$ and with Frobenius matrix given by the identity on $a_{HK}$ and the multiplication by $p$ on $b_{HK}$ (note that this agrees with the computation of \cite[Proposition 5.1]{ertl-yamada} using the overconvergent Hyodo--Kato complex). 

We can now follow \cite{coleman_iovita} and \cite{nakkajima} to describe the monodromy operator. More precisely, by \cite[Chapter I, §3, pp. 17–18]{coleman_iovita}, there is a pairing $\langle-,-\rangle C_1(\Gamma)\times C_1(\Gamma)\to \Z$ (edge pairing) satisfying $\langle \gamma, \gamma\rangle =p$ and inducing a morphism 
\[\mu_\Gamma\colon H_1(\Gamma,\Z_p)\to H^1(\Gamma,\Z_p)\]
By \cite[Corollary 11.5 and Proposition 11.10]{nakkajima}, the map $\mu_\Gamma$ is identified with the crystalline monodromy operator. Using it, we obtain
\[ N(a_{HK})=0, \quad N(b_{HK})=pa_{HK},\]
in agreement with \cite[Proposition 5.1(1)]{ertl-yamada} for $r=p$. 

We can now apply the integral transport theorem. Since $e=p$, we can apply  Theorem \ref{thm.integralHK} for the chosen uniformizer $\pi$ with $m=1$. Since the Frobenius $\phi$ of $\Z_p$ is the identity, we obtain an isomorphism of $\OK$-modules:
\begin{equation}\label{eq_integral_tate}
    \Lambda_{HK} =M_{HK}\otimes_{\Z_p} \OK=H^1_{\rm crys}(\mathcal{E}/\Z_p^0) \otimes_{\Z_p} \OK \xrightarrow{\Phi_{\pi, \mathfrak E}^{\rm int}} H^1(R\Gamma(\mathfrak E  , L\eta_p \widehat{\omega}^*_{\mathfrak{E} / \OK^\pi}))=:\Lambda_\eta\end{equation}
Writing $a_{dR}$ and $b_{dR}$ for the images of $a_{HK}$ and $b_{HK}$ respecively, they satisfy $N(a_{dR})=0$, $N(b_{dr})=pa_{dR}$. In particular $N\neq0$,  $N=0\pmod p $ and $\mathrm{coker}(N) = \OK\oplus \OK/p$. So, $H^1$ is torsion-free, but  the cokernel of the integral monodromy has nonzero $p$-torsion.

For comparison, consider
\[
\Lambda_{\rm st}
:=
H^1\!\left(
R\Gamma\!\left(
\mathfrak{E},
\widehat\omega^*_{\mathfrak E/\OKpi}
\right)
\right)
\]
i.e., the standard  $\OK$-lattice given by the de Rham cohomology of the semistable model. Let
$a_{\rm dR},\omega_{\rm dR}\in H^1_{\rm dR}(E/K)$ be the classes
represented in the ordered \v{C}ech complex of
\cite[\S5]{ertl-yamada} by
\[
(0,\ldots,0,1)
\qquad\text{and}\qquad
(d\log w_0,\ldots,d\log w_{p-1}),
\]
respectively. The direct computation following \cite{ertl-yamada} gives
\begin{equation}\label{eq:tate-three-lattices}
\Lambda_{\rm st}
=
\OK a_{\rm dR}\oplus\OK\omega_{\rm dR},
\qquad
\Lambda_\eta
=
\OK a_{\rm dR}\oplus p\OK\omega_{\rm dR}.
\end{equation}

We can also describe the lattice expected from the classical rational
Hyodo--Kato map. 
Let
\[
\phi_{\rm lin}\colon\phi^*M_{\rm HK}[1/p]\longrightarrow M_{\rm HK}[1/p]
\]
be the linearization of Frobenius $\varphi_M$ on $M_{\rm HK}$.
 Thus
\[
\phi_{\rm lin}(a_{\rm HK} )=a_{\rm HK},
\qquad
\phi_{\rm lin}(b_{\rm HK} )=p\,b_{\rm HK}.
\]
Choose the branch $\log_\pi$ of the $p$-adic logarithm for which
$\log_\pi(\pi)=0$. By
\cite[Proposition~5.1 and Proposition~6.13]{ertl-yamada}, the classical rational
Hyodo--Kato map 
\[\Psi_\pi^{\rm crys}: \Lambda_{\rm HK}[1/p] \to \Lambda_{\rm st}[1/p] \] 
satisfies
\[
\Psi_\pi^{\rm crys}(a_{\rm HK})=a_{\rm dR},
\qquad
\Psi_\pi^{\rm crys}(b_{\rm HK})=\omega_{\rm dR}.
\]
 Hence, after inverting $p$, our integral comparison map $\Phi_{\pi, \mathfrak E}^{\rm int}$ from \eqref{eq_integral_tate} is identified with $ \Psi_\pi^{\rm crys}\circ (\phi_{\rm lin}\otimes_{\Z_p} \OK)$ and 
\[\Lambda_\eta=\OK a_{\rm dR}\oplus p\OK\omega_{\rm dR}
=
\bigl(\Psi_\pi^{\rm crys}\circ (\phi_{\rm lin}\otimes_{\Z_p} \OK) \bigr)
 \bigl(\Lambda_{\rm HK}\bigr)\]
 and its $W$-submodule $W a_{\rm dR}\oplus pW\omega_{\rm dR}$ is equipped with a natural action of Frobenius and monodromy operator.
 This explains why  $L\eta_p$ is essential in this example. Indeed,
\begin{equation}\label{eq:tate-lattice-index}
\Lambda_{\rm st}/\Lambda_\eta
\simeq
\OK/p
=
\OK/\pi^p,
\qquad
\operatorname{length}_{\OK}(\Lambda_{\rm st}/\Lambda_\eta)=p, 
\end{equation}
so if we replace the right-hand side of \eqref{eq_integral_tate} with the ordinary (untwisted) de Rham lattice $\Lambda_{\rm st}$,  
$ \Psi_\pi^{\rm crys}\circ (\phi_{\rm lin}\otimes_{\Z_p} \OK)$ has its image only
$\Lambda_\eta\subsetneq\Lambda_{\rm st}$   
with cokernel $\OK/p$, preventing the map from being an isomorphism. 
After inverting $p$, this distinction disappears, as all terms coincide with the de Rham cohomology of $E$ over $K$:
\[
\Lambda_{\rm st}[1/p]
=
\Lambda_\eta[1/p]
=
H^1_{\rm dR}(E/K).
\]
If one sets $c_{\rm HK}:=\frac{1}{p}b_{\rm HK},$ 
then rationally
\[
\phi(a_{\rm HK})=a_{\rm HK},
\qquad
\phi(c_{\rm HK})=p\,c_{\rm HK},
\qquad
N(a_{\rm HK})=0,
\qquad
N(c_{\rm HK})=a_{\rm HK}.
\]
In particular, the $p$-divisibility of  the integral monodromy matrix (and thus the length in
\eqref{eq:tate-lattice-index}) are all invisible in the rational
$(\phi,N)$-module.
\end{ex}

\bibliographystyle{alpha}
\bibliography{bib}

@article{AbhinandanYoucis,
    author = {Abhinandan and Youcis, Alex}, 
    title = {An integral comparison of crystalline and de {R}ham cohomology},
    note = {ArXiv preprint: 2507.17631 \url{https://arxiv.org/abs/2507.17631}},
    year = {2025}
}

@article {AGV,
    AUTHOR = {Ayoub, Joseph and Gallauer, Martin and Vezzani, Alberto},
     TITLE = {The six-functor formalism for rigid analytic motives},
   JOURNAL = {Forum Math. Sigma},
  FJOURNAL = {Forum of Mathematics. Sigma},
    VOLUME = {10},
      YEAR = {2022},
     PAGES = {Paper No. e61, 182},
      ISSN = {2050-5094},
   MRCLASS = {14F42 (14C15 14G22 18N60)},
  MRNUMBER = {4466640},
MRREVIEWER = {Matthias\ Wendt},
       DOI = {10.1017/fms.2022.55},
       URL = {https://doi-org.pros2.lib.unimi.it/10.1017/fms.2022.55},
}

@article {coleman_iovita,
    AUTHOR = {Coleman, Robert and Iovita, Adrian},
     TITLE = {The {F}robenius and monodromy operators for curves and abelian
              varieties},
   JOURNAL = {Duke Math. J.},
  FJOURNAL = {Duke Mathematical Journal},
    VOLUME = {97},
      YEAR = {1999},
    NUMBER = {1},
     PAGES = {171--215},
      ISSN = {0012-7094,1547-7398},
   MRCLASS = {14F30 (14F40 14K15)},
  MRNUMBER = {1682268},
MRREVIEWER = {Bruno\ Chiarellotto},
       DOI = {10.1215/S0012-7094-99-09708-9},
       URL = {https://doi-org.pros2.lib.unimi.it/10.1215/S0012-7094-99-09708-9},
}

@article{Revisiting_HK_part2,
    author = {Binda, Federico and Saito, Shuji},
    title = {Hyodo--{K}ato Transfer through Twisted Prismatization },
    year = {2026},
    note = {In preparation}
}

@article{BMS2,
author= {Bhatt, Bhargav and Morrow, Matthew and Scholze, Peter},
title = {Topological Hochschild Homology and Integral p-adic Hodge theory},
fjournal = {Institut des Hautes \'{E}tudes Scientifiques. Publications Math\'{e}matiques},
journal = {Inst. Hautes \'{E}tudes Sci. Publ. Math.},
volume = {129},
year = {2019},
pages = {199--310}
}

@book{BLM,
    AUTHOR = {Bhatt, Bhargav and Lurie, Jacob and Mathew, Akhil},
    TITLE = {Revisiting the de {R}ham--{W}itt complex},
    Series = {Ast\'{e}risque},
    FSeries = {Ast\'{e}risque},
    Volume = {424},
    Year = {2021},
    Pages = {viii+165},
    Publisher = {Paris: Soci{\'e}t{\'e} Math{\'e}matique de France (SMF)}}

@misc{BLMP,
    author = {Binda, Federico and Lundemo, Tommy and Merici, Alberto and Park, Doosung},
    title = {Logarithmic prismatic cohomology, motivic sheaves, and comparison theorems},
    note = {To appear in Crelle. ArXiv preprint \url{https://arxiv.org/abs/2312.13129}},
    year = {2026}
}

@article{BLPO-HKR,
title = {A {H}ochschild-{K}ostant-{R}osenberg theorem and residue sequences for logarithmic {H}ochschild homology},
journal = {Advances in Math.},
volume = {435},
pages = {109354},
year = {2023},
issn = {0001-8708},
doi = {https://doi.org/10.1016/j.aim.2023.109354},
url = {https://www.sciencedirect.com/science/article/pii/S0001870823004978},
author={Binda, Federico and Lundemo, Tommy and Park, Doosung and {\O}stv{\ae}r, Paul Arne},
}

@article{Cesnavicius-Koshikawa,
    author = {{\v C}esunavi{\v c}ius, K{\c e}stutis and Koshikawa, Teruhisa},
    title = {The $\mathrm{A}_{\inf}$-cohomology in the semistable case},
    journal = {Compos.Math.},
    volume = {155},
    year = {2019},
    pages = {2039-2128}
}

@article{IllusiedRW,
	author = {Illusie, Luc},
	title = {Complexe de {D}e {R}ham--{W}itt et cohomologie cristalline},
	fjournal = {Annales Scientifiques de l’{\'E}cole Normale Sup{\'e}rieure.
Quatri{\`e}me S{\'e}rie.},
	journal = {Ann. scient. Ec. Norm. Sup. (4)},
	volume = {12},
	year = {1979},
	pages = {501--661}
}

@article {BMS1,
    AUTHOR = {Bhatt, Bhargav and Morrow, Matthew and Scholze, Peter},
     TITLE = {Integral {$p$}-adic {H}odge theory},
   JOURNAL = {Publ. Math. Inst. Hautes \'{E}tudes Sci.},
  FJOURNAL = {Publications Math\'{e}matiques. Institut de Hautes \'{E}tudes
              Scientifiques},
    VOLUME = {128},
      YEAR = {2018},
     PAGES = {219--397},
      ISSN = {0073-8301,1618-1913},
}

@article {BeilinsonPeriod,
    AUTHOR = {Beilinson, A.},
     TITLE = {On the crystalline period map},
   JOURNAL = {Camb. J. Math.},
  FJOURNAL = {Cambridge Journal of Mathematics},
    VOLUME = {1},
      YEAR = {2013},
    NUMBER = {1},
     PAGES = {1--51},
      ISSN = {2168-0930,2168-0949},
   MRCLASS = {14F30 (14F20 14F40)},
  MRNUMBER = {3272051},
MRREVIEWER = {Nobuo\ Tsuzuki},
       DOI = {10.4310/CJM.2013.v1.n1.a1},
       URL = {https://doi.org/10.4310/CJM.2013.v1.n1.a1},
}

@article {nakkajimaTokyo,
    AUTHOR = {Nakkajima, Yukiyoshi},
     TITLE = {{$p$}-adic weight spectral sequences of log varieties},
   JOURNAL = {J. Math. Sci. Univ. Tokyo},
  FJOURNAL = {The University of Tokyo. Journal of Mathematical Sciences},
    VOLUME = {12},
      YEAR = {2005},
    NUMBER = {4},
     PAGES = {513--661},
      ISSN = {1340-5705},
   MRCLASS = {14F30 (14G22)},
  MRNUMBER = {2206357},
MRREVIEWER = {Martin\ C.\ Olsson},
}

@article{HK,
	title = {Semi-stable reduction and cystalline cohomology with logarithmic poles},
	author = {Hyodo, Osamu and Kato, Kazuya},
	journal = {Ast\'erisque}, 
	volume = {223},
	year = {1994}
}

@article{Ya,
	title={LOGARITHMIC DE {R}HAM–{W}ITT COMPLEXES VIA THE D{\'E}CALAGE OPERATOR},
	author={Zijian Yao},
	journal={Journal of the Institute of Mathematics of Jussieu},
	year={2021}
}

@unpublished{achingeretal,
    author = {Achinger, Piotr and H\"ubner, Katharina and Lara, Marcin and Stix, Jakob},
    title = {Logarithmic geometry beyond fs},
    note = {ArXiv preprint: \url{https://arxiv.org/pdf/2411.13662}},
}

@article {CN,
    AUTHOR = {Colmez, Pierre and Nizio\l, Wies\l awa},
     TITLE = {On {$p$}-adic comparison theorems for rigid analytic
              varieties, {I}},
   JOURNAL = {M\"unster J. Math.},
  FJOURNAL = {M\"unster Journal of Mathematics},
    VOLUME = {13},
      YEAR = {2020},
    NUMBER = {2},
     PAGES = {445--507},
      ISSN = {1867-5778,1867-5786},
   MRCLASS = {14G22 (11F77 14F30)},
  MRNUMBER = {4130689},
MRREVIEWER = {Oliver\ Gregory},
}

@incollection {ogus_compositio,
    AUTHOR = {Ogus, Arthur},
     TITLE = {{$F$}-crystals on schemes with constant log structure},
      NOTE = {Special issue in honour of Frans Oort},
   JOURNAL = {Compositio Math.},
  FJOURNAL = {Compositio Mathematica},
    VOLUME = {97},
      YEAR = {1995},
    NUMBER = {1-2},
     PAGES = {187--225},
      ISSN = {0010-437X,1570-5846},
   MRCLASS = {14F30 (14F20 14F40 14G20)},
  MRNUMBER = {1355125},
MRREVIEWER = {Adolfo\ Quir\'os},
       URL = {http://www.numdam.org.pros2.lib.unimi.it/item?id=CM_1995__97_1-2_187_0},
}

@book {bo,
    AUTHOR = {Berthelot, Pierre and Ogus, Arthur},
     TITLE = {Notes on crystalline cohomology},
 PUBLISHER = {Princeton University Press, Princeton, NJ; University of Tokyo
              Press, Tokyo},
      YEAR = {1978},
     PAGES = {vi+243},
      ISBN = {0-691-08218-9},
   MRCLASS = {14F30},
  MRNUMBER = {491705},
MRREVIEWER = {G.\ Horrocks},
}

@book {EnriquesBook,
    AUTHOR = {Cossec, Francois and Dolgachev, Igor and Liedtke,
              Christian},
     TITLE = {Enriques surfaces. {I}},
   EDITION = {Second},
      NOTE = {With an appendix by S. Kondo},
 PUBLISHER = {Springer, Singapore},
      YEAR = {2025},
     PAGES = {xxi+681},
      ISBN = {978-981-96-1213-0; 978-981-96-1214-7},
   MRCLASS = {14J28},
  MRNUMBER = {4911529},
       DOI = {10.1007/978-981-96-1214-7},
       URL = {https://doi-org.pros2.lib.unimi.it/10.1007/978-981-96-1214-7},
}

@article {FKato_Def,
    AUTHOR = {Kato, Fumiharu},
     TITLE = {Log smooth deformation theory},
   JOURNAL = {Tohoku Math. J. (2)},
  FJOURNAL = {The Tohoku Mathematical Journal. Second Series},
    VOLUME = {48},
      YEAR = {1996},
    NUMBER = {3},
     PAGES = {317--354},
      ISSN = {0040-8735,2186-585X},
   MRCLASS = {14D15 (14M25)},
  MRNUMBER = {1404507},
       DOI = {10.2748/tmj/1178225336},
       URL = {https://doi-org.pros2.lib.unimi.it/10.2748/tmj/1178225336},
}

@article {KawNam,
    AUTHOR = {Kawamata, Yujiro and Namikawa, Yoshinori},
     TITLE = {Logarithmic deformations of normal crossing varieties and
              smoothing of degenerate {C}alabi-{Y}au varieties},
   JOURNAL = {Invent. Math.},
  FJOURNAL = {Inventiones Mathematicae},
    VOLUME = {118},
      YEAR = {1994},
    NUMBER = {3},
     PAGES = {395--409},
      ISSN = {0020-9910,1432-1297},
   MRCLASS = {32G05 (14D15 14J32 14J40 32S30)},
  MRNUMBER = {1296351},
MRREVIEWER = {Claire\ Voisin},
       DOI = {10.1007/BF01231538},
       URL = {https://doi-org.pros2.lib.unimi.it/10.1007/BF01231538},
}

@article {BindaKrishna,
    AUTHOR = {Binda, Federico and Krishna, Amalendu},
     TITLE = {Zero cycles with modulus and zero cycles on singular
              varieties},
   JOURNAL = {Compos. Math.},
  FJOURNAL = {Compositio Mathematica},
    VOLUME = {154},
      YEAR = {2018},
    NUMBER = {1},
     PAGES = {120--187},
      ISSN = {0010-437X,1570-5846},
   MRCLASS = {14C25 (14F30 19E15)},
  MRNUMBER = {3719246},
MRREVIEWER = {Fumio\ Hazama},
       DOI = {10.1112/S0010437X17007503},
       URL = {https://doi-org.pros2.lib.unimi.it/10.1112/S0010437X17007503},
}

@unpublished{BGV,
    author = {Binda, Federico and Gallauer, Martin and Vezzani, Alberto},
    title = {Motivic monodromy and $p$-adic cohomology theories},
    year = {2023},
 url = {https://arxiv.org/abs/2306.05099v1},
    note = {To appear in JEMS. ArXiv preprint: \url{https://arxiv.org/abs/2306.05099v1}}
}

@unpublished{ertl-yamada,
   title={Rigid analytic reconstruction of {H}yodo--{K}ato theory}, 
      author={Veronika Ertl and Kazuki Yamada},
      year={2025},
    note = {ArXiv preprint: \url{https://arxiv.org/abs/1907.10964v6}}
}

@book{ogu,
	author = {Ogus, Arthur},
	doi = {10.1017/9781316941614},
	fseries = {Cambridge Studies in Advanced Mathematics},
	isbn = {978-1-107-18773-3},
	mrclass = {14D06 (14A20 14M25)},
	mrreviewer = {Howard M. Thompson},
	pages = {xviii+539},
	publisher = {Cambridge Univ. Press},
	series = {Cambridge Stud. Adv. Math.},
	title = {Lectures on logarithmic algebraic geometry},
	url = {https://doi.org/10.1017/9781316941614},
	volume = {178},
	year = {2018}}

@incollection{katolog,
    author = {Kato, Kazuya},
    title = {Logarithmic Structures of Fontaine--Illusie},
    booktitle = {Algebraic analysis, geometry, and number
theory (Baltimore, MD, 1988)},
    publisher = {Johns Hopkins Univ. Press, Baltimore, MD},
    year = {1988},
    pages =  {191--224}
}

@article{tsuji,
    author = {Tsuji, Takeshi},
    title = {Saturated morphisms of logarithmic schemes},
    fjournal = {Tunisian Journal of Mathematics},
    journal = {Tunis. J. Math.},
    volume = {1},
    number = {2},
    year = {2019},
    pages = {185--220}
}

@article {tsuji_Faltings,
    AUTHOR = {Tsuji, Takeshi},
     TITLE = {{$p$}-adic \'etale cohomology and crystalline cohomology in
              the semi-stable reduction case},
   JOURNAL = {Invent. Math.},
  FJOURNAL = {Inventiones Mathematicae},
    VOLUME = {137},
      YEAR = {1999},
    NUMBER = {2},
     PAGES = {233--411},
      ISSN = {0020-9910,1432-1297},
   MRCLASS = {14F30 (14F20)},
  MRNUMBER = {1705837},
MRREVIEWER = {Abdellah\ Mokrane},
       DOI = {10.1007/s002220050330},
       URL = {https://doi-org.pros2.lib.unimi.it/10.1007/s002220050330},
}

@article {deg_niziol,
    AUTHOR = {D\'eglise, Fr\'ed\'eric and Nizio\l, Wies\l awa},
     TITLE = {On {$p$}-adic absolute {H}odge cohomology and syntomic
              coefficients. {I}},
   JOURNAL = {Comment. Math. Helv.},
  FJOURNAL = {Commentarii Mathematici Helvetici. A Journal of the Swiss
              Mathematical Society},
    VOLUME = {93},
      YEAR = {2018},
    NUMBER = {1},
     PAGES = {71--131},
      ISSN = {0010-2571,1420-8946},
   MRCLASS = {14F30 (11G25 14C30 14F42 14G20)},
  MRNUMBER = {3777126},
MRREVIEWER = {Adolfo\ Quir\'os},
       DOI = {10.4171/CMH/430},
       URL = {https://doi.org/10.4171/CMH/430},
}

@book{SAG,
	author = {Lurie, Jacob},
	publisher = {https://www.math.ias.edu/~lurie/papers/SAG-rootfile.pdf},
	title = {Spectral Algebraic Geometry},
	year = 2018}

@book{nakkajima,
    title = {Weight Filtration and Slope Filtration on the Rigid Cohomology of a Variety in Characteristic P>0},
    author = {Yukiyoshi Nakkajima},
    fseries = {M\'emoire de la Soci\'et\'e math\'ematique de France},
	publisher = {Soci{\'e}t{\'e} math{\'e}matique de France},
	series = {M{\'e}m. Soc. Math. Fr.},
	volume = {130-131},
	year = {2012}
}

@article{Du-Moon-Shimizu,
    title = {A log Prismatic-crystalline comparison theorem},
    author = {Du, Heng and Moon, Yong Su and Shimizu, Koji},
    journal = { },
    volume = {},
    issue = {},
    year = {},
   }

\end{document}